\documentclass[12pt]{article}
\usepackage[margin=1in]{geometry}
\usepackage{amsmath}
\usepackage{amssymb}
\usepackage{amsthm}

\usepackage{enumerate}
\usepackage{comment}
\usepackage[toc, page]{appendix}
\usepackage{subcaption}
\usepackage{tikz-cd}
\usepackage{algorithm}
\usepackage{url}
\usepackage{xcolor}
\usetikzlibrary{shadows}
\usepackage{algpseudocode}
\makeatother
\usepackage{graphicx}
\graphicspath{ {./Figures/} }
\usepackage{thm-restate}
\usetikzlibrary{calc, decorations.markings}

\tikzset{
  tick/.style={
    postaction={decorate, decoration={markings,
      mark=at position 0.5 with {\draw[thick] (0,-3pt)--(0,3pt);}}}
  },
  twoticks/.style={
    postaction={decorate, decoration={markings,
      mark=at position 0.5 with {
        \draw[thick] (-2pt,-3pt)--(-2pt,3pt);
        \draw[thick] ( 2pt,-3pt)--( 2pt,3pt);
      }}}
  },
  threeticks/.style={
    postaction={decorate, decoration={markings,
      mark=at position 0.5 with {
        \draw[thick] (-3pt,-3pt)--(-3pt,3pt);
        \draw[thick] ( 0pt,-3pt)--( 0pt,3pt);
        \draw[thick] ( 3pt,-3pt)--( 3pt,3pt);
      }}}
  }
}

\usepackage{hyperref}
\hypersetup{
    colorlinks=true,
    linkcolor=blue,
    filecolor=magenta,      
    urlcolor=cyan,
    citecolor=blue,
    }

\usepackage{framed}
\usepackage[normalem]{ulem}

\newcommand{\R}{\mathbb{R}}

\newcommand\Vr{\mathcal{V}}

\newcommand\ph[1]{\text{PH}_{#1}}
\newcommand\sort{\mathrm{sort}}
\newcommand\lune{\mathrm{lune}}

\newcommand
\lens{\mathrm{lens}}
\newcommand\rng{\mathrm{RNG}}

\newcommand\diam{\mathrm{diam}}
\newcommand\spx[1]{\mathrm{sp}_{#1}}
\newcommand\mst{\mathrm{MST}}

\newcommand\rvr{\mathcal{R}}

\newcommand\Angle{\mathrm{Angle}}

\theoremstyle{plain}
\newtheorem*{theorem*}{Theorem}
\newtheorem{theorem}{Theorem}[section]

\newtheorem{lemma}[theorem]{Lemma}

\newtheorem{corollary}[theorem]{Corollary}
\newtheorem{prop}[theorem]{Proposition}

\theoremstyle{definition}
\newtheorem{definition}[theorem]{Definition}

\newtheorem{example}[theorem]{Example}
\newtheorem{remark}[theorem]{Remark}

\title{Computation of degree-1 persistent homology on larger point-clouds using the Reduced Vietoris-Rips filtration}
\author{Musashi Koyama, Facundo M\'emoli, Vanessa Robins, Katharine Turner}
\date{\today}

\begin{document}

\maketitle

\begin{abstract}
Computing Vietoris-Rips persistent homology on a point cloud is computationally expensive, making it unsuitable for analysis of large datasets. In this paper, we present an alternative method for computing Vietoris-Rips persistent homology on point clouds in Euclidean space. We introduce the Reduced Vietoris-Rips complex, which contains an order of magnitude fewer simplices but has the same degree-$1$ persistent homology as the Vietoris-Rips filtration. By using the Reduced Vietoris-Rips complex and leveraging the geometry of Euclidean space, we are able to compute the Vietoris-Rips persistent homology for significantly larger point clouds than current implementations. 
\end{abstract}

\subsubsection*{Keywords:} Persistent homology, Vietoris-Rips filtration, Relative neighborhood graph. 

\subsubsection*{MSC:} 
Primary: 55N31 Persistent homology and applications, topological data analysis \\
Secondary: 68T09 Computational aspects of data analysis and big data

\subsubsection*{Acknowledgements}

This paper is based on content in the PhD thesis of M.K which was supervised by F.M, V.R and K.T at the Australian National University. M.K is now an ELBE postdoctoral researcher jointly supported by the Centre for Systems Biology Dresden, the Max Planck Institute of Cell Biology and Genetics and the Max Planck institute for the Physics of Complex Systems.    F.M. was supported
by the NSF through grants  IIS-1901360, CCF-1740761, CCF-
2310412, DMS-2301359 and by the BSF under grant 2020124.
K.T. was suppported by the ARC through DECRA Fellowship DE200100056.
 
\clearpage
\tableofcontents

\section{Introduction}
\label{Introduction}

Persistent homology is the study of how topological features evolve in a filtration of topological spaces. One of the most fundamental structures in persistent homology is the persistence barcode which summarises the parameter lifespan of topological features in the filtration.  The beginnings of persistent homology can be traced back to \cite{frosini_1990} where size theory, equivalent to modern day zero-dimensional persistent homology was introduced by Frosini. Modern persistent homology was developed fairly simultaneously by Robins \cite{robins1999towards}, and Edelsbrunner, Letscher and Zomorodian \cite{edelsbrunner_original}. Since then, persistent homology has been employed in several disciplines from health and neuroscience \cite{BIO21010LungCancerSurvivalisAssociatedWithPersistentHomologyofTumorImaging,stolz_2022_multiscale,moon2023using,curto-sanderson-2025} to materials science \cite{hiraoka_hierarchical_2016,lee_quantifying_2017,nearly_hexagonal_lattice}. 

The context for this paper is the problem of quantifying the shape approximated by a finite set of points in a metric space, i.e., a point-cloud, $X$.  
The application of persistent homology to such data first requires the construction of a filtration of topological spaces that capture the shape of $X$ at scale $r$. The standard choices are the \v{C}ech filtration built from the union of balls of radius $r/2$ and the Vietoris-Rips filtration consisting of subsets of diameter $\leq r$.  
To each of these spaces, we can associate a combinatorial simplicial complex: a $q$-simplex $\{x_0,\ldots,x_q\}$ spans $q+1$ points of $X$ in the \v{C}ech complex when the relevant $q+1$ balls have non-empty intersection. 
The points form a $q$-simplex of the Vietoris-Rips complex when $\diam\{x_0,\ldots,x_q\} \leq r$, 
see the text~\cite{Dey_Wang_2022} for an overview and further details.  

For point-clouds in low-dimensional Euclidean spaces there are efficient algorithms to construct the \v{C}ech filtration, namely the alpha-shape filtration \cite{edelsbrunner1995union,akkiraju1995alpha} or an approximation to this called the Delaunay-Rips filtration \cite{clemot2025delaunayripsfiltrationstudyalgorithm}.
The alpha-shape filtration is efficient because it excludes many unnecessary simplices from the full \v{C}ech construction. 
It does this by partitioning the union of balls by the Voronoi cells of $X$, guaranteeing that simplices in the alpha-complex are faces from the Delaunay triangulation. 
Note that Delaunay triangulation algorithms build top-dimensional simplices and for a point-cloud with $n$ points in $\R^D$, there are $O(n^{\lceil \frac{D}{2} \rceil})$ simplices~\cite{toth2017handbook}.

In contrast, the Vietoris-Rips construction always begins with 1-simplices -- edges between pairs of points ordered by length, and adds $q$-simplices when all pairwise distances between the $q+1$ points are less than the chosen scale. 
In the full Vietoris-Rips filtration of $n$ points there are $\binom{n}{q+1}$, thus $O(n^{q+1})$, $q$-simplices.  
The simplest way to exclude unnecessary simplices in this context is to restrict the degree of homology, because calculating degree-$q$ homology only requires knowledge of the $(q+1)$, $q$, and $(q-1)$-simplices. 
Even for $q=1$ however, this means all $O(n^3)$ $2$-simplices are constructed and ordered by diameter. 
Currently the most popular software for computing persistent homology from Vietoris-Rips filtration is $\mathsf{Ripser}$ \cite{Ripser} and its GPU-accelerated version $\mathsf{Ripser}++$ \cite{zhang_et_al:LIPIcs.SoCG.2020.70}.

In this paper we define a new filtration for persistent homology computation in degree-$q$ called the \emph{Reduced Vietoris-Rips filtration} of $X$, denoted $\rvr^{q}_\bullet(X)$, containing an order-of-magnitude fewer $(q+1)$-simplices when compared with the standard Vietoris-Rips filtration $\Vr_\bullet(X)$. 
The reduced filtration can be constructed for any finite metric space. 
Our main theoretical result proves that the Reduced Vietoris-Rips filtration has persistent homology isomorphic to that of the standard Vietoris-Rips filtration. 
 
\begin{theorem*}
Consider a finite metric space $X$.   
Then there exists a family of isomorphisms $\theta_{\bullet}$ such that the following diagram commutes

\begin{equation}
 \begin{tikzcd}[ampersand replacement=\&]
  H_{q}(\rvr^{q}_{r_1}(X)) \arrow[r, "f_{r_1}^{r_2}"] \arrow[d, "\theta_{r_1}"'] \& H_{q}(\rvr^{q}_{r_2}(X)) \arrow[d, "\theta_{r_2}"] \\
  H_{q}(\Vr_{r_1}(X)) \arrow[r, "g_{r_1}^{r_2}"] \& H_{q}(\Vr_{r_2}(X))
  \end{tikzcd} 
\end{equation}

\noindent for all $r_1$ and $r_2$ such that $0 \leq  r_1 < r_2$.  Above, $f_{r_1}^{r_2}$ and $g_{r_1}^{r_2}$ are the maps at homology level induced by the natural inclusions. 
\end{theorem*}

The main text of this paper develops the definitions and theorem proof for the case of homology degree $q=1$ in Sections~\ref{RVR-definition} and~\ref{RVR-theorem}. The general definition of $\rvr^{q}_\bullet(X)$ and proof of the above theorem is provided in the appendix. 
In devising an efficient implementation of the reduced Vietoris-Rips filtration we restrict to point-clouds in Euclidean space as this enables the use of geometric data structures (described in Section~\ref{GDS}) that speed up location searching, specifically the $kd$-tree.  
The algorithm and overall complexity analysis is given in Section~\ref{Algorithm}. 
The experiments presented in Section~\ref{sec:experiments} compute persistent homology in degree-1 for point-clouds in $\R^{10}$ where the use of alpha-shapes would require the construction of $O(n^5)$ 10-dimensional simplices and then their $2$-dimensional faces.  The reduced Vietoris-Rips approach instead directly constructs $O(n^2)$ $2$-simplices. 
The performance of our implementation is compared with that of $\mathsf{Ripser}$ and discussed in Section~\ref{Discussion:Ripser}.

Previous work related to reducing the number of simplices used for computing persistent homology includes edge collapses, \cite{glisse2022swap}, \cite{boissonnat2020edge}, featured in the computational topology software suite GUDHI \cite{10.1007/978-3-662-44199-2_28}. We compare our software to that of GUDHI's edge collapses in Section \ref{sec:comparing-euclideanph1-with-gudhis-edge-collapse} and~\ref{Discussion:EdgeCollapse}. Other work has reduced the number of simplices by computing an approximation of $\ph{1}(X)$ \cite{sheehy2012linear} or building a filtration using a subset of points from the point-cloud \cite{de2004witness}, \cite{graf2026flood}. 

\section{Preliminaries}
\label{Preliminaries}

In this section we briefly review notation and  preliminary concepts necessary for content later in the paper.

\subsection{Notation and terminology}
\label{Notation}

The primary object of study is a finite set of points $X$ (a point cloud) embedded as a subset of $D$-dimensional Euclidean space $\mathbb{R}^{D}$ with a bijection $\psi: X \rightarrow \{1,....|X|\}$ which assigns each point a natural number from $1$ to $n = |X|$.
The distance between two points in $x_1, x_2 \in \R^D$ will be written $d(x_1,x_2)$. In this paper, $n$ will exclusively be used to denote the size of a point cloud, that is $n = |X|$. 
A $q$-simplex consisting of the vertices $x_0,\ldots,x_q$ will be denoted as $\langle x_0\ldots x_q \rangle$.

\subsection{Homology}

Here we will give a terse introduction to simplicial homology, its main purpose is to establish notation. We present only what is absolutely necessary and do not state the results in their utmost generality. The reader who desires more details is directed towards Munkres' textbook \emph{Elements of Algebraic Topology}  \hspace{0.5pt}\cite{munkres2018elements}.

First we will define the notion of a simplicial complex.

\begin{definition}[Simplicial complex]
    A simplicial complex with vertex set $V = \{v_{1},...,v_{n}\}$ is a set $K \subset 2^{V}$ which satisfies the following properties. 

    \begin{itemize}
        \item $\emptyset \in K$
        
        \item $\{v_{i}\} \in K$ for all $i\in \{1,...,n\}$. 

        \item If $\sigma \in K$, then all subsets of $\sigma$ are also in $K$. 

    \end{itemize}
\end{definition}

Next we define the notion of a simplex. 

\begin{definition}[$q$-simplex]
    Consider a simplicial complex $K$ with vertex set $V = \{v_1,...,v_n\}$. Then let $\sigma$ be a subset of $V$ with $q+1$ elements. Then we refer to $\sigma$ as a $q$-simplex.
\end{definition}

We will make reference to the dimension of a simplex in Definition \ref{binary-relation-on-simplices}. 

\begin{definition}[Dimension of a simplex]
    Consider a simplicial complex $K$. Let $\sigma$ be a $q$-simplex, then we say that $\sigma$ has dimension $q$ and denote this by $\dim (\sigma) = q$
\end{definition}

It will be convenient later on to have a special term for when one simplex is a subset of another simplex. 

\begin{definition}[Faces and cofaces]
    Consider a simplicial complex $K$ with simplices $\sigma, \tau$ with $\sigma \subset \tau$. Then we say that $\sigma$ is a face of $\tau$ and $\tau$ is a coface of $\sigma$. 
\end{definition}

From here on in we write $\sigma = \{w_0,...,w_q\}$ as $\langle w_0...w_{q}\rangle$, following  standard notation for oriented simplices. Since we are working with $\mathbb{Z}_{2}$-coefficients we can effectively ignore the orientation of the simplices. This means that we can refer to a given simplex using any permutation of its vertices. For example, the $2$-simplex $\langle xyz \rangle$ can equally be referred to as $\langle xzy \rangle = \langle zxy \rangle = \langle zyx \rangle = \langle yzx \rangle = \langle yxz \rangle$.  

The addition of simplices is formalised in the next definition.

\begin{definition}[Simplicial $q$-chain]
    A simplicial $q$-chain is a finite formal sum of $q$ simplices,  
    \begin{equation}
        \sum_{ i = 1}^{N}c_i\sigma_i
    \end{equation}
    In this paper, the coefficients $c_i$ are taken from $\mathbb{Z}_2$. 
\end{definition}

We now define three important vector spaces. 

\begin{definition}[Chain group]
    $C_{q}(K)$ is the free abelian group with coefficients in $\mathbb{Z}_2$ with generating set consisting of all $q$-simplices. It is customary to set $C_{-1}(K) = 0$. 
\end{definition}

\begin{remark}
    It is worth noting that $C_{q}(K)$ is actually a vector space since $\mathbb{Z}_{2}$ is a field and that from this point on, any time the word ``group" is mentioned it could be replaced with ``vector space". We continue to use the word group to adhere to the ``traditional" presentation of persistent homology, though the reader unfamiliar with groups can replace them with vector spaces for the purposes of this paper. 
\end{remark}

\begin{definition}[Boundary map]
    Let $K$ be a simplicial complex. Consider the map $\partial_{q+1}^{K}: C_{q+1}(K)\rightarrow C_{q}(K)$ defined as follows. Let $\sigma = \langle v_{0},...,v_{q}\rangle$. Then we define $\partial_{q+1}^{K} (\sigma)$ as follows:

    \begin{equation}
        \partial_{q+1}^{K} (\sigma) = \sum_{i=0}^{q}\langle v_{0}...\hat{v_i}...v_q \rangle 
    \end{equation}
    We extend this linearly to a map on $C_{q+1}(K)$.
    Here $\hat{v_i}$ means that $v_i$ is to be omitted from $\langle v_0...v_i...v_q\rangle$. When it is clear what $K$ is, we may write $\partial_{q+1}^K$ as $\partial_{q+1}$. When $q$ is also clear, we may simply write $\partial_{q+1}$ as $\partial$.
    The map $\partial_{0}^K: C_{0}(K) \rightarrow C_{-1}(K) := 0$ is simply the zero-map. 
\end{definition}

\begin{definition}[$q$-cycles]
    Consider a simplicial complex $K$. We denote $\mathrm{ker}(\partial_q^K) = Z_{q}(K)$. We call a $q$-chain $c \in Z_{q}(K)$ a $q$-cycle. $Z_{q}(K)$ will be referred to as the group of $q$-cycles. 
\end{definition}

\begin{definition}[$q$-boundaries]
    Consider a simplicial complex $K$. We denote $\mathrm{im}(\partial_{q+1}) = B_{q}(K)$. We call a $q$-chain $c \in B_{q}(K)$ a $q$-boundary. $B_{q}(K)$ will be referred to as the group of $q$-boundaries. 
\end{definition}

We are now ready to define the homology groups of a simplicial complex $K$, but before we do, we state an extremely easy to verify lemma. 

\begin{lemma}
\label{boundary-of-boundary}
    Consider a simplicial complex $K$. Then we have $B_{q}(K) \subset Z_{q}(K)$. 
\end{lemma}

\begin{definition}[Homology groups of a simplicial complex]
Consider a simplicial complex $K$. Then the $q$th homology group is defined as $H_{q}(K) = Z_{q}(K)/B_{q}(K)$ 
\end{definition}

Note that Lemma $\ref{boundary-of-boundary}$ is necessary to show that $B_{q}(K)$ is indeed a subgroup of $Z_{q}(K)$ and hence the quotient group can be taken. It is here that we note that elements of $H_{q}(K)$ will be written as $[\gamma]$ to denote the fact that $\gamma$ is a representative of the class $[\gamma] \in H_{q}(K)$. Sometimes we will also use coset notation and write $[\gamma]$ as $\gamma + H_{q}(K)$. 

\subsection{Persistent Homology}

The following is a brief summary of some basic definitions in persistent homology. The reader who desires more context and details is directed towards chapter 7 of \emph{Computational Topology} \cite{book}, where the definitions below come from. 

\begin{definition}[Filtration of simplicial complexes]
Given an index set $\mathcal{I}$ and a set of simplicial complexes $(K_{i})_{i\in \mathcal{I}}$, if for $i \leq j$ in $\mathcal{I}$ we have $K_{i} \subset K_{j}$, we call the collection $(K_{i})_{i\in \mathcal{I}}$ a filtration of simplicial complexes.
\end{definition}

In the algorithm for computing persistent homology, we require 
a particular type of filtration.

\begin{definition}[Simplex-wise filtrations]
\label{def-simplex-wise-filtrations}
$(K_{i})_{i\in \mathcal{I}}$ is a simplex-wise filtration when $\mathcal{I} = \{0,...,m\}$ and $K_{i} = K_{i-1}\cup \sigma_{i}$  for $i \leq m$ where $\sigma_{i}$ is a single simplex and it is understood that $K_{0} = \emptyset$.
\end{definition}
As per this definition, $m$ will always denote the number of simplices of all possible dimensions in the filtration. 

Consider a simplex-wise filtration. For $i<j$ we apply the degree $q$ homology functor $H_{q}(-)$ to the inclusion $K_{i} \subset K_{j}$ to obtain a linear homomorphism $f_{i}^{j}: H_{q}(K_{i})\rightarrow H_{q}(K_j)$. Persistent homology quantifies how the homology changes across the parameter range. 

\begin{definition}[Birth of a cycle]
\label{definition-of-birth}
Consider a simplex-wise filtration $(K_{i})_{i\in \{0,...,m\}}$. A homology class represented by $\gamma + B_{q}(K_{j}) \in H_{q}(K_{j})$ is said to be born at index $j$ if $j$ is the smallest index such that for all $j' < j$ there is no $\beta + B_{q}(K_{j'}) \in H_{q}(K_{j'})$ with $f_{j'}^{j}(\beta + B_{q}(K_{j'})) = \gamma + B_{q}(K_{j})$. The simplex $\sigma_j$ is said to create the homology class represented by $\gamma + B_{q}(K_{j})$. 
\end{definition}

\begin{definition}[Death of a cycle]
\label{definition-of-death}
Consider a simplex-wise filtration $(K_{i})_{i\in \{0,...,m\}}$. A homology class represented by $\gamma + B_{q}(K_{j})$ born at $j$ is said to die at index $l$ if $l$ is the smallest index such that $f_{j}^{l}(\gamma + B_{q}(K_{j})) = f_{j'}^{l}(\beta + B_{q}(K_{j'}))$ for some $j' < j$ and some $\beta + B_{q}(K_{j'}) \in H_{q}(K_{j'})$. The simplex $\sigma_{l}$, such that $K_{l} = K_{l-1}\cup \sigma_{l}$, is said to kill the homology class represented by $\gamma + B_{q}(K_{j})$. 
\end{definition}

For a simplex-wise filtration $(K_{i})_{i\in \{0,...,m\}}$ each simplex $\sigma_{i}$ can only create or kill a homology class. To demonstrate this we prove some lemmas which are well known facts, but the author could not find any references which explicitly proved these lemmas only using
Definitions \ref{definition-of-birth} and \ref{definition-of-death} and thus we provide their proof here.

\begin{lemma}
\label{can-only-birth}
    Let $(K_{i})_{i\in \{0,...,m\}}$ be a simplex-wise filtration and consider the change in homology between $K_{i-1}$ and $K_{i} = K_{i-1} \cup \sigma_{i}$, where $\sigma_{i}$ is a $q$-simplex. Then $\partial \sigma_{i} \in B_{q-1}(K_{i-1})$ if and only if $\sigma_{i}$ creates a degree-$q$ homology class. 
\end{lemma}

\begin{proof}
    Suppose $\partial \sigma_{i} \in B_{q-1}(K_{i-1})$. Then we must have that $\partial \sigma_{i} = \sum_{j\in A} \partial \sigma_{j}$ where $A \subset \{0,...,i-1\}$. Then we have that $\partial \sigma_{i} - \sum_{j\in A}\partial \sigma_{j} = 0$ and thus $\partial (\sigma_{i} - \sum_{j \in A} \sigma_{j}) = 0$. Then we have that $\sigma_{i} - \sum_{j\in A}\sigma_{j} + B_{q}(K_{i})$ is a homology class which is born at $i$. To see this, note that $\sigma_{i} - \sum_{j\in A} \sigma_{j} + B_{q}(K_{i})$ cannot possibly be  expressed in the form $f_{l}^{i}(\gamma + B_{q}(K_{l})) = \sigma_{i} - \sum_{j\in A}\sigma_{j} + B_{q}(K_{i}), l < i$  since this would require $\gamma + \sum_{j \in A}\sigma_{j} - \sigma_{i} \in B_{q}(K_{i})$ and this cannot occur since all elements of $\gamma + \sum_{j \in A}\sigma_{j}$ must be in $K_{i-1}$. Now we prove the converse, that is suppose that the addition of $\sigma_{i}$ entering the filtration gives birth to a degree-$q$ homology class $\gamma_1 + B_{q}(K_{i}) \in H_{q}(K_{i})$. We will show that $\gamma_{1}$ must be of the form $\sum_{k \in B}\sigma_{k} + \sigma_{i}$, where $B \subset \{0,...,i-1\}$. If this was not the case, i.e $\gamma_{1}$ was of the form $\sum_{k\in B}\sigma_{k}$ then we would have $\gamma_{1} + B_{q}(K_{\max(B)}) \in H_{q}(K_{\max (B)})$ meaning $f_{\max (B)}^{i}(\gamma_{1} + B_{q}(K_{\max(B)})) = \gamma_{1} + B_{q}(K_{i})$ contradicting the fact that $\gamma_1 + B_{q}(K_{i})$ was born upon the addition of $\sigma_{i}$. Since we know that $\partial (\gamma_{1}) = 0$ we have that $\partial (\sigma_{i} + \sum_{k \in B} \sigma_{k}) = 0$ which means that $ \partial \sigma_{i} = \sum_{k\in B}\partial \sigma_{k}$. Hence we have $\partial \sigma_{i} \in B_{q-1}(K_{i-1})$. 
\end{proof}

\begin{lemma}
\label{can-only-kill}
    Let $(K_{i})_{i\in \{0,...,m\}}$ be a simplex-wise filtration and consider the change in homology between $K_{i-1}$ and $K_{i} = K_{i-1} \cup \sigma_{i}$, where $\sigma_{i}$ is a $q$-simplex. Then $\partial \sigma_{i} \notin B_{q-1}(K_{i-1})$ if and only if $\sigma_{i}$ kills a degree-$(q-1)$ homology class. 
\end{lemma}

\begin{proof}
    First we will show that if the addition of $\sigma_{i}$ to the filtration kills a degree-$(q-1)$ homology class then we have $\partial \sigma_{i} \notin B_{q-1}(K_{i-1})$. To this end suppose that $\sigma_{i}$ kills a degree $(q-1)$ homology class $\gamma + B_{q-1}(K_{j})$ born at index $j$. Then that must mean we have some $\beta + B_{q-1}(K_{j'})$ with $j' < j$ such that $f_{j'}^{i}(\beta + B_{q-1}(K_{j'})) = f_{j}^{i}(\gamma + B_{q-1}(K_{j}))$. Then we have that $\beta + B_{q-1}(K_{i}) = \gamma + B_{q-1}(K_{i})$ and thus we have $\beta - \gamma \in B_{q-1}(K_{i})$. By Definition \ref{definition-of-death} we also know that $\beta - \gamma \notin B_{q-1}(K_{i-1})$, otherwise $f_{j'}^{i-1}(\beta + B_{q-1}(K_{j'})) = f_{j}^{i-1}(\gamma + B_{q-1}(K_{j}))$ which would violate the minimality of $i$. Now suppose $\partial \sigma_{i} \in B_{q-1}(K_{i-1})$, then we would have $B_{q-1}(K_{i-1}) = B_{q-1}(K_{i})$, implying $\beta -\gamma \in B_{q-1}(K_{i-1})$ a contradiction. Hence $\partial \sigma_{i} \notin B_{q-1}(K_{i-1})$. Now we prove the converse, suppose $\partial \sigma_{i} \notin B_{q-1}(K_{i-1})$. Since $\partial \sigma_{i} \notin B_{q-1}(K_{i-1})$ we have at index $i-1$ that $\partial \sigma_{i} + B_{q-1}(K_{i-1})$ is a non-trivial homology class. Consider \emph{all} homology classes $\gamma + B_{q-1}(K_{j})$ born at index $j<i$ such that $f_{j}^{i-1}(\gamma + B_{q-1}(K_{j})) = \partial \sigma_{i} + B_{q-1}(K_{i-1})$. Let $\gamma' + B_{q-1}(K_{j'})$ be the oldest of such homology classes, we now show that $\gamma' + B_{q-1}(K_{j'})$ dies at index $i$. We have $\gamma' + B_{q-1}(K_{j'})$ is certainly dead no later than index $i$ since $f_{0}^{i}( 0 + B_{q-1}(K_{0})) = \partial \sigma_{i} + B_{q-1}(K_{i}) = f_{j'}^{i}(\gamma' + B_{q-1}(K_{j'}))$. Now $\gamma' + B_{q-1}(K_{j'})$ cannot die at index earlier than $i$, for if it did then there would exist some $\beta + B_{q-1}(K_{l})$ born at index $l<j'$ such that $f_{l}^{i-1}(\beta + B_{q-1}(K_{l})) = f_{j'}^{i-1}(\gamma' + B_{q-1}(K_{j'})) = \partial \sigma_{i} + B_{q-1}(K_{i-1})$ which violates the fact that $\gamma' + B_{q-1}(K_{j'})$ was selected to be the oldest homology class such that $f_{j'}^{i-1}(\gamma' + B_{q-1}(K_{j'})) = \partial \sigma_{i} + B_{q-1}(K_{i-1})$.
\end{proof}

\begin{lemma}
\label{can-only-birth-or-kill}
Let $(K_{i})_{i\in \{0,...,m\}}$ be a simplex-wise filtration and consider the change in homology between $K_{i-1}$ and $K_{i} = K_{i-1} \cup \sigma_{i}$, where $\sigma_i$ is a $q$-simplex.  Then one, and only one, of the following must occur. 

\begin{itemize}
        \item $\sigma_i$ kills a homology class of degree $q-1$

        \item $\sigma_i$ gives birth to a homology class of degree $q$. 
    \end{itemize}
\end{lemma}

\begin{proof}
    Either $\partial \sigma_{i} \in B_{q-1}(K_{i-1})$ or $\partial \sigma_{i} \notin B_{q-1}(K_{i-1})$. Only one of these statements can be true and one of these statements must be true. By Lemma \ref{can-only-birth} the case  $\partial \sigma_{i} \in B_{q-1}(K_{i-1})$ corresponds to a birth of a degree-$q$ homology class and the case $\partial \sigma_{i} \notin B_{q-1}(K_{i-1})$ corresponds to a death of a degree-$(q-1)$ homology class by Lemma \ref{can-only-kill}. 
\end{proof}

\begin{definition}[Persistence pair]
Consider a simplex-wise filtration $(K_{i})_{i\in \{0,...,m\}}$. If $[\gamma] \in H_{q}(K_{i})$ was born at index $i$ and died at index $j$ then $(i,j)$ is said to be a degree-$q$ persistence pair. If there is no confusion as to what the value of $q$ is, sometimes we will refer to $(i,j)$ simply as a persistence pair. 
\end{definition}

\subsection{Vietoris-Rips complexes and filtrations}

In this section we briefly discuss one of the main structures of interest for this paper. Before doing so, we define a measure of size for a simplex.

\begin{definition}[Diameter of a simplex]
    Consider a finite point set $A \in \mathbb{R}^D$. Then the diameter $A$, denoted $\diam (A)$ is defined as $\max_{x,y\in A} d(x, y)$. For a simplex $\sigma = \langle x_{0}...x_{p}\rangle $, we have $\diam (\sigma) := \max_{i,j \in \{0,...,p\}}d(x_{i}, x_{j})$. 
\end{definition}

Vietoris-Rips complexes first appeared in \cite{vietoris1927hoheren} and were originally called Vietoris complexes. Eliyahu Rips applied the complex to the study of hyperbolic groups and Mikhail Gromov popularised the term Rips complex . The term Vietoris-Rips complex was coined by Hausmann in \cite{hausmann1994vietoris}.

\begin{definition}[Vietoris-Rips complex at scale $r$]
Let $X$ be a point cloud. For $r\in [0,\infty)$ we construct the Vietoris-Rips complex at scale $r$, $\Vr _{r} (X)$, as follows. If $\{x_0,...,x_p\} \subset X$ is such that $\diam(\{x_0,...,x_p\}) \leq r$ then $\langle x_0...x_p \rangle $ is a $p$-simplex in $\Vr_{r} (X)$. 
\end{definition}

In order to construct a filtration of simplicial complexes, we state an extremely easy to prove lemma without proof. 

\begin{lemma}
    \label{vr-is-actually-a-filtration}
    Let $X$ be a point cloud and let $0 \leq r_1 \leq r_2$. Then we have $\Vr_{r_1}(X) \subseteq \Vr_{r_2}(X)$. 
\end{lemma}

\begin{definition}[Vietoris-Rips filtration]
    The Vietoris-Rips filtration on $X$ is the nested collection of spaces   $\Vr _{\bullet} (X) := \{\Vr_{r_1}(X)\subseteq \Vr_{r_2}(X)\}_{0 \leq r_1 \leq r_2}$. 
\end{definition}

\begin{remark}
    $\Vr_{\infty}(X)$ consists of the $(|X|-1)$-simplex spanning  all points of $X$ and all lower-degree faces. It is the simplicial complex built from the power set (i.e., the set of all subsets, denoted $2^{X}$) of $X$. 
\end{remark}

The Vietoris-Rips filtration is not a simplex-wise filtration, but is easily modified to be so. We follow the method described in \cite{Ripser}.  

We first extend the function $\psi: X \rightarrow \{1,...,|X|\}$ which indexes the vertices, to a function that labels each simplex in $\Vr_{\infty}(X)$. 

\begin{definition}[Extension of $\psi$]
\label{extention-of-psi}
    Given $\psi: X \rightarrow \{1,...,|X|\}$ we extend its domain and range to $\psi: \Vr_{\infty}(X)\rightarrow 2^{\{1,...,|X|\}}$ in the following fashion. Consider a simplex $\sigma = \langle x_{0}...x_{p}\rangle$, then $\psi(\sigma)$ is the set of vertex labels $\{\psi(x_{0}),...,\psi(x_{p})\}$.
\end{definition}

Next, we define a function that sorts the integer labels in an element of $2^{\{1,...,|X|\}}$ so they are listed in increasing order.

\begin{definition}
\label{sort-definition}
Let $A(n)$ be the set of \emph{ordered} subsets of $\{1,...,n\}$. 
That is, $S = \{s_1, \ldots, s_k\} \subset \{1,...,n\} $ is in $A(n)$ if and only if $s_1 < s_2 < \cdots < s_k$. 
We write $\sort(S)$ for the function that maps a set of integers to its ordered version. 
To shorten notation we will also write $\sort(\sigma)$ when we really mean $\sort(\psi(\sigma))$. 

\end{definition}

We now ``stretch out'' the Vietoris-Rips filtration and turn it into a simplex-wise filtration by using a length-lexicographic ordering on simplices with the same diameter. 

\begin{definition}
    \label{binary-relation-on-simplices}
    We define a binary relation $ <$ on $\Vr_{\infty}(X)$ as follows. 

\begin{itemize}
    \item $\sigma <\tau$ if $\diam (\sigma) < \diam(\tau)$. 

    \item If $\diam(\sigma) = \diam(\tau)$ then $\sigma < \tau$ if $\dim (\sigma) < \dim (\tau)$. 

    \item If $\diam(\sigma) = \diam(\tau)$ and $\dim(\sigma) = \dim(\tau)$ then $\sigma < \tau$ if $\sort (\sigma) <_{lex} \sort(\tau)$ according to lexicographical order, $<_{lex}$. 
\end{itemize}
\end{definition}

Recall lexicographical ordering on elements of $A(n)$ with the same cardinality is defined as follows. 
Given $S, T \in A(n)$ we have $S = \{s_1 < s_2 \cdots < s_k\}$ and $T = \{t_1 < t_2 \cdots < t_k\}$. Then $S <_{lex} T$ in lexicographic ordering if there is some $1 \leq j \leq k$ such that $s_i = t_i$ for $i < j$, and $s_j < t_j$.   

The following lemma can be readily verified as length-lexicographic ordering is known to be a total order for finite sequences. The fact that $<$ defines a total order will be used to define what will be called the ``simplex-wise Vietoris-Rips filtration". 

\begin{lemma}
    The binary relation $<$ given in Definition \ref{binary-relation-on-simplices} is a total order on $\Vr_{\infty}(X)$. 
\end{lemma}

\begin{definition}[Simplex-wise Vietoris-Rips filtration]
\label{def-VR-total-order}
Let $X$ be a point cloud. Suppose there are $m$ simplices in $\Vr _{\infty} (X)$. We use the total order $<$ on $\Vr_{\infty}(X)$ to construct a bijection $\phi:\Vr_{\infty}(X) \rightarrow \{1,...,m\}$ by mapping the lowest element according to $<$ to $1$, the next lowest element to $2$ and so on. The filtration $(K_{i})_{i \in \{0,...,m\}}$ where $K_{0} = \emptyset$ and $K_{i} = \cup_{j=1}^{i}\phi ^{-1}(j)$ for $i>0$ is referred to as the simplex-wise Vietoris-Rips filtration on $X$.
\end{definition}

\begin{remark}
    It is customary to write $\phi^{-1}(j)$ as $\sigma_j$. Thus we write $K_{i} = \cup_{j=1}^{i}\sigma_j$
\end{remark}

\begin{definition}[Birth and death values for  Vietoris-Rips filtrations]
Consider the collection $\Vr_{\bullet}(X)$ and its corresponding simplex-wise filtration. Let $(i,j)$ be a persistence pair for the simplex-wise filtration. Then $\mathrm{diam}(\sigma_i)$ is said to be the birth value of the homology class that is born when $\sigma_i$ is added and $\mathrm{diam}(\sigma_j)$ is said to be the death value of this homology class. 
\end{definition}

\begin{definition}[Persistence]
Consider $\Vr_{\bullet}(X)$ with its corresponding simplex-wise filtration. Let $(i,j)$ be a persistence pair for the simplex-wise filtration. Then the persistence of $(i,j)$ is defined as $\mathrm{diam}(\sigma_j) - \mathrm{diam}(\sigma_i)$.
\end{definition}

Note that the persistence may be zero. In this case we say that the persistence pair $(i,j)$ has trivial persistence.

\begin{definition}[Persistence barcode, $\ph{q}(X)$]
Consider a point cloud $X$. The multiset of all left-closed, right-open intervals $\left[\mathrm{diam}(\sigma_i), \mathrm{diam}(\sigma_j)\right)$ such that $(i,j)$ is a degree-$q$ persistence pair of the simplex-wise Vietoris-Rips filtration with non-trivial persistence is called the degree-$q$ Vietoris-Rips persistence barcode of $X$. We will denote it by $\ph{q}(X)$.

\end{definition}

The persistence barcode may be visualised by drawing each interval stacked above the real line as in Figure \ref{circle_of_circle_diag}. 
An alternative visualisation is the persistence diagram, where the points $(\mathrm{diam}(\sigma_i), \mathrm{diam}(\sigma_j))$ are plotted on cartesian axes for each persistence pair $(i,j)$.

\begin{figure}
\begin{subfigure}[t]{.5\textwidth}
    \centering
    \includegraphics[scale = 0.5, clip,trim=4cm 9cm 4cm 9cm]{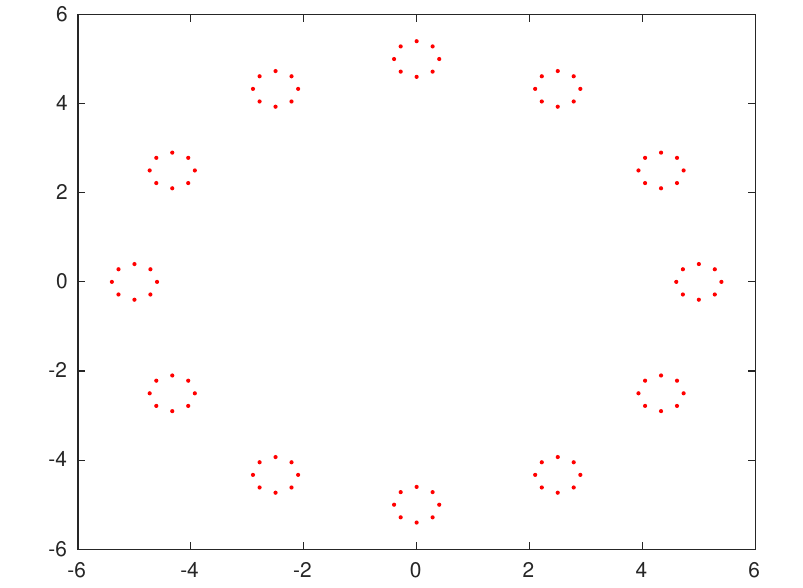}
    \subcaption{}
    
\end{subfigure}%
\begin{subfigure}[t]{.5\textwidth}
    \centering
\begin{tikzpicture}

\draw [|->] (0,0) -- (6.5,0);
\foreach \i in {0,2,4,6,8,10,12}
    \node [below] at (0.5*\i,0) {\i};

\foreach \i in {1,...,12}
    \draw [thick,red] (0.155,0.2*\i)--(0.4,0.2*\i);

\draw [thick, red] (0.9, 3.0) -- (4.6,3.0);
\end{tikzpicture}
 \subcaption{}
  
    \end{subfigure}
\caption{(a) A point cloud consisting of points on many small circles placed around a larger circle together with (b) its degree-1 persistence barcode. }
\label{circle_of_circle_diag}
\end{figure}

Throughout this paper we focus on the case $q=1$. We will sometimes say we are ``calculating the persistent homology of $X$'' and it should be understood that we mean we are calculating degree-$1$ persistent homology of the Vietoris-Rips filtration of $X$.

\subsection{The matrix reduction algorithm for computing persistent homology}
\label{matrix-reduction-algorithm}

In this section, we briefly review the standard matrix reduction algorithm for computing the persistent homology of a simplex-wise filtration. The reader wishing for more details and generality is directed towards \cite{book}. 

We start by defining the boundary matrix $\partial$. If there are $m$ simplices in the filtration then the boundary matrix will be an $m \times m$ matrix with the $i$th column encoding the boundary of $\sigma_i$.  If $\sigma_i$ is a $q$-simplex, then column-$i$ of $\partial$ has a 1 in each row corresponding to the index of a $(q-1)$-simplex that is a face of $\sigma_i$.  
The \emph{lowest 1} of a column is the entry with largest row-index. The filtration ordering ensures this row index is smaller than the column index.

\begin{figure}[h!]
    \centering
    \begin{tikzpicture}

        \draw[thin, gray] (0,4) -- (12,4); 
        \draw[thin, gray] (0,9) -- (12,9);
        \draw[thin, gray] (0,14) -- (12,14);
        \draw[thin, gray] (0,-0.5) -- (12,-0.5);
        \draw[thin, gray] (0,19) -- (12,19);
        \draw[thin, gray] (0,-0.5) -- (0,19);
        \draw[thin, gray] (12,-0.5) -- (12,19);
        \draw[thin, gray] (6,-0.5) -- (6,19);

        \draw[thick, fill=red] (7,1) -- (10,3) -- (11,1) -- (7,1);
        \draw[fill] (7,1) circle [radius=0.10];
        \node [below right] at (7,1) {$3$};
        \draw[fill] (11,1) circle [radius=0.10];
        \node [below right] at (11,1) {$2$};
        \draw[fill] (10,3) circle [radius=0.10];
        \node [above right] at (10,3) {$1$};

        \draw[thick] (1,1) -- (4,3) -- (5,1) -- (1,1);
        \draw[fill] (1,1) circle [radius=0.10];
        \node [below right] at (1,1) {$3$};
        \draw[fill] (5,1) circle [radius=0.10];
        \node [below right] at (5,1) {$2$};
        \draw[fill] (4,3) circle [radius=0.10];
        \node [above right] at (4,3) {$1$};

        \draw[thick] (4,8) -- (5,6);
        \draw[fill] (1,6) circle [radius=0.10];
        \node [below right] at (1,6) {$3$};
        \draw[fill] (5,6) circle [radius=0.10];
        \node [below right] at (5,6) {$2$};
        \draw[fill] (4,8) circle [radius=0.10];
        \node [above right] at (4,8) {$1$};

        \draw[thick] (7,6) -- (10,8) -- (11,6);
        \draw[fill] (7,6) circle [radius=0.10];
        \node [below right] at (7,6) {$3$};
        \draw[fill] (11,6) circle [radius=0.10];
        \node [below right] at (11,6) {$2$};
        \draw[fill] (10,8) circle [radius=0.10];
        \node [above right] at (10,8) {$1$};

        \draw[fill] (7,11) circle [radius=0.10];
        \node [below right] at (7,11) {$3$};
        \draw[fill] (11,11) circle [radius=0.10];
        \node [below right] at (11,11) {$2$};
        \draw[fill] (10,13) circle [radius=0.10];
        \node [below right] at (10,13) {$1$};

        \draw[fill] (5,11) circle [radius=0.10];
        \node [below right] at (5,11) {$2$};
        \draw[fill] (4,13) circle [radius=0.10];
        \node [below right] at (4,13) {$1$};

        \draw[fill] (10,18) circle [radius=0.10];
        \node [below right] at (10,18) {$1$};

        \node [right, blue] at (3,10) {$K_2$};
        \node [blue] at (3,15) {$K_{0} = \emptyset$};
        \node [right, blue] at (9,10) {$K_3$};
        \node [right, blue] at (9,15) {$K_1$};

        \node [right, blue] at (3,5) {$K_4$};
        \node [right, blue] at (9,5) {$K_5$};
        \node [right, blue] at (3,0) {$K_6$};
        \node [right, blue] at (9,0) {$K_7$};
        
    \end{tikzpicture}
    \caption{A simplex-wise filtration of a triangle. Each complex in the filtration differs from the previous one by a single simplex}\label{fig:simplex-wise}
\end{figure}

\newpage
\begin{example}
In this example we consider the simplex-wise filtration depicted in Figure \ref{fig:simplex-wise}. We have labelled the 0-simplices with the integers $1, 2, 3$. 
The 1-simplices are added in the order $\langle 12 \rangle, \langle 13 \rangle, \langle 23 \rangle$ and the 2-simplex $\langle 123 \rangle$ is added last. 

For this filtration the corresponding boundary matrix $\partial$ is given by 

\begin{equation}
\partial = 
\begin{bmatrix}
0 & 0 & 0 & 1 & 1 & 0 & 0\\
0 & 0 & 0 & 1 & 0 & 1 & 0\\
0 & 0 & 0 & 0 & 1 & 1 & 0\\
0 & 0 & 0 & 0 & 0 & 0 & 1\\
0 & 0 & 0 & 0 & 0 & 0 & 1\\
0 & 0 & 0 & 0 & 0 & 0 & 1\\
0 & 0 & 0 & 0 & 0 & 0 & 0\\
\end{bmatrix}
\end{equation}

Note that in column 7, corresponding to the addition of the 2-simplex, we have non-zero entries in rows 4,5 and 6, corresponding to the 1-simplices which form the boundary of the 2-simplex. 
\end{example}

To find the persistence barcode from the boundary matrix $\partial$ we reduce each column of $\partial$ in order of increasing $i$, but only  allow the addition of columns with smaller index. 
When we reduce $\partial$ in this way we obtain a matrix $R$ which has the property that no two columns have their lowest 1 in the same row. From $R$, we can extract the persistence pairs of the simplex-wise filtration. 

\begin{lemma}[Pairing lemma \cite{book}]
Consider the boundary matrix $\partial$ for a simplex-wise filtration and let $R$ be the result of reducing $\partial$. Suppose column $j$ has its lowest 1 in row $i$. Then $(i,j)$ is a persistence pair. 
\end{lemma}
Note that in terms of the filtration $(i,j)$ being a persistence pair means that $\sigma_{j}$ killed the homology class which was born when $\sigma_{i}$ was added to the filtration. 

Some persistence pairs can be obtained from $\partial$ without the need for any reduction. In \cite{Ripser}, these are referred to as apparent pairs and a method to identify them was used to speed up the computation of Vietoris-Rips persistent homology. 

\begin{definition}[Apparent pair]
\label{apparent-pair}
Suppose that $(i,j)$ is a degree-$q$ persistence pair which satisfies the following conditions: 
\begin{itemize}
\item $\sigma_{i}$ is the $q$-simplex face of $\sigma_{j}$ which is last added to the filtration, 
\item $\sigma_{j}$ is the $(q+1)$-simplex coface of $\sigma_{i}$ which is first added to the filtration.
\end{itemize}
Then, following \cite{Ripser}, we call $(i,j)$ an apparent pair. 
\end{definition}

Indeed, if $\sigma_{i}$ is the last added face of $\sigma_{j}$ then this means that the lowest 1 in the $j$th column is at row $i$. The second condition means that $\sigma_{j}$ is the first coface to be added which has $\sigma_{i}$ as a face meaning that no column to the left of column $j$ will have its lowest 1 in row $i$. 

It should be noted that this method to calculate persistence barcodes applies to all degrees of homology. However, for Vietoris-Rips filtrations the matrix $\partial$ would be enormous. 
Since we are only interested in computing $\ph{1}(X)$ we only need columns which correspond to 2-simplices and rows which correspond to 1-simplices. This means a boundary matrix for computing $\ph{1}(X)$ has $O(n^3)$ columns and $O(n^2)$ rows. 
Note also that for Vietoris-Rips filtrations, some persistence pairs $(i,j)$ will correspond to homology classes with zero persistence, when $\textrm{diam}(\sigma_{i}) = \textrm{diam}(\sigma_{j})$. 
Apparent pairs are particularly useful in computing $\ph{1}(X)$ because they help us identify those persistence pairs which have trivial persistence. The following is Theorem 3.10 from \cite{Ripser}.

\begin{theorem}
\label{unique-distances-lemma}
Consider a point cloud $X$ with unique pairwise distances. Then the degree-1 apparent pairs for the Vietoris-Rips filtration $\Vr _{\bullet} (X)$ are precisely the persistence pairs with trivial persistence. 
\end{theorem}

\section{The Reduced Vietoris-Rips Complex}
\label{RVR-definition}

In this section we cover a new result on degree-1 Vietoris-Rips persistent homology of a point cloud $X$. The result proved here will be used extensively in later sections. The main benefit is that we reduce the number of 2-simplices to be examined from $O(n^3)$ to $O(n^2)$.

\subsection{Construction of the Reduced Vietoris-Rips Complex}

We first introduce the notion of the lune of a $1$-simplex. It differs from the geometric definition of the lune (given in Definition \ref{definition-traditional-lune}) by depending on the total order $<$ rather than geometry. 

\begin{definition}[Lune]
\label{def-lune}
Consider two points $y$ and $z$ in a point cloud $X$. The lune of a 1-simplex $\langle yz \rangle $, denoted by $\lune (\langle yz \rangle)$ is defined as the subset 
\begin{equation}
\{ x \in X \;|\; \langle yx \rangle < \langle yz \rangle \text{ and } \langle zx \rangle < \langle yz \rangle \}
\end{equation} 
\end{definition}

Recall that the symbol `$<$' refers to the total order for the simplices in the filtration as per Definition~\ref{binary-relation-on-simplices}. 
It should be noted that if one assumes pairwise unique distances then the above reduces to the geometric definition of the lune given by Toussaint in \cite{ToussaintRNG} and is repeated in Definition \ref{definition-traditional-lune}.  Note that in \cite{ToussaintRNG} the lune is only defined for point clouds with unique pairwise distances. 

\begin{definition}[Geometric definition of the lune]
\label{definition-traditional-lune}
    Consider a point cloud $X$ where $X$ has unique pairwise distances. Then for a given 1-simplex $\langle yz \rangle$ we define $\lune (\langle yz \rangle)$ as the following subset of $X$. 
    \begin{equation}
    \{ x \in X \;|\; d(x,y) < d(y,z) \text{ and } d(x,z) < d(y,z)\}
    \end{equation} 
    
\end{definition}

We can visualise the geometric definition of the lune using the intersection of two open balls of radius $d(y,z)$ centred at $y$ and $z$. This is shown in Figure \ref{lune-diagram}. It should be noted that with our definition of the lune there may be points of $X$ on the boundary of this region. 

\begin{figure}
    \centering
    \begin{tikzpicture}
        \draw[red] (-1,0) circle [radius=2];
        \draw[red, fill] (-1,0) circle [radius = 0.05];
        \draw[red] (1,0) circle [radius = 2];
        \draw[red, fill] (1,0) circle [radius = 0.05];
        \draw[red] (-1,0) -- (1,0);
        \node [below left] at (-1,0) {$y$};
        \node [below right] at (1,0) {$z$};
        \draw [blue, fill] (0.5, 0.5) circle [radius = 0.05];
        \node [left] at (0.5,0.5) {$x$};
        \draw [blue, fill] (2, 0.5) circle [radius = 0.05];
        \node [left] at (2,0.5) {$w$};
    \end{tikzpicture}
    \caption{ Depiction of the ``traditional" lune of $\langle yz \rangle$. Points of $X$ contained in the intersection of the two balls are in $\lune (\langle yz \rangle)$. In this figure, $x \in \lune (\langle yz \rangle) $ and $ w \notin \lune (\langle yz \rangle)$.}
    \label{lune-diagram}
\end{figure}

We quantify structure in this subset by its connectivity. 

\begin{definition}[Connected Components of a Lune]
\label{def-connected-components-of-lune}
Let $V = \lune (\langle yz \rangle)$ for two points $y,z \in X$, and construct a graph on $V$ by joining two points $p, q \in V$ if $\langle pq \rangle  < \langle yz \rangle$. Suppose this graph has $c$ connected components, then we say that the $\lune (\langle yz \rangle)$ has $c$ connected components. 

\end{definition}

In Section~\ref{Algorithm}, we need to quickly determine if a lune has only one connected component. To facilitate this, we introduce the lens of a lune. 

\begin{definition}[Lens]
\label{lens}
Consider a point cloud $X$. Let $\langle yz \rangle$ be a 1-simplex. Then we define the $\lens (\langle yz \rangle)$ to be all points $x \in X$ such that $\Angle (yxz) > \frac{5\pi}{6}$. In Figure \ref{lens-diagram-after-definition}, the red region shows which points of $X$ would be contained in $\lens(\langle yz \rangle)$. 
\end{definition}

\begin{figure}[h!]
    \centering
        \begin{tikzpicture}[scale = 0.7]
            \draw [black] (3,0) arc [radius = 6, start angle = 0, end angle = 60];
            \draw [black] (-3,0) arc [radius = 6, start angle = 180, end angle = 120];
            \draw [black] (-3,0) arc [radius = 6, start angle = 180, end angle = 240];
            \draw [black, fill] (0,0) circle [radius = 0.25];
            \draw [black] (3,0) arc [radius = 6, start angle = 360, end angle = 300];
            \draw [black] (3,0) arc [radius = 6, start angle = 60, end angle = 120];
            \draw [black] (-3,0) arc [radius = 6, start angle = 240, end angle = 300];
            \draw [black] (-3,0) -- (3,0);
            \node [below left] at (-3,0) {$z$};
            \node [below right] at (3,0) {$y$};
            \draw [black, fill= red] (-3,0) to [out=30,in=150] (3,0) to [out =210, in= 330] (-3,0); 
            \node [right] at (-0.6, -0.2) {$x$};
            \draw [black,fill = black] (-0.6,-0.2) circle [radius = 0.1];
            \node [right] at (1,2) {$w$};
            \draw [black,fill = black] (1,2) circle [radius = 0.1];
    \end{tikzpicture}
    \caption{Points of $X$ that lie in the red region are contained in $\lens (\langle yz \rangle)$. Here, $x\in \lens(\langle yz \rangle)$ and $w\notin \lens(\langle yz \rangle)$. Note that the red region itself is not $\lens (\langle yz \rangle)$! 
    }
    \label{lens-diagram-after-definition}
\end{figure}

Using this definition of the lens of a $1$-simplex we give a lemma that provides a sufficient condition for the lune to have a single connected component.

\begin{lemma}
\label{lens-lemma}
Let $X \subset \mathbb{R}^{D}$ be a Euclidean point-cloud. Consider a 1-simplex $\langle yz \rangle $. If there is a point that $x$ such that $\Angle (yxz) > \frac{5\pi}{6}$ then $\lune (\langle yz \rangle)$ consists of only one connected component. 
\end{lemma}

\begin{proof}

Before we begin the proof observe that for $w\in \lune(yz)$, $\Angle(ywz) \geq \frac{\pi}{3}$ because $\langle yz \rangle$ is the longest edge of the triangle $\langle w y z \rangle$. If $x$ is the only point in $\lune(\langle yz \rangle)$ then we are done. Otherwise let $w$ be another point in $\lune(\langle yz \rangle)$. Four points $x,y,z,w \in R^D$ span at most a three dimensional subspace so without loss of generality we assume $x,y,z,w\in \mathbb{R}^3$. We first show that the three dimensional case can always be reduced to the two dimensional. Note if $x$ lies on $\langle yz \rangle$ then $x,y,z,w$ can be viewed as being on the same plane so consider the case when $x$ does not lie on  $\langle yz \rangle$. 

Construct a new point $w'$ by rotating the triangle $\langle ywz\rangle$ about an axis along $\langle yz\rangle$ until it lies in the plane containing $x,y$ and $z$, such that $w'$ lies on the other side of $\langle yz\rangle$ than $x$. Let $p$ be the point in the intersection of $\langle yz\rangle$ and $\langle xw'\rangle$. This construction is illustrated in Figure \ref{fig:lens-test-proof-picture-new}.

\begin{figure}\label{fig:rotated w}
\begin{center}
\begin{tikzpicture}[scale=1.6, line join=round, line cap=round]

  \fill[blue!6]  (-2.5,-1.4) -- (4.7,-1.4) -- (5.5,1.05) -- (-1.7,1.05) -- cycle;
  \draw[blue!45] (-2.5,-1.4) -- (4.7,-1.4) -- (5.5,1.05) -- (-1.7,1.05) -- cycle;

  \coordinate (y)  at (0, 0);
  \coordinate (z)  at (3.5, 0);
  \coordinate (x)  at (0.5, -0.7);     
  \coordinate (wp) at (3.25, 0.525);   
  \coordinate (p)  at (2.071, 0);      
  \coordinate (w)  at (2.875, 1.562);  

  \draw[thick]             (y) -- (z);
  \draw[thick, tick]       (y)  -- (wp);
  \draw[thick, twoticks]   (wp) -- (z);
  \draw[thick]             (x) -- (p);          
  \draw[thick, threeticks] (p) -- (wp);

  \draw[thick, tick]       (y) -- (w);
  \draw[thick, twoticks]   (z) -- (w);
  \draw[thick, threeticks] (p) -- (w);

  \draw[thick] (x) -- (y);
  \draw[thick] (x) -- (z);

  \foreach \pt in {y, z, x, wp, w, p}
    \fill (\pt) circle (1.5pt);

  \node[left]        at (y)  {$y$};
  \node[right]       at (z)  {$z$};
  \node[below]       at (x)  {$x$};
  \node[above right] at (wp) {$w'$};
  \node[above]       at (w)  {$w$};
  \node[above left]  at (p)  {$p$};

\end{tikzpicture}
\end{center}
\caption{Reducing the proof to the case of four points in $\R^2$.}
\label{fig:lens-test-proof-picture-new}
\end{figure}

By construction we have $d(w',y)=d(w,y)|\leq d(y,z)$ and $d(w',z)=d(w,z)\leq d(y,z)$. We also have $\Angle(yw'z)=\Angle(ywz)\geq \frac{\pi}{3}$.  
By construction we have $d(x,w')=d(x,p)+d(p,w)$ and $d(p,w)=d(p,w')$. By the triangle inequality, $d(x,w)\leq d(x,p)+d(p,w)$ so $d(x,w)\leq d(x,w')$. It is this sufficient to show that when $x,y,z,w'$ all in the plane then $d(x,w')<d(y,z)$.  

Suppose by way of contradiction that $d(x,w')\geq d(y,z)$. As both $d(y,x)\leq d(y,z)$ and $d(y,w')\leq d(y,z)$ we have $\langle xw'\rangle$ must be the largest edge in the triangle (possibly with ties), and by the largest angle property of triangles $\Angle(xyw')\geq \Angle(yxw')$. Similarly, $\Angle(xzw')\geq \Angle(zxw')$. Together these imply $$\Angle(xyw')+\Angle(xzw')\geq \Angle(yxw')+\Angle(w'xz)=\Angle(yxz).$$
By assumption $\Angle(yxz)>5\pi/6$. So the angle sum of the quadrilateral with vertices $x,z,w',y$ satisfies
$$\Angle(yxz)+\Angle(xzw')+\Angle(zw'y)+\Angle(w'yx)>5\pi/6+5\pi/6+\pi/3=2\pi.$$
This contradicts that the angle sum of any quadrilateral is $2\pi$ so we can conclude that $d(w',x)<d(y,z)$.
\end{proof}

We now describe how to build the reduced Vietoris-Rips complex. First we define a Lune function that records the selection of a single point per connected component of a lune. 

\begin{definition}[Lune function]
\label{lune-function}
    Let $\Vr^{1}_{\infty}(X)$ be the set of $1$-simplices in $\Vr_{\infty}(X)$. We define a lune function $L:\Vr^{1}_{\infty}(X)\rightarrow 2^{X}$ as a function that takes a one simplex $\langle yz \rangle$, with $c_{yz}$ connected components in its lune, to a set $\{x_{1},...,x_{c_{yz}}\}$. The points $x_{1},...,x_{c_{yz}}$ are chosen from the connected components of $\langle yz \rangle$, one point from each connect component. 
\end{definition}

\begin{definition}[Reduced Vietoris-Rips complex]
\label{def-reduced-vietoris-rips-complex}
Consider a point-cloud $X$ and a lune function $L$. A Reduced Vietoris-Rips complex of $X$ with scale $r$, 
denoted $\mathcal{R}^{L}_r(X)$ 
is the simplicial complex such that 

\begin{itemize}
    \item 0-simplices are all points of $X$,

    \item 1-simplices are edges $\langle yz \rangle $ with $d(y,z) \leq r$,

    \item 2-simplices are $\langle yzx \rangle$, where $\diam(\langle yzx \rangle) \leq r$ and $x\in L(\langle yz \rangle)$. 

\end{itemize}

\end{definition}

Since $\rvr_r^{L}(X)$ is a simplicial complex and $\rvr^{L}_{r'}(X) \subset \rvr^{L}_{r}(X)$ for $r' <r$ it follows that we 
have a filtration for increasing scale $r$.   
\begin{definition} 
\label{def-reduced-vietoris-rips-filtration}
A reduced Vietoris-Rips filtration on $X$ is any filtration of the form $\rvr^{L}_{\bullet}(X)$, where $L$ is a lune function for $X$.
\end{definition}

When $r$ is the diameter of $X$, $\rvr^{L}_{r}(X)$ will have $O(n)$ 0-simplices and $O(n^2)$ 1-simplices. 
The following Lemma~\ref{bounded-connected-components-proof} shows that when $X$ is a subset of Euclidean space $\mathbb{R}^D$, the number of connected components of a lune is bounded by a constant independent of $n$. It follows that at most $O(1)$ 2-simplices will be added in each lune and thus there will be $O(n^2)$ 2-simplices in $\mathcal{R}^{L}_{r}(X)$.

\begin{lemma}
\label{bounded-connected-components-proof}
    Consider a Euclidean point-cloud $X \subset \mathbb{R}^{D}$. The number of connected components of a lune is bounded above by a constant independent of $|X| = n$ and dependent on $D$. 
\end{lemma}

\begin{proof}
Consider a $1$-simplex $\langle yz \rangle$ and let $r = d(y,z)$. Consider the following subset:
\begin{equation}
    Q_1 := \{x\in \mathbb{R}^D \;|\; d(x,y) \leq r \text{ and } d(x,z) \leq r \}. 
\end{equation}

Note that $Q_{1}$ is the closure of $B(y,r) \cap B(z,r)$.
Then the maximum number of connected components in $\lune(\langle yz \rangle)$ is bounded by the maximum number of disjoint open balls of radius $\frac{r}{2}$ with centres in $Q_1$. We can put an upper bound on this number by considering a ball of radius $2r$, $Q_2$, centred at the midpoint of $y$ and $z$ as depicted in Figure \ref{connected-components-proof}. Certainly, the number of disjoint balls with radius $\frac{r}{2}$ that can be placed inside $Q_2$ is more than the maximum number of connected components in $\lune(\langle yz \rangle)$. To see this, consider a set of open balls of radius $\frac{r}{2}$ with centres in $Q_1$ such that all the balls are mutually disjoint. All these balls are contained in $Q_{2}$ and we can always place one more ball centred $\frac{3r}{2}$ away from the midpoint of $y$ of $z$, see Figure \ref{connected-components-proof}. This means that the number of mutually disjoint balls of radius $\frac{r}{2}$ with centres in $Q_1$ is bounded by the number of disjoint balls of radius $\frac{r}{2}$ that can be packed inside $Q_2$. 
This number is bounded by dividing the volume of $Q_2$ by the volume of a ball with radius $\frac{r}{2}$. It follows that the number of connected components of a $\lune (\langle yz \rangle)$ is bounded by $4^D$. It should be said that this is a rather crude overestimate of the maximum number of connected components.  
\end{proof}

\begin{figure}
\centering
\begin{tikzpicture}[scale = 0.5]
    \draw[very thick, red] (0,-2) -- (0,2);
    \draw [very thick, red] (0,2) arc [radius = 4, start angle = 90, end angle = 150];
    \draw [very thick, red] (0,-2) arc [radius = 4, start angle = 270, end angle = 210];
    \draw [very thick, red] (0,2) arc [radius = 4, start angle = 90, end angle = 30];
    \draw [very thick, red] (0,-2) arc [radius = 4, start angle = 270, end angle = 330];
    \node [below right] at (0,0) {$\frac{y+z}{2}$};
    \draw [red, fill] (0, 0) circle [radius = 0.1];
    \draw [red] (0,0) circle [radius = 8];
    \draw [very thick, blue] (0,6) circle [radius = 2];
    \draw [blue, fill] (0,6) circle [radius = 0.05];
    \draw [blue, fill] (-2.5,0.5) circle [radius = 0.05];
    \draw [blue] (-2.5,0.5) circle [radius = 2];
    \draw [blue, fill] (1.5, 1.5) circle [radius = 0.05];
    \draw [blue] (1.5, 1.5) circle [radius = 2];
    \node [above right] at (5.657, 5.657) {$Q_2$};
    \node [right] at (0,6) {$q_2$};
    \node [below right] at (2,-1.2) {$Q_1$};
    \node [above right] at (0,2) {$y$};
    \draw [red, fill] (0,2) circle [radius = 0.1];
    \draw [red, fill] (0,-2) circle [radius = -0.1];
    \node [below right] at (0,-2)  {$z$};
    \draw [blue, dashed] circle [radius = 4];
\end{tikzpicture}
\caption{Two points $p_1$ and $p_2$ inside $Q_1$ with their respective balls of radius $\frac{r}{2}$. $Q_2$ is centered at $\frac{y+z}{2}$ and has radius $2r$. Suppose we place $p_1,...,p_m$ inside $Q_1$ such that $\{B(p_i,\frac{r}{2})\}_{i\in \{1,...,m\}}$ are mutually disjoint. No matter which disjoint balls of radius $\frac{r}{2}$ we centre in $Q_1$, $B(q_2,r)$ will always be able to be placed inside $Q_2$ such that $B(q_2,\frac{r}{2})$ is disjoint from $\cup_{i=1}^{m} B(p_i,\frac{r}{2})$. The boundary of $B(\frac{y+z}{2}, r)$ is included as a dashed line for visualisation assistance. }
\label{connected-components-proof}
\end{figure}

\begin{remark}
    Lemma \ref{bounded-connected-components-proof} can be generalized for $X$ a subset of a metric space with finite doubling dimension.
\end{remark}

\begin{corollary}
\label{corollary-bounded-connected-components}
    For every $r\geq 0$, the number of $0$-simplices of $\rvr_{r}^{L}(X)$ is bounded by $O(n)$ and the number of $1$-simplices is bounded by $O(n^2)$. If $X$ is a point-cloud in $ \mathbb{R}^D$, then Lemma \ref{bounded-connected-components-proof} shows that the number of $2$-simplices in $\rvr_{r}^{L}(X)$ will be bounded above by $O(4^{D}n^2)$. 
\end{corollary}

In the interest of reducing clutter from here on in, we will drop the $L$ from $\rvr_{\bullet}^{L}(X)$ and simply write $\rvr_{\bullet}(X)$ with the implicit understanding that a specific $L$ has been chosen.

\section{The Relationship between Vietoris-Rips complexes and their reduced counterparts}
\label{RVR-theorem}

So far we have detailed what a Reduced Vietoris-Rips complex is but we have not shown why it is useful. Here we state and prove a new result 
which relates the degree-1 persistent homology of the  standard Vietoris-Rips filtration of $X$ with the degree-1 persistent homology of the Reduced Vietoris-Rips filtration of $X$.


\begin{theorem}[Main Theorem]
\label{theorem-VR-RVR-isomorphism}
Consider a finite metric space $X$.   
Then there exists a family of isomorphisms $\theta_{\bullet}$ such that the following diagram commutes

\begin{equation}
 \begin{tikzcd}[ampersand replacement=\&]
  H_{1}(\rvr_{r_1}(X)) \arrow[r, "f_{r_1}^{r_2}"] \arrow[d, "\theta_{r_1}"'] \& H_{1}(\rvr_{r_2}(X)) \arrow[d, "\theta_{r_2}"] \\
  H_{1}(\Vr_{r_1}(X)) \arrow[r, "g_{r_1}^{r_2}"] \& H_{1}(\Vr_{r_2}(X))
  \end{tikzcd} 
\end{equation}

\noindent for all $r_1$ and $r_2$ such that $0 \leq  r_1 < r_2$.  Above, $f_{r_1}^{r_2}$ and $g_{r_1}^{r_2}$ are the maps at homology level induced by the natural inclusions. 
\end{theorem}

\paragraph{The proof of the Main Theorem} 

Let $r$ be an arbitrary non-negative number. We first define $\theta_{r}$ and then show it is an isomorphism. Let $\gamma + B_{1}(\rvr_{r}(X))$ be an element in $H_{1}(\rvr_{r}(X))$. Then we define $\theta_{r}$ as follows:

\begin{equation}
\theta_{r}(\gamma + B_{1}(\rvr_{r}(X))) := \gamma + B_{1}(\Vr_{r}(X))
\end{equation}

This mapping is well defined because $B_{1}(\rvr_{r}(X)) \subset B_{1}(\Vr_{r}(X))$.

\begin{lemma}
\label{theta-is-surjective}
    $\theta_{r}$ is surjective for all $r \geq 0$. 
\end{lemma}

\begin{proof}
    Let $\gamma + B_{1}(\Vr_{r}(X)) \in H_{1}(\Vr_{r}(X))$. 
    Since $\gamma \in Z_1(\Vr_{r}(X))$ and all 1-simplices of $\Vr_{r}(X)$ are also in $\rvr_{r}(X)$, it follows that $\gamma + B_{1}(\rvr_{r}(X))$ is mapped to $\gamma + B_{1}(\Vr_{r}(X))$ by $\theta_{r}$.
\end{proof}

In order to show that $\theta_{r}$ is an isomorphism for all $r$ we need to show that $\theta_{r}$ is injective for all $r$. Before we do so, we will need the following preliminary lemmas. 

\begin{lemma}
\label{isomorphism-inclusions}
    If $\theta_{s}$ is an isomorphism, then it follows that $B_{1}(\Vr_{s}(X)) = B_{1}(\rvr_{s}(X))$. 
\end{lemma}

\begin{proof}
The inclusion $B_{1}(\rvr_{s}(X)) \subset B_{1}(\Vr_{s}(X))$ follows from the fact that $\rvr_{s}(X) \subset \Vr_{s}(X)$. To get the reverse inclusion let $\gamma \in B_{1}(\Vr_{s}(X))$. Then, as the $1$-simplices are the same for $\Vr_s(X)$ and $\rvr_s(X)$, we have $\theta_{s}(\gamma + B_{1}(\rvr_{s}(X))) = \gamma + B_{1}(\Vr_{s}(X)) = 0 + B_{1}(\Vr_{s}(X))$. Since $\theta_{s}$ is an isomorphism it follows that $\gamma \in B_{1}(\rvr_{s}(X))$. Thus we have the reverse inclusion $B_{1}(\Vr_{s}(X)) \subset B_{1}(\rvr_{s}(X))$.
\end{proof}

\begin{lemma}
\label{first-edge-proof}
    Consider arbitrary $r>0$ and suppose that $\theta_{s}$ is injective for all $s<r$. Let all the $1$-simplices of diameter $r$ be $\langle y_1z_1 \rangle < \cdots < \langle y_{m_{r}} z_{m_{r}} \rangle $. Then 

    \begin{equation}
    \label{same-edges-equation-1}
         \{ \partial (\langle x y_1 z_1 \rangle)\; | \; x\in \lune (\langle y_1 z_1\rangle) \} \subset B_{1}(\rvr_{r}(X)).
    \end{equation}

\begin{proof}
    First observe that by Lemma \ref{theta-is-surjective} that $\theta_s$ is an isomorphism for all $s < r$. This means that $B_{1}(\rvr_{s}(X)) = B_{1}(\Vr_{s}(X))$ for $s < r$ by Lemma \ref{isomorphism-inclusions}. 

    We note that any point $x \in \lune (\langle y_1 z_1 \rangle)$ must necessarily have  $\max\big(d(x,y_1), d(x,z_1)\big) < r$. To see why this is so, suppose $d(x,y_1) = r$, then since $x\in \lune (\langle y_1 z_1 \rangle)$ it follows that $ \langle x y_1 \rangle < \langle y_1 z_1 \rangle $, but this is a contradiction since there should be no $1$-simplex of diameter $r$ that appears in the filtration before $ \langle y_1 z_1 \rangle$. An analogous argument can be used to show that $d(x,z_1) < r$. 

To this end, let $x$ be any point in $\lune (\langle y_1 z_1 \rangle)$ and consider the $2$-simplex $\sigma = \langle x y_1 z_1 \rangle$. 
Let $w \in L(\langle y_1 z_1 \rangle)$ be in the same connected component as $x$.

If $x = w$ then there is nothing to prove since $\sigma \in \rvr_{r}((X))$ by definition and thus $\partial \sigma \in B_{1}(\rvr_{r}(X))$. 
Otherwise, since $x$ and $w$ are in the same component we know there exist $v_0, ..., v_{q} \in \lune(\langle y_{1} z_{1} \rangle)$ such that $x = v_0,v_1,...,v_{q}=w$ forms a path with $\langle v_{i} v_{i+1} \rangle < \langle y_{1} z_{1} \rangle$,  $d(v_i, y_1) < r$ and $d(v_i, z_1) < r$. 
Using the fact that $\langle y_1 z_1 \rangle$ is the first $1$-simplex of diameter $r$ to appear in the filtration, it also follows that $d(v_{i}, v_{i+1}) < r$ for $i = 1,...,q-1$. Using the fact that $\partial (\partial (\langle y_1 z_1 v_{i} v_{i+1} \rangle)) = 0$ we see that 
\begin{equation}
\label{proof-equation}
\partial (\langle y_1 z_1 v_{i} \rangle )  = \partial (\langle y_1 z_1 v_{i+1} \rangle ) + \partial (\langle v_{i} v_{i+1} y_1 \rangle) + \partial (\langle v_{i} v_{i+1} z_1 \rangle). 
\end{equation}


\medskip
\noindent
The 2-simplex $\langle v_{i}v_{i+1}y_1 \rangle$ must have diameter less than $r$ because all $1$-simplex faces have diameter less than $r$. Similarly, $\diam (\langle v_{i} v_{i+1} z_1 \rangle) < r$. Since $\theta_{\diam(\langle v_{i} v_{i+1} y_1 \rangle)}$ and $\theta_{\diam (\langle v_{i} v_{i+1} z_1 \rangle)}$ are isomorphisms, it follows that 
\begin{equation}
\partial (\langle v_{i} v_{i+1} y_1 \rangle ) \in B_{1}(\Vr_{\diam(\langle v_{i} v_{i+1} y_1 \rangle)}(X)) = B_{1}(\rvr_{\diam(\langle v_{i} v_{i+1} y_1 \rangle)}(X)) \subset \rvr_{r}(X)
\end{equation}
and 
\begin{equation}
\partial (\langle v_{i}v_{i+1}z_1 \rangle) \in B_{1}(\Vr_{\diam (\langle v_{i} v_{i+1} z_1 \rangle)}(X)) = B_{1}(\rvr_{\diam (\langle v_{i} v_{i+1} z_1 \rangle)}(X)) \subset \rvr_{r} (X)
\end{equation}
by Lemma \ref{isomorphism-inclusions}. 
Plugging in $i = q-1$ into (\ref{proof-equation}) we obtain the following. 
\begin{equation}
\label{second-proof-equation}
    \partial (\langle y_1 z_1 v_{q-1} \rangle )  = \partial (\langle y_1 z_1 v_{q} \rangle ) + \partial (\langle v_{q-1} v_{q} y_1 \rangle) + \partial (\langle v_{q-1} v_{q} z_1 \rangle). 
\end{equation}
We have already shown that the second and third term on the right hand side of (\ref{second-proof-equation}) are in $B_{1}(\rvr_{r}(X))$. 
The first term on the right hand side of (\ref{second-proof-equation}) is in $B_{1}(\rvr_{r}(X))$ because $v_q = w \in L(\langle y_1 z_1 \rangle)$, meaning the $2$-simplex $\langle y_1 z_1 v_{q} \rangle \in \rvr_{r}(X)$ by definition. Thus it follows that $\partial (\langle y_1 z_1 v_{q-1} \rangle) \in B_{1}(\rvr_{r}(X))$. Using (\ref{proof-equation}) repeatedly by letting $i = q-2, q-3, ...,0$ we can subsequently show that $\partial (\langle y_1 z_1 v_{q-2}\rangle), ..., \partial (\langle y_1 z_1 v_{0}\rangle) \in B_{1}(\rvr_{r}(X))$. Remembering that $v_{0} = x$ we thus have that $\partial \sigma =\partial( \langle x y_{1} z_{1} \rangle) \in B_{1}(\rvr_{r}(X))$.
\end{proof}

\end{lemma}

\begin{lemma}
\label{the-big-lemma}
    Consider $r>0$ arbitrary and suppose that $\theta_{s}$ is injective for all $s<r$. Let $\sigma$ be a $2$-simplex in $\Vr_{r}(X)$ such that $\diam (\sigma) = r$, then we have that $\partial \sigma \in B_{1}(\rvr_{r}(X))$. That is to say, $B_{1}(\rvr_{r}(X)) = B_{1}(\Vr_{r}(X))$. 
\end{lemma}

\begin{proof}
    First observe by Lemma \ref{theta-is-surjective} that $\theta_s$ is an isomorphism for all $s < r$. This means that $B_{1}(\rvr_{s}(X)) = B_{1}(\Vr_{s}(X))$ by Lemma \ref{isomorphism-inclusions}. 
    
    Let all the $1$-simplices of diameter $r$ be $\langle y_1z_1 \rangle < \cdots < \langle y_{m_{r}} z_{m_{r}} \rangle $. We wish to show that all $2$-simplices $\sigma$, with $\diam(\sigma) = r$ are such that $\partial \sigma \in B_{1}(\rvr_{r}(X))$. This is equivalent to showing that 
    \begin{equation}
    \label{same-edges-equation}
         \{ \partial (\langle x y_j z_j \rangle)| x\in \lune (\langle y_j z_j \rangle) \} \subset B_{1}(\rvr_{r}(X))
    \end{equation}
for $j = 1,...,m_{r}$. We will prove that Equation \ref{same-edges-equation} holds for $j = 1,...,m_{r}$ by induction. That is, we will show the following:

\begin{itemize}
    \item Equation \ref{same-edges-equation} holds for $j=1$. This is just Lemma \ref{first-edge-proof}.

    \item If Equation \ref{same-edges-equation} holds for $j=1,...,k$, then it holds for $j = k+1$. 
\end{itemize}

The proof of the second statement is similar to that of Lemma \ref{first-edge-proof} but has some subtle differences. 

Let $x$ be any point in $\lune (\langle y_{k+1} z_{k+1} \rangle)$ and consider the $2$-simplex $\langle x y_{k+1} z_{k+1} \rangle$. 
Let $w \in L(\langle y_{k+1}z_{k+1} \rangle)$ belong to the same connected component as $x$. If $x=w$ then there is nothing to prove since $\sigma \in \rvr_{r}(X))$ by definition and thus $\partial \sigma \in B_{1}(\rvr_{r}(X))$. Otherwise, since $x$ and $w$ are in the same component we know there exists $v_{0},...,v_{l} \in \lune (\langle y_{k+1} z_{k+1} \rangle)$ such that $x = v_0, v_1, ..., v_l = w$ forms a path such that $\langle v_{i}v_{i+1} \rangle < \langle y_{k+1} z_{k+1} \rangle $ for $i = 0,...,l-1$. We know that $\partial (\partial (\langle y_{k+1} z_{k+1} v_{i}v_{i+1} \rangle )) = 0$ and thus 
\begin{equation}
\label{equation-general-proof}
    \partial (\langle y_{k+1}z_{k+1} v_{i} \rangle ) = \partial (\langle y_{k+1}z_{k+1} v_{i+1} \rangle) + \partial (\langle v_{i}v_{i+1}y_{k+1} \rangle) + \partial (\langle v_{i}v_{i+1}z_{k+1} \rangle).
\end{equation}
We now show that $\partial (\langle v_{i}v_{i+1}y_{k+1} \rangle) \in B_{1}(\rvr_{r}(X))$:

\begin{itemize}
    \item We already established that $\langle v_{i}v_{i+1} \rangle < \langle y_{k+1}z_{k+1} \rangle$. 

    \item $\langle v_{i} y_{k+1} \rangle < \langle y_{k+1}z_{k+1} \rangle$ since $v_{i} \in \lune (\langle y_{k+1}z_{k+1} \rangle)$. 

    \item $\langle v_{i+1}y_{k+1} \rangle < \langle y_{k+1}z_{k+1} \rangle$ since $v_{i+1} \in \lune (\langle y_{k+1}z_{k+1} \rangle)$. 
\end{itemize}

Since all faces of $\langle v_{i}v_{i+1}y_{k+1} \rangle$ appear in the filtration before $\langle y_{k+1}z_{k+1} \rangle$ it follows that 
either $\diam(\langle v_{i}v_{i+1}y_{k+1} \rangle) < r$, or one of the edges has diameter equal to $r$, in which case that edge must be one of $\langle y_jz_j \rangle$ for $j= 1, \ldots, k$. 
By our inductive assumption, it follows that $ \partial (\langle v_{i}v_{i+1}y_{k+1} \rangle) \in B_1(\rvr_r(X))$.  
Similarly, $ \partial (\langle v_{i}v_{i+1}z_{k+1} \rangle) \in B_1(\rvr_r(X))$.

Plugging $i = l-1$ into (\ref{equation-general-proof}) we obtain the following.
\begin{equation}
\label{equation-general-proof-q-1}
    \partial (\langle y_{k+1}z_{k+1} v_{l-1} \rangle ) = \partial (\langle y_{k+1}z_{k+1} v_{l} \rangle) + \partial (\langle v_{l-1}v_{l}y_{k+1} \rangle) + \partial (\langle v_{l-1}v_{l}z_{k+1} \rangle).
\end{equation}

We have just shown that the second and third term on the right hand side of (\ref{equation-general-proof-q-1}) are in $B_{1}(\rvr_{r}(X))$. Remembering that $v_{l} = w \in L(\langle y_{k+1}z_{k+1} \rangle)$, the first term on the right hand side of (\ref{equation-general-proof-q-1}) is in $B_{1}(\rvr_{r}(X))$ because $\langle y_{k+1}z_{k+1}v_{l} \rangle$ is a $2$-simplex in $\rvr_{r}(X)$ by the definition of the Reduced Vietoris-Rips complex. Thus it follows that \\ $\partial (\langle y_{k+1}z_{k+1}v_{l-1} \rangle) \in B_{1}(\rvr_{r}(X))$. 
Using (\ref{equation-general-proof}) repeatedly by letting $i = l-2, l-3, ...,0$ we can subsequently show that $\partial (\langle y_{k+1}z_{k+1}v_{l-2}),...,\partial(\langle y_{k+1}z_{k+1}v_{0}\rangle) \in B_{1}(\rvr_{r}(X))$. Remembering that $v_{0} = x$ we thus have that $\partial \sigma = \partial (\langle xy_{k+1}z_{k+1} \rangle ) \in B_{1}(\rvr_{r}(X))$.

\end{proof}

We are finally in a position to show that $\theta_{r}$ is injective for all $r \geq 0$.

\begin{lemma}
    $\theta_{r}$ is injective for all $r\geq 0$.
\end{lemma}

\begin{proof}
We will use a proof by induction by inducting on the value of $r$. That is to say, we will prove the following.

\begin{itemize}
    \item $\theta_{0}$ is injective. 

    \item If $\theta_{s}$ is injective for all $s<r$, then $\theta_{r}$ is injective. 
\end{itemize}

Since $H_{1}(\Vr_{0}(X))$ and $H_{1}(\rvr_{0}(X))$ are zero the injectivity of $\theta_{0}$ is immediate. Now we show that if $\theta_{s}$ is injective for all $s<r$, then $\theta_{r}$ is injective. Suppose that \\ $\theta_{r}(\gamma + B_{1}(\rvr_{r}(X))) = \gamma + B_{1}(\Vr_{r}(X)) = 0 + B_{1}(\Vr_{r}(X))$. If we can show that \\ $\gamma \in B_{1}(\rvr_{r}(X))$ then it would follow that $\theta_{r}$ is injective. Since \\ $\gamma + B_{1}(\Vr_{r}(X)) = 0 + B_{1}(\Vr_{r}(X))$ it follows that $\gamma \in B_{1}(\Vr_{r}(X))$. By Lemma \ref{the-big-lemma} we know that $B_{1}(\rvr_{r}(X)) = B_{1}(\Vr_{r}(X))$. Hence $\gamma \in B_{1}(\rvr_{r}(X))$ and thus $\theta_{r}$ is injective. 

\end{proof}

Finally, we show that the diagram in Theorem \ref{theorem-VR-RVR-isomorphism} is commutative to complete the proof of Theorem $\ref{theorem-VR-RVR-isomorphism}$. 

\begin{lemma}
    The diagram in Theorem \ref{theorem-VR-RVR-isomorphism} is commutative.
\end{lemma}

\begin{proof}
    Let $0 < r_1 < r_2$ since the case where $r_1 = 0$ is trivial. Let $\gamma + B_{1}(\rvr_{r_1}(X))$ be an element of $H_{1}(\rvr_{r_1}(X))$. Then \\
    $g_{r_1}^{r_2}(\theta_{r_1} (\gamma + B_{1}(\rvr_{r_1}(X))) = g_{r_1}^{r_2}(\gamma + B_{1}(\Vr_{r_1}(X))) = \gamma + B_{1}(\Vr_{r_2}(X))$ and \\ $\theta_{r_2}(f_{r_1}^{r_2}(\gamma + B_{1}(\rvr_{r_1}(X))) = \theta_{r_2}(\gamma + B_{1}(\rvr_{r_2}(X))) = \gamma + B_{1}(\Vr_{r_2}(X))$. Thus commutativity is proven.
\end{proof}

\begin{remark}
    \label{remark:not-a-homotopy}
    Theorem \ref{theorem-VR-RVR-isomorphism} does not arise from an underlying homotopy. In general, $\rvr_{\bullet}(X)$ will not be homotopy equivalent to $\Vr_{\bullet}(X)$ since $\Vr_{\bullet}(X)$ is able to have higher degree homology groups. The subfiltration $\rvr_{\bullet}(X)$ is not even homotopy equivalent to the $2$-skeleton of $\Vr_{\bullet}(X)$. Consider four points in $\mathbb{R}^3$ such that they span a tetrahedron. Then the $2$-skeleton of $\Vr_{\bullet}(X)$ consists of all four $2$-simplex faces, whereas $\rvr_{\bullet}(X)$ consists of only three of those faces. These are clearly not homotopy equivalent as one is contractible and the other is not. 
\end{remark}

\section{Geometric data structures}
\label{GDS}
In this section we introduce some geometric data structures used to implement our method for computing VRPH. These data structures are frequently utilised in algorithms working with finite sets of points in Euclidean spaces and finite metric spaces more generally.

\subsection{Minimum Spanning Trees}

Minimum Spanning Trees are fundamental structures used in the study of graphs. The first algorithm for computing a minimum spanning tree was developed by Bor\r{u}vka in order to find an efficient electrical coverage of Moravia \cite{boruvka1926prispevek}. Other commonly used algorithms for computing the minimum spanning tree include Prim's algorithm \cite{prim1957shortest}, actually first invented by Vojtech Jarnik \cite{jarnik1930jistem} and Kruskal's Algorithm \cite{kruskal1956shortest}. For us, minimum spanning trees are useful because the edges are the 1-simplices that correspond to deaths of degree-0 homology classes in the Vietoris-Rips filtration. We briefly go over the definition of a minimum spanning tree and state some results about their relation to Vietoris-Rips persistent homology. Whilst we work in the context of Euclidean point-clouds, these results hold for arbitrary point-clouds. In what follows, recall that $\spx{1}(\Vr_{\infty}(X))$ denotes the set of $1$-simplices in $\Vr_{\infty}(X)$ and, in a similar fashion, $\spx{0}(\Vr_{\infty}(X))$ denotes the set of $0$-simplices in $\Vr_{\infty}(X)$. 

\begin{definition}
Consider a point-cloud $X$. Then a subset of $1$-simplices $S \subset \spx{1}(\Vr_{\infty}(X))$ is said to be acyclic if $H_{1}(S\cup \spx{0}(\Vr_{\infty}(X))) = 0$. 
\end{definition}
 
\begin{definition}
Consider a point-cloud $X$. Then a subset of $1$-simplices $S \subset \spx{1}(\Vr_{\infty}(X))$ is said to be spanning if $H_{0}(S\cup \spx{0}(\Vr_{\infty}(X)))=\mathbb{Z}_{2}$.
\end{definition}

\begin{definition}
Consider a point-cloud $X$. Then a subset of $1$-simplices $S \subset \spx{1}(\Vr_{\infty}(X))$ is said to be a spanning tree of $X$ if it is both acyclic and spanning.     
\end{definition}

We now define the notion of the weight of a graph so that we can then define a \emph{minimum} spanning tree. 

\begin{definition}
Consider a graph $G$ consisting of edges from $\spx{1}(\Vr_{\infty}(X))$ and vertices from $\spx{0}(\Vr_{\infty}(X))$.  For $\sigma \in G$, let $w(\sigma) = \diam (\sigma)$ and define 

\begin{equation}
w(G) := \sum_{\sigma \in G} w (\sigma)
\end{equation}

We call $w(G)$ the weight of the graph. 
\end{definition}

\begin{definition}
A minimum spanning tree is a spanning tree $S$ for $X$ with minimum value of $w(S)$ out of all spanning trees on $X$.  
\end{definition}

It should be noted that if all pairwise distances are unique then we can talk about \emph{the} minimum spanning tree for $X$ which we denote as $\mst (X)$. If $X$ does not have unique pairwise distances, we use Definition \ref{def-second-mst-definition} to select a specific minimum spanning tree. 

\begin{definition}
\label{def-second-mst-definition}
    Consider a point cloud $X$, the labelling of its points $\psi: X \rightarrow \{1,...,n\}$
    and the Vietoris-Rips simplex-wise filtration of $\Vr_{\bullet}(X)$ defined by the ordering $\phi$ given in Definition~\ref{def-VR-total-order}. 
    Then we construct a graph, $G$, with vertex set $X$ and edges chosen as all 1-simplices that correspond to the death of a degree-0 homology class in $\Vr_{\bullet}(X)$. 
\end{definition}

We now show that $G$ is a minimum spanning tree for $X$. 

\begin{lemma}
    \label{second-def-is-actually-a-mst}
    The construction of $G$ in Definition \ref{def-second-mst-definition} is indeed a minimum spanning tree for $X$. 
\end{lemma}

\begin{proof}
    Let $X$ be a point cloud and consider the Vietoris-Rips filtration $\Vr_{\bullet}(X)$. Adding a $1$-simplex $\sigma$ that kills a degree-$0$ homology class means that it joins two distinct connected components that are now merged. 
    By Lemma \ref{can-only-birth-or-kill} this means a $1$-cycle is not formed and by Kruskal's algorithm \cite{kruskal1956shortest} for building a minimum spanning tree, $\sigma$ would be added to the $\mst(X)$. Thus the $\mst(X)$ described in Definition \ref{def-second-mst-definition} is just the minimum spanning tree that would be constructed from Kruskal's algorithm, and thus is indeed a minimum spanning tree. 
\end{proof}

From now on we refer to the graph $G$ as `the' minimum spanning tree $\mst(X)$.
We know the edges of $\mst (X)$ correspond to the deaths of degree-0 homology classes  and we also know that every 1-simplex \emph{not} in $\mst (X)$ corresponds to the birth of a degree-1 homology class from Lemma \ref{can-only-birth-or-kill}. We will later show in Lemma \ref{RNG-lemma-3} that most of these degree-$1$ homology classes have trivial persistence.

\subsection{The Relative Neighborhood Graph}
\label{section-rng}

Relative neighborhood graphs were invented by \cite{ToussaintRNG}, who  proposed them as a graph that could be constructed from a point-set which would match the human perception of the shape of the point-set. 

\begin{definition}[Relative neighborhood graph]
\label{def-rng}
Consider a point cloud $X \in \mathbb{R}^D$. The Relative Neighborhood Graph of $X$, denoted $\rng (X)$ is the graph with vertices $X$ and edges $E =\{\langle yz \rangle \;|\; \lune (\langle yz \rangle ) = \emptyset\}$.

\end{definition}

The following lemma is an extension of Theorem 1 and 2 from \cite{ToussaintRNG} which only apply to point clouds with unique pairwise distances in the plane. 
In \cite{jaromczyk1992relative} they claimed Lemma \ref{lemma-mst-in-rng} holds for any Euclidean point cloud, even those with non-unique pairwise distances. However, the definition of $\rng (X)$ in \cite{jaromczyk1992relative} differs from the definition we give as it  uses Definition \ref{definition-traditional-lune} for the lune but replaces $<$ with $\leq$. Here, we give a proof using our modified definition of the relative neighborhood graph.

\begin{lemma}
\label{lemma-mst-in-rng}
Consider a point cloud $X \subset \R^D$, then 

\begin{equation}
\mst (X) \subset \rng (X) 
\end{equation}

\end{lemma}

\begin{proof}

We show $ \mst (X) \subset \rng (X)$ by showing the contrapositive. 
Suppose the edge $ \langle yz \rangle \notin \rng (X)$.
This means there is a point $x\in X$ with  $x \in \lune (\langle yz \rangle )$.
From our definition of $\lune (\langle yz \rangle)$ we have that $\langle xz \rangle, \langle xy \rangle < \langle yz \rangle$ in the filtration ordering. 
This means that $y$ and $z$ are part of the same connected component before $ \langle yz \rangle$ is added to the filtration and  it follows that $\langle yz \rangle$ cannot be part of $\mst (X)$.  
\end{proof}

Note that both the $\mst (X)$ and $\rng (X)$ can be defined when $X$ is any finite metric space, and that the inclusion $\mst (X) \subset \rng (X)$ still holds in this setting. 

The next lemma establishes the importance of edges which are in $\rng (X)$ but are not in $\mst (X)$. 

\begin{lemma}
\label{RNG-lemma-3}
Consider a point-cloud $X$. Then $\rng (X)$  contains all $1$-simplices that give birth to a degree-1 homology class corresponding to a persistence pair which is \emph{not} an apparent pair. Furthermore there is a one-to-one correspondence between the $1$-simplices in $\rng (X) \setminus \mst (X)$ and the homology classes that correspond to a persistence pair which is \emph{not} an apparent pair.  
\end{lemma}

\begin{proof}
Since $\mst (X)$ consists of all the 1-simplices that correspond to the death of a degree-0 homology class it follows that the $1$-simplices of $\rng (X) \setminus \mst (X)$ correspond to births. Thus we  need to show that these $1$-simplices cannot be part of an apparent pair. Consider a $1$-simplex $\langle yz \rangle \in \rng (X) \setminus \mst (X)$. Suppose for contradiction that $\langle yz \rangle$ was part of an apparent pair $(\langle yz \rangle, \sigma)$. From Definition \ref{apparent-pair} we know that $\sigma$ must take the form $\langle xyz \rangle$ since $\langle yz \rangle$ needs to be the face of $\sigma$ that appears latest in the filtration. From the definition of the simplex-wise Vietoris-Rips filtration this means that $\langle xy \rangle, \langle xz \rangle < \langle yz \rangle$. Thus $x \in \lune (\langle yz \rangle)$ which contradicts the fact that $\langle yz \rangle \in \rng (X)$. 

We now show that all $1$-simplices not in $\rng(X)$ must be part of an apparent pair. Consider $\langle yz \rangle \notin \rng(X)$. Then let $x \in \lune (\langle yz \rangle)$ such that $\langle x \rangle$ appears in the filtration before all other $\langle x' \rangle, x'\in \lune(\langle yz \rangle)$. Then $(\langle yz \rangle, \langle yzx \rangle)$ is an apparent pair.  
\end{proof}

Note that a related result appears in \cite{smith2024generic} and is stated in terms of short, medium and long edges of a filtration; Lemma 
\ref{RNG-lemma-3} was established independently of their work.

Lemma \ref{RNG-lemma-3} combined with Lemma \ref{unique-distances-lemma} means that we know the birth values for degree-1 Vietoris-Rips persistence intervals of non-apparent pairs simply by constructing the $\rng (X)$ and $\mst (X)$. The birth values of the degree-1 Vietoris-Rips persistence intervals of the non apparent pairs will be the diameters of the $1$-simplices in $\rng(X)\setminus \mst(X)$.  In the case of unique pairwise distances, this means we have the birth values for all intervals of positive length.  In the case of non-unique pairwise distances, there can be non-apparent pairs of simplices with the same birth and death values. 
In either case, knowing the number of non-apparent birth values gives us a stopping criterion when building the Reduced Vietoris-Rips complex: once the correct number of death simplices is found, no further edges need to be analysed when computing degree-1 persistent homology. 

\subsection{$kd$-Trees}

In this section we briefly recall the notion of a $kd$-tree data structure and the reader who desires more details is directed towards \cite{toth2017handbook}. 
The use of $kd$-trees is restricted to point sets in Euclidean space and is one reason we have introduced this assumption for $X$ in this paper.  

Given a point cloud $X \subset \R^D$, a $kd$-tree is a data structure that facilitates efficient geometric searching within $X$. A $kd$-tree of $X$ is a balanced binary tree with nodes formed by the points of $X$. Each node of the $kd$-tree partitions its descendants according to their coordinate values along a given axis, changing the ordinate used as we progress through layers of the tree. $kd$-trees are particularly useful for answering ``nearest neighbor queries" and ``radius search queries", as required when testing whether there are points in the lune of an edge.   
Recall that a nearest neighbor query is one where a point $z \in \R^D$ is given and the point from $X$ closest to $z$ is returned. A radius search is when a point $z$ and a search radius $R$ are given and all points $x\in X$ such that $d(x,z) < R$ are returned.

\subsection{Algorithms and complexity analysis }
\label{algorithms-and-complexity-analysis}

Let us examine the algorithms for $\rng (X)$ and their complexities. Supowit \cite{supowit_RNG} presents an algorithm which finds $\rng (X)$ in $O\left(n \log (n)\right)$ time if $X\subset \mathbb{R}^2$ by first computing $\mathrm{DT} (X)$.  Supowit also presents another algorithm which finds $\rng (X)$ in $O(n^2)$ time if $X \subset \mathbb{R}^D$ for $D \geq 3$. For the latter algorithm, there is an implicit dependence on the dimension $D$. Let $c(D)$ be the minimum number of radius-$\frac{1}{2}$ balls needed to cover the unit sphere $S^{D-1}$. Then  the complexity of Supowit's second algorithm is actually $O(c(D)n^2)$. A lower bound $c(D) > D^{\frac{3}{2}} \cdot (2)^{D-1}$ was given in \cite{verger2005covering} and thus, despite its quadratic dependence on $n$, the complexity is at least exponential  in the dimension  $D$. In \cite{Agarwal1992RelativeNG}, a randomised algorithm is given which computes the $\rng (X)$ in expected time $O(n^{2(1-\frac{1}{D+1})+\epsilon})$ where $\epsilon > 0$ is arbitrarily small. The algorithm we utilise in our implementation is given in \cite{r10_rng} and has a complexity of $O(n^2)$ and performs better than Supowit's algorithm in practice. 

One can always find $\mst (X)$ in $O(n^2 \log (n))$ time since there are $O(n^2)$ edges in the complete graph with vertex set $X$.  

Lastly, we discuss the algorithms for $kd$-trees and their complexities. For a point-cloud $X \subset \mathbb{R}^D$ the complexity of building the $kd$-tree is $O(n\log (n))$ \cite{cormen2022introduction}. Searching for the $k$ nearest neighbors of a given query point has an average complexity $O(k\log (n))$ and worse case complexity $O(kn)$.

\section{An algorithm for computing degree-1 Vietoris-Rips persistent homology}
\label{Algorithm}

In this section we discuss our algorithm, and its implementation $\mathsf{EuclideanPH1}$ to compute $\ph{1}(X)$ for a point cloud $X$ in Euclidean space, using the Reduced Vietoris-Rips complex. We will first briefly describe some differences between $\mathsf{EuclideanPH1}$ and $\mathsf{Ripser}$. 
Then we will give a high level description of the algorithm underpinning $\mathsf{EuclideanPH1}$, highlighting the main components and then we  give a detailed explanation of the implementation used for each part of the algorithm. 

There are a few main differences between $\mathsf{EuclideanPH1}$ and $\mathsf{Ripser}$ \cite{Ripser}. One of the biggest differences is that $\mathsf{EuclideanPH1}$ does not compute cohomology, instead computing homology directly. Another difference is $\mathsf{Ripser}$ will analyse all $1$-simplices up to the radius of the bounding sphere of the point cloud (or up to a user chosen maximum parameter value), even if there are no more non-apparent pairs to be found. In the case of $\mathsf{EuclideanPH1}$, it stops after finding the last persistent pair corresponding to a non-apparent pair. There is also a difference in the number of simplices utilized in each Algorithm. $\mathsf{Ripser}$ utilizes an indexing system which assigns a natural number to each simplex in the filtration up to the maximum diameter. With larger point clouds, the largest such index can be $O(n^3)$, and this can cause issues since it can lead to numbers larger than the machine can deal with. In the case of $\mathsf{EuclideanPH1}$, we also assign an index to simplices, but since we only use $O(n^2)$ simplices, this is less of an issue for $\mathsf{EuclideanPH1}$. 

\subsection{Outline of $\mathsf{EuclideanPH1}$}
\label{high-level-description-of-the-algorithm}

$\mathsf{EuclideanPH1}$ proceeds by analysing each 1-simplex in order of increasing length with ties broken using Definition~\ref{binary-relation-on-simplices}. For each 1-simplex, we find its lune and select a point to represent each connected component, in effect constructing the lune function; these are used to build the Reduced Vietoris-Rips complex. The matrix reduction and persistence pairing is performed as each 2-simplex is added.  The iteration ceases when we have found the pre-determined number of homology classes that do not correspond to apparent pairs.

\subsubsection*{Initialise data structures}

The input is a list of $n$ points, $X \subset \mathbb{R}^D$, and an integer $k$. 
The value of $k$ can be set by the user, but defaults to $\sqrt{n}$. 
Guidelines and considerations on what constitutes an ideal value of $k$ are discussed further in Section \ref{implementation-details}.

The first procedure is building the $kd$-tree for $X$. 
We then use this to initialise two further data structures that assist in enumerating 1-simplices by increasing length. These are a table, $N$, listing up to $k$ nearest neighbors for each point in $X$ and a minimum heap, $H$, for efficient sorting of 1-simplices. Details about minimum heap data structures are given in~\cite{williamsalgorithm}.

\subsubsection*{Build the RNG}

We next compute $\rng(X)$ of the point-cloud $X$ to establish the stopping condition for the main loop. 
The quantity $\mathsf{total\_bars}$ is calculated as the number of edges in $\rng(X)$ minus $(n-1)$, the number of edges that must be in any $\mst(X)$; this 
tells us the number of persistence intervals that do not correspond to an apparent pair. 

Initialise $\mathsf{death\_count} = 0$ to count the number of non-apparent pairs found while iterating the main loop. 

\subsubsection*{Start of Main Loop}

\subsubsection*{Step 1: Find the next 1-simplex}

Find the next 1-simplex $\langle yz \rangle$ in the simplex-wise filtration ordering. This is facilitated by the minimum heap $H$ which contains the next candidate 1-simplex for each point, and stores these sorted by the total order of Definition \ref{binary-relation-on-simplices}. 

\subsubsection*{Step 2: Find the lune and compute its connected components}

Find the lune of the 1-simplex $\langle yz \rangle$ from Step 1. 
If the lune is empty return to Step 1, otherwise continue. 
Now find $c$, the number of connected components of $\lune (\langle yz \rangle)$ and define the lune function $L$. 
That is, for $i=1,...,c$, choose a point $x_{i}$ from each connected component and set $L(\langle yz \rangle) = \{x_1,...,x_c\}$. 

\subsubsection*{Step 3: Construct a ``column'' for each connected component}

If $\mathrm{lune}(\langle yz \rangle)$ has $c=1$ then we add one column to the boundary matrix corresponding to the boundary of $\langle yzx_{1} \rangle$. This column does not need to be reduced as its lowest~1  corresponds to the row-index of $\langle yz \rangle$ which has only just been added to the filtration. 
Return to Step 1. 

Otherwise, if $\mathrm{lune}(\langle yz \rangle)$ has $c\geq 2$ then we need to add one column for each 2-simplex $\langle yzx_i  \rangle$, in increasing lexicographic order.  With the exception of the first column added, all other columns need to be reduced.   Proceed to Step 4. 

As will be elaborated in Section~\ref{implementation-details}, we don't actually construct a matrix but store columns that require reduction in a hash table.

\subsubsection*{Step 4: Perform matrix reduction and persistence pairing}

In the case that $c\geq 2$, each of $c-1$ matrix columns must be reduced. 
If a column reduces to zero, that 2-simplex has given birth to a degree-2 homology class so $\mathsf{death\_count}$ remains the same.
Otherwise, if a lowest~1 remains then its location in the matrix at row-$i$ and column-$j$ defines a non-apparent persistence pair and we increment $\mathsf{death\_count}$ by 1. 

The barcode interval end points are found as the diameters of the appropriate birth and death simplices (i.e., $\sigma_i$ and $\sigma_j$). If the persistence of this interval is 0 then it is not added to the persistence barcode. 

\subsubsection*{Iteration test}

If $\mathsf{death\_count} = \mathsf{total\_ bars}$ then end the iteration, otherwise proceed to Step 1 and execute the main loop again.  
 
\subsubsection*{Output}

Once all non-apparent pairs are found, the list of barcode intervals with non-trivial persistence is returned.  

\subsection{Implementation details}
\label{implementation-details}

In this section we document our implementation of the Reduced Vietoris-Rips algorithm in the $C{++}$ code, $\mathsf{EuclideanPH1}$ \cite{EuclideanPH1} which accepts point clouds $X \in \mathbb{R}^D$.
We also briefly discuss the computational time complexity of each step. In what follows, $B(x,r)$ denotes the open ball centred at $x$ with radius $r$. For a set $A$, $\overline{A}$ denotes the closure of $A$.

\paragraph{Initialize Data Structures}
\label{initialize-data-structures}
We first construct a $kd$-tree on the points of $X$; this has a time complexity of $O(n\log (n))$. 
In $\mathsf{EuclideanPH1}$ we use the package nanoflann~\cite{blanco2014nanoflann}. 
We then construct a table $N$ with row-$i$ containing nearest neighbors of $\psi^{-1}(i) = x \in X$, recall $\psi$ labelled each $x\in X$ with an index value and was mentioned in Definition \ref{extention-of-psi}. 
To populate row-$i$, we start by listing the $k$ closest neighbors to $x$ and then  
remove those points $y$ having $\psi(y) < \psi(x)$. 
The remaining points in this row are then sorted by increasing distance from $x$, with ties ordered by their index. 
Note that the last row in this table is empty.
The parameter $k$ has a default value of $\sqrt{n}$, but its optimisation is left to the user; we discuss some heuristics below.

Building the table $N$ has time complexity  $O(kn\log (n))$. If one chooses $k = n$, then the construction of $N$ will have complexity bounded by $O(n^2\log (n))$. 

To facilitate the execution of the main loop we need a method that returns the next $1$-simplex in the filtration ordering.
Enumerating and sorting all $1$-simplices by diameter has a complexity of $O(n^2 \log (n))$, which is rather costly. 
Instead, we use the table $N$ and a minimum heap data structure, $H$. Elements of the heap $H$ are tuples $(\psi(y),\psi(z),t,r)$ with $\psi(y) < \psi(z)$. The variable $t$ records the location of $z$ in the $\psi(y)$-row of the table $N$, i.e., $N[\psi(y)][t] = \psi(z)$, and $r = \diam(\langle yz \rangle) = d(y,z)$. The heap is initialised with $n-1$ elements 
$(i, N[i][1] , 1 , r)$, where $r$ is the distance between point $\psi^{-1}(i)$ and $\psi^{-1}(N[i][1])$, the nearest neighbor with larger index. Constructing the initialised minimum heap has complexity $O(n\log (n))$. 

During the main loop we remove the element with smallest value of $r$ (and point indices) and then insert the next candidate 1-simplex from $N$. Inserting a new element into the minimum heap has $O(\log (n))$ complexity. The point of creating and maintaining the table $N$ and the heap $H$ is that they  \emph{may} allow us to avoid computing all pairwise distances. 

This is an opportune point to discuss the choice of $k$. 
Recall that the table $N$ has at most $k$ points listed in each row. 
Suppose that at some point in the iteration of the main loop, we pop the element $(\psi(y),\psi(z),t,r)$ from the heap and find that there are no further neighbors of $y$ stored in $N$. 
We must now supplement the table by finding and sorting further pairwise distances from $y$. 
In our current implementation, if we reach this situation for a given $y$, we compute and sort all additional pairwise distances $d(y, z)$ with $\psi(y) < \psi(z)$. 
If we have to do this for every row, we incur a computational complexity of $O(n^2 \log (n))$ and a memory cost of $O(n^2)$, which we would like to avoid. Let $(\tau, \sigma)$ be the non-apparent persistent pair such that $\sigma$ appears latest in the filtration. Let $\eta $ be the face of $\sigma$ that appears latest in the filtration. Thus one wants to choose $k$ large enough so that $N$ includes both $0$-dimensional faces of the $1$-simplex $\eta$, with one face being in the neighbor list of the other. 
 
Without prior knowledge of the structure of $X$, we do not have any way to determine an optimal value of $k$ except by trial and error. 
In practice, if a given run of the code seems to be using too much memory, it could be that the chosen value of $k$ is too small, causing many extra extensions of the table $N$. 
Conversely, one does not want to choose $k$ \emph{too} large, otherwise $N$ will use more storage than actually required to complete the calculation of $\ph{1}(X)$, hence the implementation uses  the choice of $k = \sqrt{n}$ in the absence of a user chosen value of $k$.

\paragraph{Build the RNG}
The algorithm in \cite{r10_rng} will construct $\rng(X)$ in $O(n^2)$ time. In lower dimensions, one could utilize the Delaunay triangulation in order to obtain $\rng(X)$ faster, however this is not the focus of this paper, since in low dimensional Euclidean space, one can make use of Alpha-shape persistent homology.

\paragraph{Start of Main Loop}
Here we discuss the implementation details of the main loop. 

\paragraph{Step 1: Find the next 1 simplex}
At the start of each iteration, we pop the top element of the minimum heap $H$, $(\psi (y), \psi (z), t, r)$. This element has the smallest value of $r$, with ties broken by lexicographic order of the pair ($\psi(y)$, $\psi(z)$).  
We then insert $(\psi(y), \psi(z'), t+1, r')$ where $z'$ is the $(t+1)$th neighbor of $y$ in the table $N$ and $r' = d(y,z') \geq r$. 
As described above, if $(t+1)$ is larger than the number of neighbors currently stored in $N$, then we extend the row $N[\psi(y)]$ by finding and sorting all points $x$ with $\psi(x) > \psi(y)$. 
Once we have exhausted all possible pairs with $\psi(x) > \psi(y)$ nothing is inserted.

\paragraph{Step 2: Find the lune and compute its connected components}
Just as when building $\rng(X)$, we find $\lune (\langle y z \rangle)$ by using the $kd$-tree to return the points of $X$ in $\overline{B}(y,d(y,z)) \cap \overline{B}(z,d(y,z))$.

We then find the number of connected components of $\lune (\langle yz \rangle)$.
To do this, we perform union-find on the points of $\lune(\langle yz \rangle)$. This will require us to do up to $O(n^2)$ unions on up to $O(n)$ points meaning this step will have time complexity bounded by $O(n^2 \alpha (n))$. This worst case complexity is rather high and applies to \emph{each} $1$-simplex that is analysed. 
Thus it is worth some effort to speed up the analysis of lunes, especially since the overwhelming majority of $1$-simplices will have a lune with one connected component. To speed things up we conduct some preliminary tests based on Lemma~\ref{lens-lemma} that will guarantee that there is only one connected component in the lune.

For a given 1-simplex $\langle yz \rangle$ the first test we do is to check the ball of radius $\frac{1}{2}(2-\sqrt{3})r$ centred at the midpoint $(y+z)/2$. 
This is the ball circumscribed by the boundary of the region  $\{x\in \mathbb{R}^D \;|\; \Angle (yxz) > \frac{5\pi}{6} \}$.
Points of $X$ in this region are in $\lens(\langle yz \rangle)$ (Definition~\ref{lens}), and  Lemma~\ref{lens-lemma} guarantees that $\lune(\langle yz \rangle)$ has a single connected component if $\lens(\langle yz \rangle) \neq \emptyset$. Listing points of $X$ that are in this ball is considerably faster than finding all points of $\lune(\langle yz \rangle)$, and can provide a significant short-cut.  
To define the lune function in this case, we simply take one of the points found within the ball, $x$ say, and form the 2-simplex $\langle xyz \rangle$. 

The next test we implement requires us to find all points within the lune. We then test each point $x$ within the lune to see if $\Angle (yxz) > \frac{5\pi}{6}$. If we find a point $x$ satisfying this condition then we know that $\lens(\langle yz \rangle) \neq \emptyset$ and again deduce that  $\lune (\langle yz \rangle)$ has one connected component. Therefore we stop at the first such point found and form the 2-simplex $\langle xyz \rangle$.

If both of these lens tests come back negative, then we simply perform union-find on the points in the lune as described above. 
Neither of the lens tests will increase the worst case complexity of $O(n^2\alpha (n))$ as both tests have worst case complexity of $O(n)$.

\paragraph{Step 3: Construct a ``column'' for each connected component}
Here we discuss $\mathsf{EuclideanPH1}$'s implementation of the boundary matrix for the simplex-wise filtration 
and explain why the word ``columns'' appears in quotation marks. 

If we \emph{were} building a matrix, then each time 
a $2$-simplex $\langle xyz \rangle$ is added to the filtration we would need to create a column vector with non-zero entries at rows $\phi(\langle xy \rangle), \phi (\langle yz \rangle)$ and $\phi (\langle xz \rangle)$. Recall $\phi$ was defined in Definition \ref{def-VR-total-order}.
The subsequent matrix reduction step searches the previous columns to determine the correct persistence pairing. 
Since this is the only operation required of the boundary matrix, our implementation makes use of hash table $T$ to improve the efficiency of the matrix reduction process.   
We use hash tables \cite{10.1145/356643.356645} via the standard $\mathsf{unordered \_ map}$ structure in $C{++}$.

So, a ``column'' of the boundary matrix is instead represented by a (key, value) pair of $T$. The entries of the value field and the key field depend on how many connected components $\lune (\langle yz \rangle)$ has, this will be explained in more detail in Step 4.
Ultimately, at the end of the algorithm $T$ will consist of the same information as the reduced boundary matrix, albeit in a slightly different form.  

\paragraph{Step 4: Perform matrix reduction and persistence pairing}

We now discuss what we are actually adding to $T$. If $\lune (\langle yz \rangle)$ is empty, then no (key, value) pair is added to $T$. Suppose $\lune (\langle yz \rangle)$ has only one connected component and let $L(\langle yz \rangle)$, the lune function of $\langle yz \rangle$, be $\{x \}$. We then add a (key, value) pair to $T$ whose value field is $\{ \phi(\langle xy \rangle), \phi (\langle yz \rangle), \phi (\langle xz \rangle) \}$ and set its key field to  the highest index of the three 1-simplices. This must be $\phi (\langle yz \rangle)$ because $\langle xy \rangle, \langle xz \rangle < \langle yz \rangle$. 
In other words, the value field stores only the non-zero row-indices and the key field stores the row index of the lowest 1. 
In this case of  $\lune (\langle yz \rangle)$ having one component, we do not need to perform any matrix reduction as no previous columns can have a lowest $1$ in the $\phi(\langle yz \rangle)$-th row.

Now suppose that $\lune (\langle yz \rangle)$ has more than one connected component, that is to say $|L(\langle yz \rangle)| \geq 2$. 
For each element, $x \in L(\langle yz \rangle)$ we add a node to $T$ as follows. 
First, we create a (key, value) pair $g_{x}$ which sets the column field as $\{\phi(\langle xy \rangle, \phi (\langle xz \rangle), \phi(\langle yz \rangle) \}$ and the key field as $\phi(\langle yz \rangle)$. If $g_{x}$ is the first such (key,value) pair created then we add it to $T$ without needing to perform any reduction, as no previous columns can have a lowest $1$ in the $\phi(\langle yz \rangle))$-th row. 
Before we add this (key, value) pair to $T$ we perform the matrix reduction procedure as follows. 
We search $T$ to see if there are any columns with the same lowest one (key value) as $g_x$. 
If $\langle yzx \rangle$ is the $j$th two-simplex added to the filtration, then since we are using a hash table this takes worst case $O(j)$ time, though in practice it will take $O(1)$ time.  

If we find a (key, value) pair with the same key (lowest one) then we ``add'' the value field of this (key, value) pair to the value field of $g_x$. 
Since we are working with coefficients in $\mathbb{Z}_{2}$, ``adding'' in this context is  simply taking the symmetric difference of the two sets of indices. 
We then set the key field of $g_x$ to be the maximum element of the reduced value field. 
We repeat this process until either the key of $g_x$ cannot be found in $T$ or the value field of $g_x$ is empty. 
The first case is equivalent to reducing the column in the standard matrix reduction algorithm until the column has lowest 1 different from all the columns reduced thus far; we add this reduced $g_x$ to $T$. 
In the case that the column field is empty, this means a degree-2 homology class was created by adding the 2-simplex $\langle yzx \rangle$ to the filtration; we do not add anything to $T$ in this case. We will discuss the complexities of constructing and searching $T$ in Section \ref{complexity-analysis}.

\subsection{Overall complexity analysis}
\label{complexity-analysis}

In this section we find an upper bound on the complexity of the algorithm described in sections \ref{high-level-description-of-the-algorithm} and \ref{implementation-details}. 

\paragraph{Worst time complexity bound for  Euclidean point clouds.} We now bound the overall worst case time complexity for the algorithm in the case $X \subset \mathbb{R}^D$ using the discussion in Section \ref{implementation-details}. 
\begin{prop}
\label{overall-time-complexity}
    For a finite point cloud $X\subset \R^D$ with $n$ points the overall time complexity of \textsf{EuclideanPH1} can be bounded above by $O(n^6)$.
\end{prop}

\begin{proof}

First we find the complexity of constructing $N$ and $\rng(X)$. The time complexity of constructing $N$ is bounded above by $O(n^2\log (n))$ and the time complexity of constructing $\rng(X)$ is bounded above by $O(n^2)$. Thus the ``Initialise data structures" and ``Build the RNG" parts of Section \ref{implementation-details} is bounded by $O(n^2\log (n))$. 

Now let us analyse the complexity of the main loop. 
Consider the $j$-th 1-simplex $\langle yz \rangle$ analysed in the algorithm. Popping the $j$-th element from the heap $H$ (and then restoring the heap property) can be done in $O(\log (j))$ time and inserting a new element into the heap  takes $O(\log(j))$ time. Since $j$ is bounded above by $O(n^2)$ it follows this step is bounded by $O(\log (n))$. 

Finding $\lune (\langle yz \rangle)$ has complexity bounded by $O(n)$ and finding the connected components of the lune has complexity $O(n^2 \alpha(n))$, where $\alpha(n)$ is the inverse Ackerman function. Thus Step 2 of Section \ref{implementation-details} has complexity bounded by $O(n^2 \alpha (n))$. 

The complexity of Steps 3 and 4 depend on the number of connected components of $\lune(\langle yz \rangle)$.
If $\lune (\langle yz \rangle)$ is empty then there is nothing more for us to do and we move onto the next iteration of the main loop. 

If $\lune (\langle yz \rangle)$ has one connected component and $L(\langle yz \rangle) = \{x\}$ then we only need to add a (key, value) pair to $T$, the hash table \cite{10.1145/356643.356645} introduced in Step 4 of Section \ref{implementation-details}. The worst case complexity of searching for an element in a self-balancing binary search tree, \cite{guibas1978dichromatic} is lower than that of a hash table. Therefore we will take $T$ to be a self-balancing binary search tree in this analysis. This is despite hash tables being actually utilised in the implementation, since their average time complexity is $O(1)$. 

If $T$ is a self balancing binary search tree, then inserting an element into $T$ has worst case complexity $O((\log(j)))$. Since $j$ is bounded by $O(n^2)$ this step has worst case complexity bounded by $O(\log (n))$ for all $j$. 

If $\lune(\langle yz \rangle)$ has more than one connected component then we need to search $T$ up to $j$ times. Each time we search and find another node with the same lowest 1 we add the value fields of the respective (key, value) pairs, which has complexity bounded by $O(j)$. This results in an overall complexity of $O(j(\log(j) + j))$. If $\lune (\langle yz \rangle)$ has $c$ connected components then we need to do this step $c-1$ times. However, for fixed $D$, $c$ is bounded by a constant by Lemma \ref{bounded-connected-components-proof}, so it follows that the overall complexity for inserting a node in this case is still $O(j(\log (j) + j))$. Since $j$ is no more than $O(n^2)$ it follows that we can bound the complexity of Steps 3 and 4 by $O(n^2(\log (n) + n^2)) = O(n^4)$. 

Regardless of how many connected components the lune has, the complexity of the main loop is bounded by $O(n^4)$. Since the main loop will be executed a maximum of $O(n^2)$ times it follows that the complexity of the whole algorithm can be bounded by $O(n^6)$.
\end{proof}

\section{Experiments: Comparing EuclideanPH1 with $\mathsf{Ripser}$}
\label{sec:experiments}

Here we present timing results for computing $\ph{1}(X)$ of various Euclidean point clouds processed using $\mathsf{Ripser}$~\cite{Ripser} and $\mathsf{EuclideanPH1}$.
All results in this section were obtained using a high-end personal computer with the following specifications. 

\begin{itemize}
            \item Intel Xeon(R) w7-2475X × 40

            \item 128GB DDR5 RAM

            \item OS: 64bit Ubuntu 22.04.5 LTS 
        \end{itemize}

In some experiments, for the machine used, $\mathsf{Ripser}$ was unable to compute $\ph{1}(X)$ for the point clouds listed. This was due to $\mathsf{Ripser}$ exhausting the memory of the machine. 

All runtimes were ``wall-clock" measurements and code was run with no other processes, except those necessary for the operating system to function. While running the code, the machine was not connected to the internet. In all runs of the code, the O3 optimisation was used. The specific version of $\mathsf{Ripser}$ used was 1.2.1 and the specific version of $\mathsf{EuclideanPH1}$ used was the one featured on the following github address \\ \url{https://github.com/musashi-7430251/EuclideanPH1/tree/main}.

\subsection{Solid torus in $\mathbb{R}^3$ embedded in $\mathbb{R}^{10}$}
\label{sec:solid-torus-in-r3-embedded-in-r10-ripser-experiment}
For this set of experiments we generated point-clouds of size $n=1000, 2000, 3000, ..., 25000$ in a solid torus of major radius $3$ and minor radius $1$ in $\mathbb{R}^3$. For each value of $n$ we generated five point-clouds. We then rotated these point-clouds in $\mathbb{R}^{10}$ using the Householder transformation, which preserves distances, with a randomly selected ten dimensional non-zero vector $v$. 

\begin{equation}
    Q(v) = I - \frac{2}{v^{T}v} vv^{T}
\end{equation}

This results in a $10$ dimensional point-cloud isometric to the original point-cloud where the last seven co-ordinates are not always zero, allowing us to test for ambient co-ordinate dimension effects. For these point-clouds, $\ph{1}(X)$ was computed using $\mathsf{Ripser}$ and $\mathsf{EuclideanPH1}$. The runtimes for these computations are shown in Figure \ref{solid-torus-in-r10}. In every run of $\mathsf{EuclideanPH1}$ $k$ was set to $100$. 

\begin{figure}[h!]

    \begin{subfigure}[t]{.45\textwidth}
    \includegraphics[scale = 0.45]{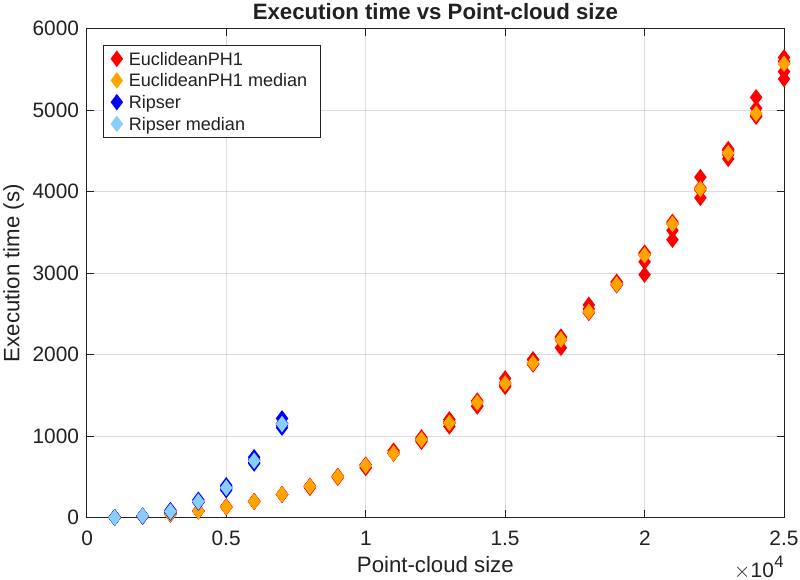}
    \caption{Execution time against point-cloud size for the solid torus embedded in $\mathbb{R}^{10}$. $\mathsf{Ripser}$ could not compute $\ph{1}(X)$ for point clouds with more than 7000 points as it exhausted the memory of the machine.}
    \label{solid-torus-in-r10}
    \end{subfigure}
    \hfill
    \begin{subfigure}[t]{.45\textwidth}
    \includegraphics[scale = 0.45]{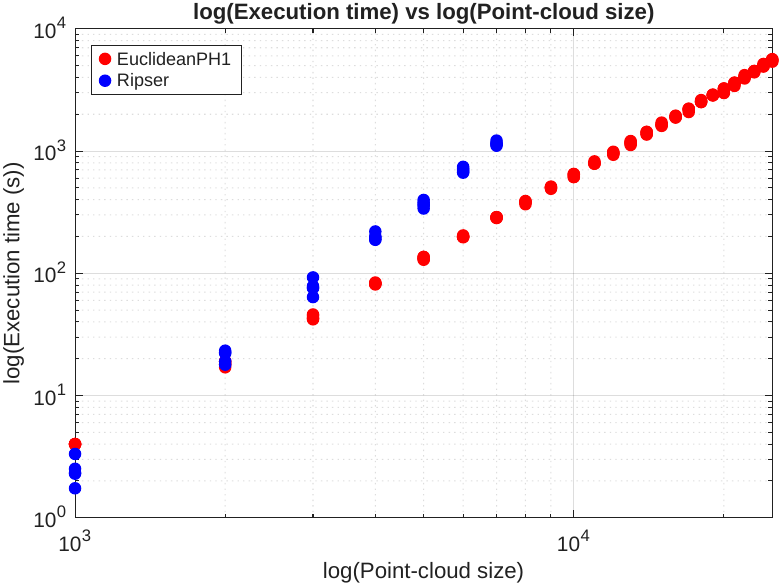}
    \caption{The log-log graph of execution time against point-cloud size for the solid torus data set. MATLAB gives the slope of the line of best fit through the $\mathsf{EuclideanPH1}$ points as $2.26$. MATLAB gives the slope of the line of best fit through the $\mathsf{Ripser}$ points as $3.17$.}
    \label{log-log-solid-torus}
\end{subfigure}
\caption{Runtimes and corresponding log-log plot for the solid tori embedded in $\mathbb{R}^{10}$. }
\label{solid-torus-runtimes-and-log}
\end{figure}

\subsection{Noisy hollow $2$-torus in $\mathbb{R}^{3}$ embedded in $\mathbb{R}^{10}$.}
\label{sec:noisy-hollow-torus-in-r3-embedded-in-r10-ripser-experiment}
For this set of experiments point-clouds of size $n= 1000, 2000, 3000, ..., 15000$ were generated on a hollow torus of major radius 3 and minor radius 1 and then noise was added. For each value of $n$ five point-clouds were generated. Like for the solid torus, these point-clouds were then rotated in $\mathbb{R}^{10}$ using the householder transformation with a randomly selected non-zero ten dimensional vector $v$. For these point-clouds, $\ph{1}(X)$ was computed using $\mathsf{Ripser}$ and $\mathsf{EuclideanPH1}$. The runtimes for these computations are shown in Figure \ref{hollow-torus-in-r10}. In every run of $\mathsf{EuclideanPH1}$ $k$ was set to $100$. 

\begin{figure}[h!]

    \begin{subfigure}[t]{.45\textwidth}
    \includegraphics[scale = 0.45]{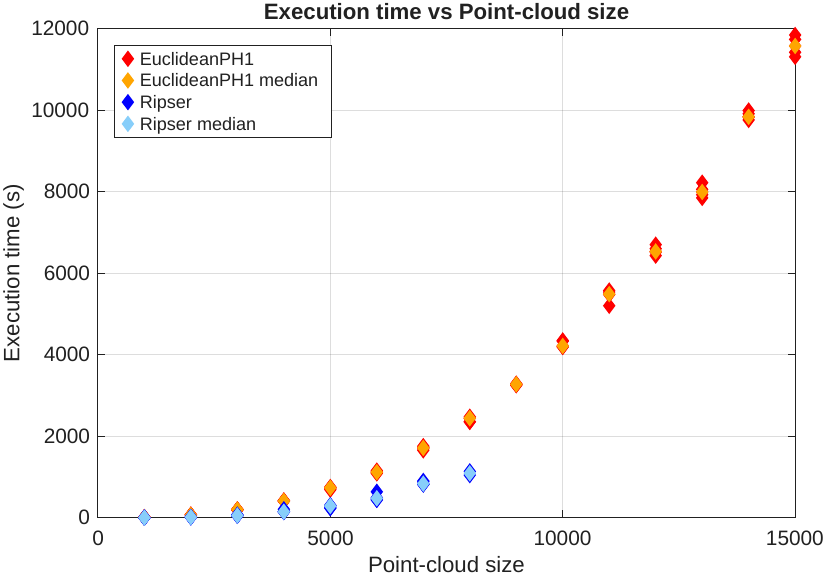}
    \caption{Execution time against point-cloud size for the hollow $2$-torus embedded in $\mathbb{R}^{10}$. At \\$n=8000$, $\mathsf{Ripser}$ completed 2 out of 5 runs.}
    \label{hollow-torus-in-r10}
    \end{subfigure}
    \hfill
    \begin{subfigure}[t]{.45\textwidth}
    \includegraphics[scale = 0.45]{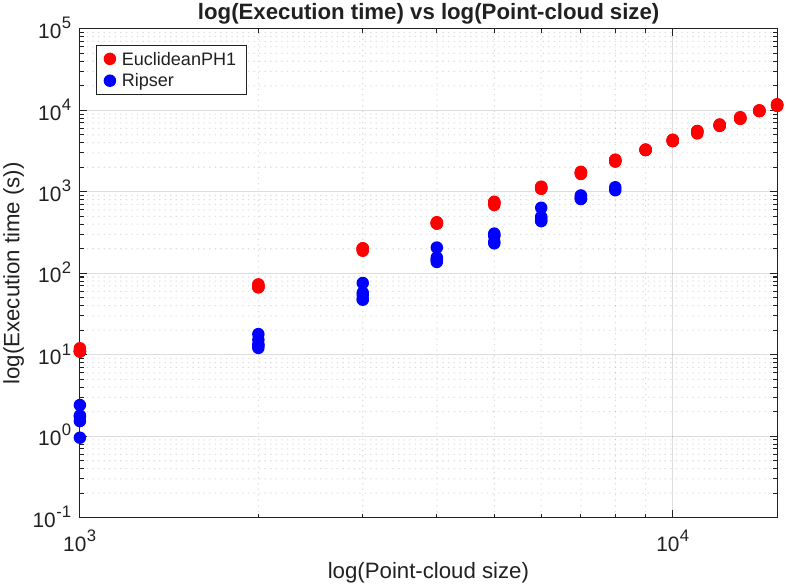}
    \caption{The log-log graph of execution time against point-cloud size for the hollow $2$-torus data set. MATLAB gives the slope of the line of best fit through the $\mathsf{EuclideanPH1}$ points as $2.56$. MATLAB gives the slope of the line of best fit through the $\mathsf{Ripser}$ points as $3.20$.}
    \label{log-log-hollow-torus}
\end{subfigure}
\caption{Runtimes and corresponding log-log plot for the hollow $2$-tori embedded in $\mathbb{R}^{10}$. For $n \geq 9000$, $\mathsf{Ripser}$ could not compute $\ph{1}(X)$ due to exhausting the memory of the machine.}
\label{hollow-torus-runtimes-and-log}
\end{figure}

\subsection{The noisy $5$-torus}
\label{sec:noisy-cursed-torus-experiments-ripser}
One might argue that the previous data sets are not really high-dimensional as the point-clouds are still intrinsically three-dimensional. 
Here we compare $\mathsf{Ripser}$ against $\mathsf{EuclideanPH1}$ on an intrinsically higher dimensional data set. For this set of experiments, point-clouds of size \\ $n = 1000, 2000, ..., 10000$ were generated by sampling from circles of radii $0.5, 1.0, 1.5, 2.0, 2.5$, with each circle using two of the the ten co-ordinates and then uniform random noise with  range $[0, 0.05]$ was added to each coordinate.

The noise was added by taking a value from the uniform distribution over $[0, 0.05]$ to each ordinate. For each value of $n$, five point-clouds were generated. In Figure \ref{cursed-torus-pictures}, a 4D projection of the data set is shown, with color representing the fourth dimension. 

\begin{figure}[h!]

    \begin{subfigure}[t]{.45\textwidth}
    \includegraphics[scale = 0.25]{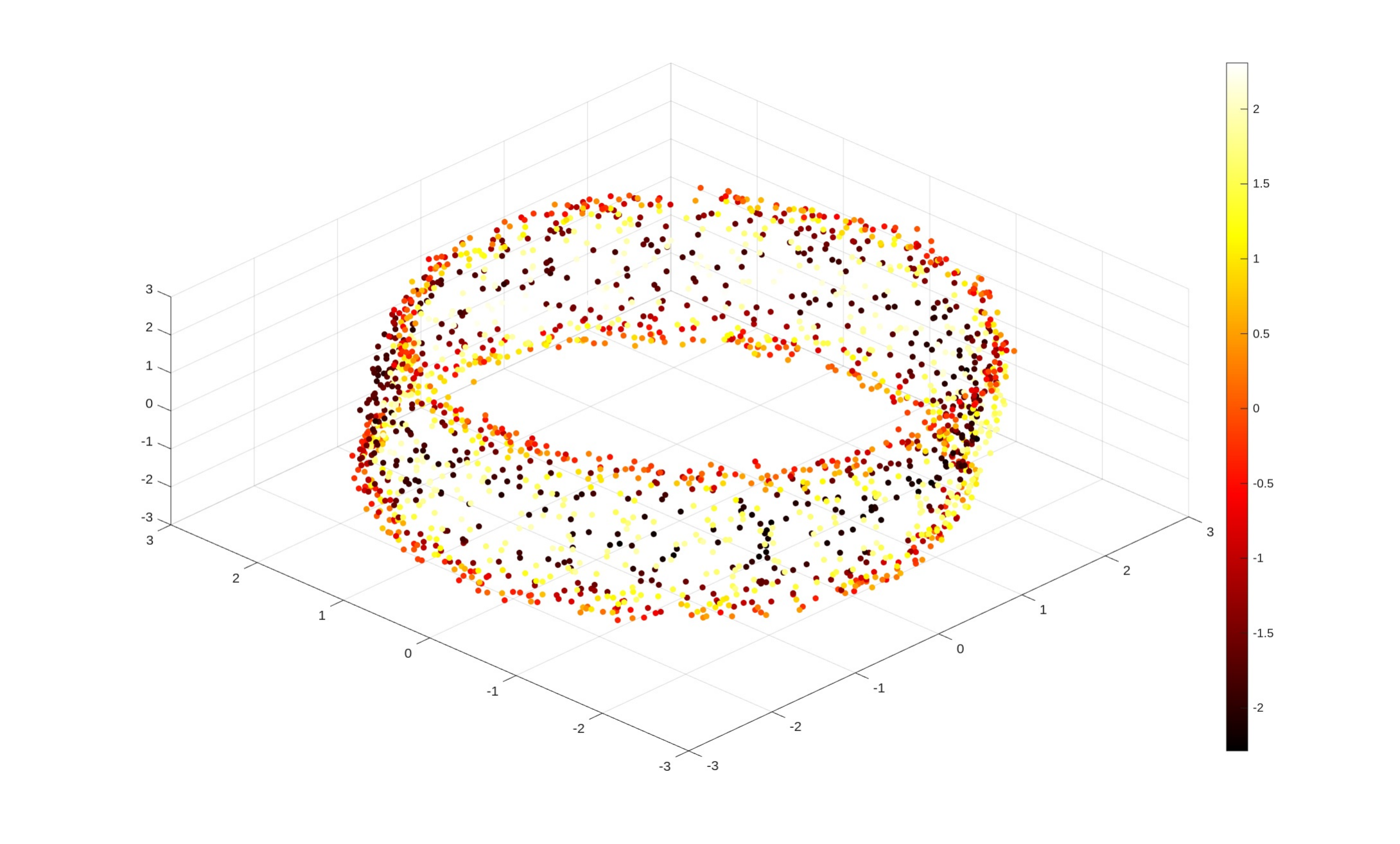}
    \label{cursed-torus-view-1}
    \end{subfigure}
    \hfill
    \begin{subfigure}[t]{.45\textwidth}
    \includegraphics[scale = 0.25]{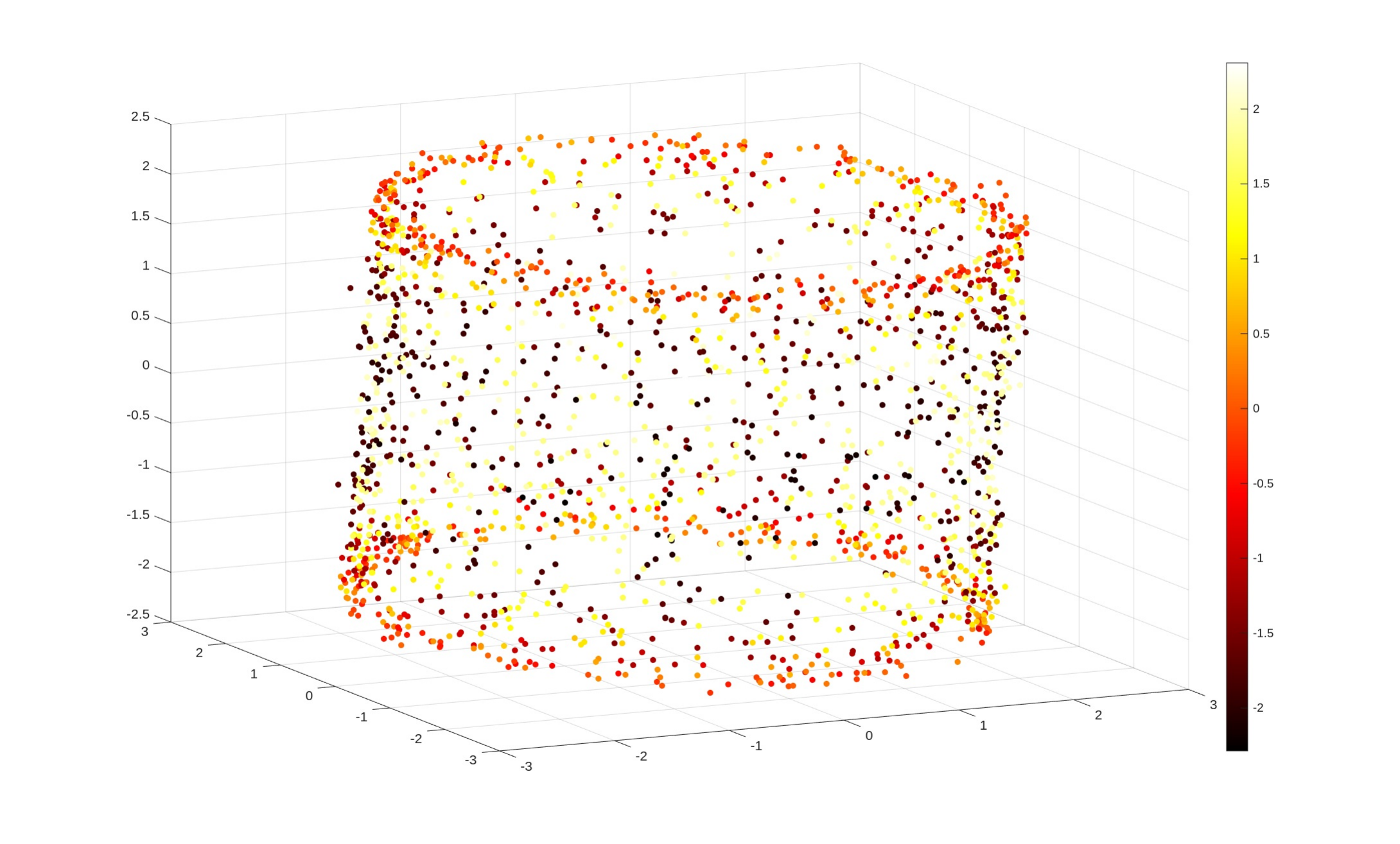}
    \label{cursed-torus-view-2}
\end{subfigure}
\caption{Two different viewing angles of a ``4D" projection of the $5$-torus. The fourth dimension is represented by colour.}
\label{cursed-torus-pictures}
\end{figure}

For these point-clouds, $\ph{1}(X)$ was computed using $\mathsf{Ripser}$ and $\mathsf{EuclideanPH1}$. The runtimes for these computations are shown in Figure \ref{cursed-torus-runtimes-and-log}. In every run of $\mathsf{EuclideanPH1}$ $k$ was set to 100. 

\begin{figure}[h!]

    \begin{subfigure}[t]{.45\textwidth}
    \includegraphics[scale = 0.45]{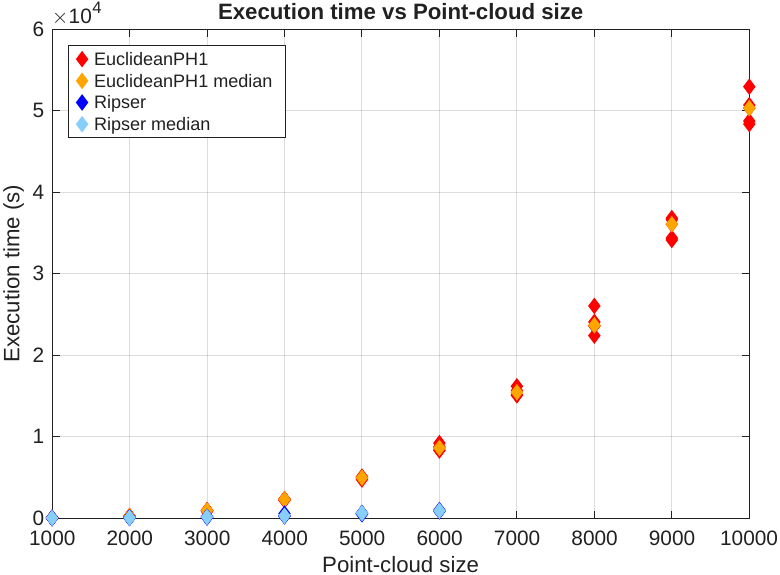}
    \caption{Execution time against point-cloud size for the noisy $5$-torus. At $n=6000$, $\mathsf{Ripser}$ completed 3 out of 5 runs.}
    \label{cursed-torus-in-r10}
    \end{subfigure}
    \hfill
    \begin{subfigure}[t]{.45\textwidth}
    \includegraphics[scale = 0.45]{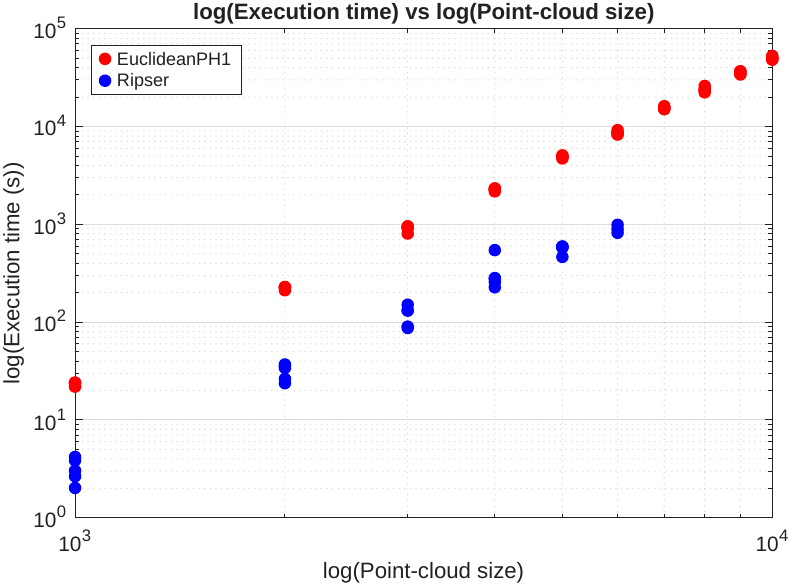}
    \caption{The log-log graph of execution time against point-cloud size for the noisy $5$-torus data set. MATLAB gives the slope of the line of best fit through the $\mathsf{EuclideanPH1}$ points as $3.34$. MATLAB gives the slope of the line of best fit through the $\mathsf{Ripser}$ points as $3.22$.}
    \label{log-log-cursed-torus}
\end{subfigure}
\caption{Runtimes and corresponding log-log plot for the $5$-torus. For $n \geq 7000$ $\mathsf{Ripser}$ could not compute $\ph{1}(X)$ as it exhausted the memory of the machine.}
\label{cursed-torus-runtimes-and-log}
\end{figure}

\section{Discussion: $\mathsf{EuclideanPH1}$ vs $\mathsf{Ripser}$.}
\label{Discussion:Ripser}
The results show that overall $\mathsf{EuclideanPH1}$  provides a practical advantage on larger
point-clouds by completing computations that $\mathsf{Ripser}$ could not complete within the available memory.
However when it comes to smaller point-clouds, $\mathsf{Ripser}$ maintained its speed advantage over $\mathsf{EuclideanPH1}$ in two out of three of the experiments. It should be said that none of the experiments using $\mathsf{EuclideanPH1}$ exhausted the memory of the machine, and it is was certainly possible to compute $\ph{1}(X)$ for larger point clouds, but we did not due to the large run times.  
These three experiments are discussed in more detail in the following sections. 

\subsection{Solid tori embedded in $\mathbb{R}^{10}$}

It is not that surprising that $\mathsf{EuclideanPH1}$ fared best against $\mathsf{Ripser}$ for the solid tori point-clouds. Even if these point-clouds are embedded in $\mathbb{R}^{10}$, the fact that the tori are solid means that the lens test described in Lemma \ref{lens-lemma} can be used more often. This means that the time taken to compute the number of connected components of a given lune is less, leading to a dramatic reduction in runtime. This is observed in the runtimes shown in Figure \ref{solid-torus-in-r10}, where one can see that even for smaller point-clouds, $\mathsf{EuclideanPH1}$ is faster than $\mathsf{Ripser}$. 

\subsection{Hollow noisy  $2$-tori embedded in $\mathbb{R}^{10}$}
In the case of the the hollow $2$-tori we see a less favourable performance from $\mathsf{EuclideanPH1}$. Whilst $\mathsf{EuclideanPH1}$ is able to compute $\ph{1}(X)$ for larger point-clouds than $\mathsf{Ripser}$, when $\mathsf{Ripser}$ is able to compute $\ph{1}(X)$ the data in Figure \ref{hollow-torus-in-r10} shows that $\mathsf{Ripser}$ is faster. Unlike the case with solid tori, the lens test in Lemma \ref{lens-lemma} will be utilised considerably less due to the fact that a larger number of lunes will have empty lens, a consequence of the fact that the $2$-torus is hollow. Thus it is not particularly surprising that we see the time performance of $\mathsf{EuclideanPH1}$ suffer in comparison to that of the solid tori. Despite this disadvantage we can see from the log-log graph in Figure \ref{log-log-hollow-torus} that $\mathsf{EuclideanPH1}$  has a better log-log slope than $\mathsf{Ripser}$ over the tested range, suggesting that even with unlimited memory one would potentially see more favourable runtimes from $\mathsf{EuclideanPH1}$ in the long run. 

\subsection{$5$-tori}
In this section we see the time performance of $\mathsf{EuclideanPH1}$ suffers particularly badly in comparison to the previous two experiments. As in the case of the hollow tori, the lens test will not be utilised as frequently meaning that there will not be as many saving from that. Furthermore, since the point-cloud is \emph{genuinely} a higher dimensional point-cloud the lunes are ``bigger" in the sense that the maximum number of connected components in a given lune is higher, meaning that it will take longer to determine the number of connected components in each lune, leading to the large increase in runtime we observe in Figure \ref{cursed-torus-in-r10}. For the point-clouds where $\mathsf{Ripser}$ can compute $\ph{1}(X)$, $\mathsf{Ripser}$  defeats $\mathsf{EuclideanPH1}$ in terms of runtime. However, like in the previous experiments, $\mathsf{EuclideanPH1}$ provides a practical advantage on longer instances by completing computations that $\mathsf{Ripser}$ could not with the available memory, albeit at the cost of a longer runtime. It is interesting to note in Figure \ref{log-log-cursed-torus} that $\mathsf{EuclideanPH1}$ and $\mathsf{Ripser}$ have similar log-log slopes over the tested range, though the slope for $\mathsf{EuclideanPH1}$ is slightly larger. 

\section{Experiments: Comparing $\mathsf{EuclideanPH1}$ with \\ GUDHI's $\mathsf{Edge\_collapse}$}
\label{sec:comparing-euclideanph1-with-gudhis-edge-collapse}

GUDHI \cite{10.1007/978-3-662-44199-2_28} is a software package developed by INRIA and is an extensive library containing many tools useful in the field of persistent homology. Edge collapses, which we will refer to as $\mathsf{Edge\_collapse}$ described in \cite{boissonnat2020edge} and \cite{glisse2022swap} is a method which has a similar philosophy to $\mathsf{EuclideanPH1}$: reduce the overall number of simplices in the filtration and then use that filtration to obtain the persistence pairs. In GUDHI, persistent homology can be computed using the function $\mathsf{Persistent\_Cohomology}$ which can be found in \cite{persistent-cohomology}. This function uses the algorithm of \cite{de2009persistent} and \cite{tamal_paper} and the compressed annotation matrix implementation of \cite{compressed_paper}. From here on in, for the sake of brevity, we will refer to the combined use of first using $\mathsf{Edge\_collapse}$ and then using $\mathsf{Persistent\_Cohomology}$ simply by $\mathsf{Edge\_collapse}$ as it is the method that fits with the theme of this paper: reducing the number of simplices used in computation. One of the key ways this is done is by the use of dominated edges. Dominated edges can be seen, in some sense, as a middle ground between lunes which have non-empty lens and lunes which simply have one connected component. We recast the definition of a dominated edge in \cite{glisse2022swap} in our notation. 

\begin{definition}
    Consider a $1$-simplex $\langle yz \rangle$. The $1$-simplex $\langle yz \rangle$ is said to be dominated by a point $x\in \lune (\langle yz \rangle)$ if all points in $x' \neq x,\; x' \in \lune (\langle yz \rangle)$ satisfy $\langle x x' \rangle < \langle yz \rangle$. A $1$-simplex $\langle yz \rangle$ is said to be dominated if there exists such an $x\in \lune(\langle yz \rangle)$. 
\end{definition}

\begin{remark}
\label{remark:lens-inside-dominated-inside-onecc}
    A $1$-simplex $\langle yz \rangle$ such that $\lens (\langle yz \rangle) \neq \emptyset$ is dominated by all points that lie in its lens. The converse is not true, one can have $\langle yz \rangle$ be dominated by a point $x$ with $\lens (\langle yz \rangle)$. By definition, all dominated $1$-simplices must be such that $\lune(\langle yz \rangle)$ has one connected component, but the converse is not true. A $1$-simplex $\langle yz \rangle$ can have one connected component without being dominated by any of the points in $\lune(\langle yz \rangle)$. An example is shown in Figure \ref{fig:dominated-non-dominated}.
\end{remark}

\begin{figure}[h!]

    \begin{subfigure}[t]{.45\textwidth}
    \centering
    \includegraphics[scale = 0.5]{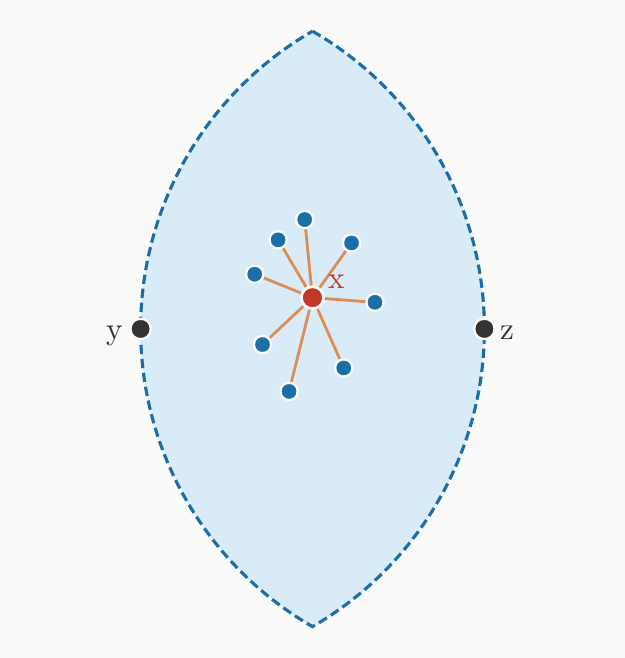}
    \label{dominated-edge-example}
    \end{subfigure}
    \hfill
    \begin{subfigure}[t]{.45\textwidth}
    \centering
    \includegraphics[scale = 0.5]{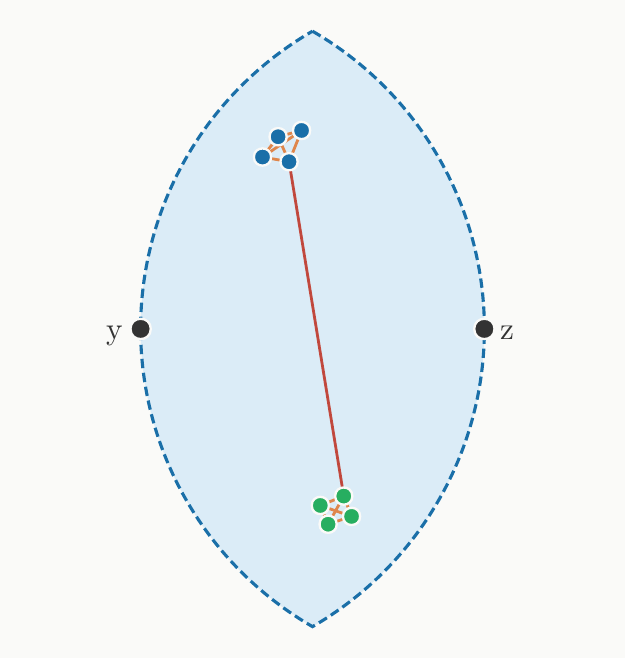}
    \label{one-cc-non-dominated}
\end{subfigure}
\caption{On the left: The $1$-simplex $\langle yz \rangle$ is dominated by $x$. Note that not all edges of the proximity graph are shown. On the right: $\lune(\langle yz \rangle)$ has one connected component but is not dominated by any of the points in $\lune(\langle yz \rangle)$.}
\label{fig:dominated-non-dominated}
\end{figure}

Dominated $1$-simplices are important because they can allow one to remove several simplices from the filtration. The reader interested in further details is directed towards \cite{glisse2022swap}, \cite{boissonnat2020edge}. 

The experiments in this section were performed using the same machine as in Section \ref{sec:experiments}. The same point-clouds used in the experiments in Sections \ref{sec:solid-torus-in-r3-embedded-in-r10-ripser-experiment}, \ref{sec:noisy-hollow-torus-in-r3-embedded-in-r10-ripser-experiment} and \ref{sec:noisy-cursed-torus-experiments-ripser} were used in the experimental results depicted in Figures \ref{fig:gudhi-experiments solid}, \ref{fig:gudhi-experiments-hollow} and \ref{fig:gudhi-experiments-cursed}. All results for $\mathsf{EuclideanPH1}$ were not continued in the interest of time, none of the $\mathsf{EuclideanPH1}$ experiments exhausted the memory of the machine. For the $\mathsf{Edge\_collapse}$ results featured in Figures \ref{fig:gudhi-experiments solid} and \ref{fig:gudhi-experiments-hollow}, $\mathsf{Edge\_collapse}$ also did not exhaust the memory of the machine and also were not continued in the interest of time. However, for the $\mathsf{Edge\_collapse}$ results in Figure \ref{fig:gudhi-experiments-cursed}, these were not continued because $\mathsf{Edge\_collapse}$ \emph{did} exhaust the memory of the machine. For $\mathsf{EuclideanPH1}$, the value of $k$ was taken to be 100 for all runs of the code.

In these experiments as well, the O3 optimisation was used and all other processes, except those necessary to run the operating system, were shut down. The machine was also disconnected from the internet. All measured times are ``wall-clock" times. The specific version of GUDHI used was 3.5.0. The version of $\mathsf{EuclideanPH1}$ used was the same version as the one used in Section \ref{sec:experiments}.

As stated above, the process $\mathsf{Edge\_ collapse}$ stands for the process of performing the edge collapse (obtaining the subfiltration) and then computing  the persistent homology of that filtration using $\mathsf{Persistent\_Cohomology}$. All runtimes listed in the following experiments refer to the time to construct the subfiltration and then compute persistent homology from the subfiltration. $\mathsf{Edge\_collapse}$ had some parameters which could be chosen. We state here our choices and why we believe these to give a fair comparison to $\mathsf{EuclideanPH1}$:

\begin{itemize}
    \item Threshold: This sets the diameter of the largest simplex that $\mathsf{Edge\_collapse}$ will analyse. This was not set since $\mathsf{EuclideanPH1}$ is also not given a maximal diameter of simplex to analyse. 

    \item Maximum dimension: This refers to the maximum degree of persistent homology to compute. This was naturally set to $1$ as $\mathsf{EuclideanPH1}$ only computes $\ph{1}(X)$. 

    \item Edge collapse iteration number: This is the number of times $\mathsf{Edge\_collapse}$ performs the process of finding the subfiltration. This was set to $1$ as $\mathsf{EuclideanPH1}$ only makes a single pass of the $1$-simplices when constructing the Reduced Vietoris-Rips filtration. 

    \item Minimum persistence: This tells $\mathsf{Edge\_collapse}$ to ignore persistence pairs with persistence smaller than a certain size. This was set to $0$, meaning $\mathsf{Edge\_collapse}$ reports all persistence pairs. This is because $\mathsf{EuclideanPH1}$ also reports all persistence pairs. 
\end{itemize}

\begin{figure}[htbp]
    \centering
 
 
    \begin{subfigure}[b]{0.45\textwidth}
        \centering
        \includegraphics[width=\textwidth]{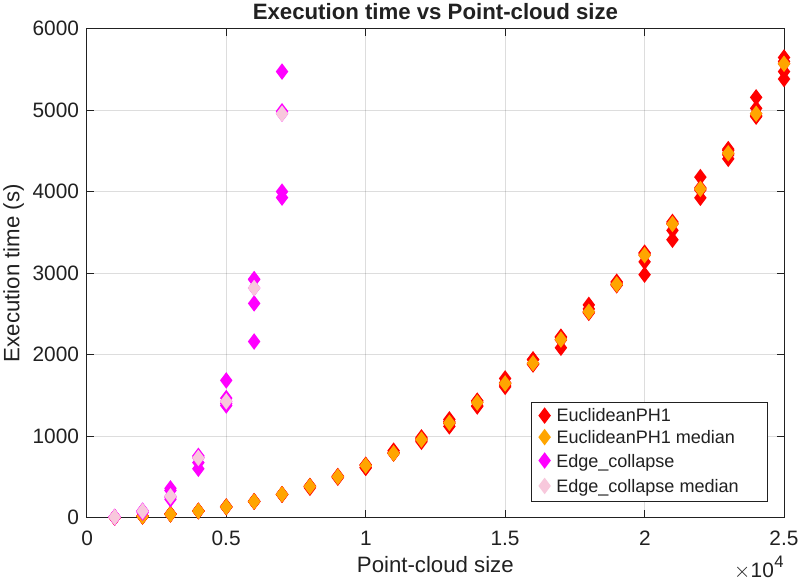}
        \caption{Execution times vs point-cloud size for point-clouds taken from the solid torus in $\mathbb{R}^{10}$. Both the results for $\mathsf{EuclideanPH1}$ and $\mathsf{Edge\_collapse}$ were not taken further in the interest of time. Neither exhausted the memory of the machine.}
        \label{fig:gudhi-experiments-solid-torus-runtimes}
    \end{subfigure}
    \hfill
    \begin{subfigure}[b]{0.45\textwidth}
        \centering
        \includegraphics[width=\textwidth]{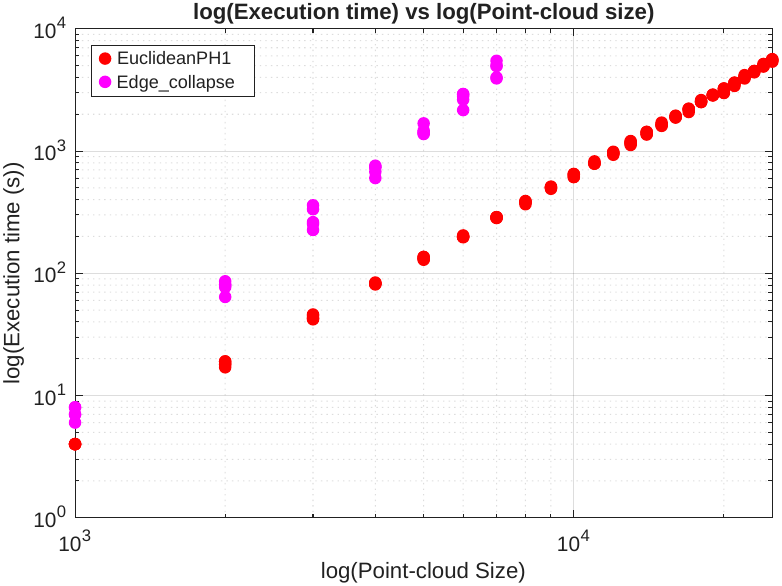}
        \caption{The log-log plot of Figure \ref{fig:gudhi-experiments-solid-torus-runtimes}. The $\mathsf{EuclideanPH1}$ slope is 2.26 and the $\mathsf{Edge\_collapse}$ slope is 3.30. }
        \label{fig:gudhi-experiments-solid-torus-log-log-runtimes}
    \end{subfigure}
 
    \vspace{1em}   
 
 
    \begin{subfigure}[b]{0.45\textwidth}
        \centering
        \includegraphics[width=\textwidth]{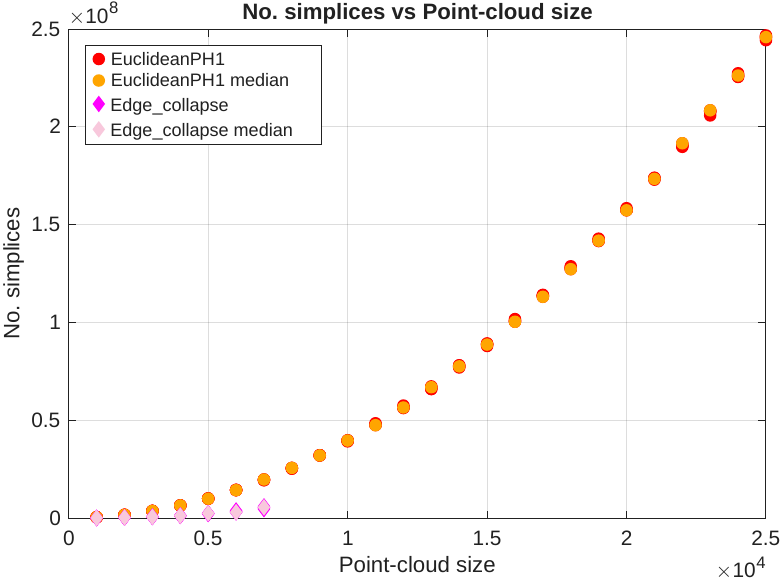}
        \caption{Number of simplices used in the computation of $\ph{1}(X)$ vs point-cloud size for point-clouds taken from the solid torus in $\mathbb{R}^{10}$. }
        \label{fig:gudhi-experiments-solid-torus-simplices}
    \end{subfigure}
    \hfill
    \begin{subfigure}[b]{0.45\textwidth}
        \centering
        \includegraphics[width=\textwidth]{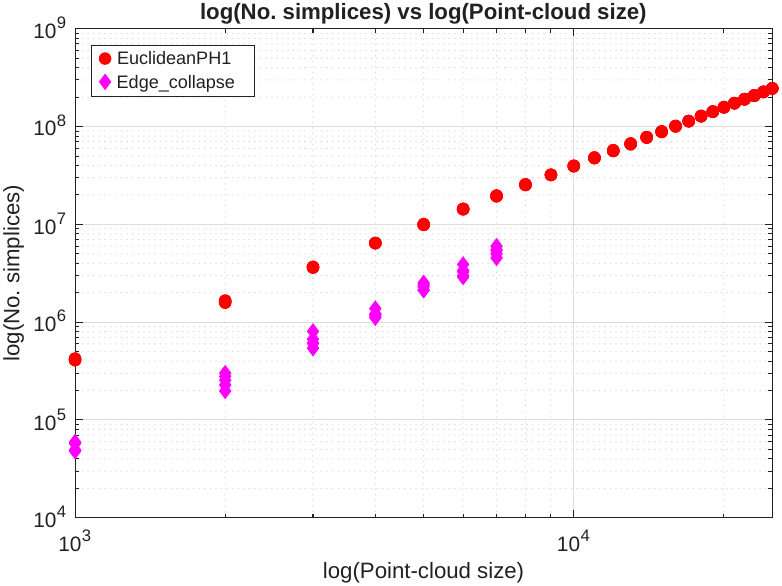}
        \caption{The log-log plot of Figure \ref{fig:gudhi-experiments-solid-torus-simplices}. The $\mathsf{EuclideanPH1}$ slope is 1.99 and the $\mathsf{Edge\_collapse}$ slope is 2.35.}
        \label{fig:gudhi-experiments-solid-torus-log-log-simplices}
    \end{subfigure}
 
    \caption{%
        Experimental results comparing $\mathsf{EuclideanPH1}$ against $\mathsf{Edge\_collapse}$ for point-clouds taken from a solid torus embedded in $\mathbb{R}^{10}$. 
    }
    \label{fig:gudhi-experiments solid}
 
\end{figure}

\begin{figure}[htbp]
    \centering
 
 
    \begin{subfigure}[b]{0.45\textwidth}
        \centering
        \includegraphics[width=\textwidth]{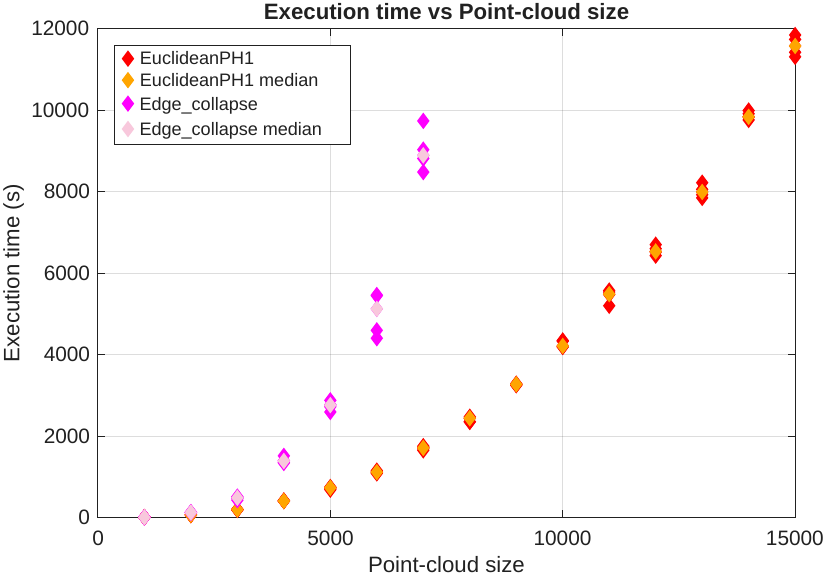}
        \caption{Execution times vs point-cloud size for point-clouds taken from a hollow noisy $2$-torus in $\mathbb{R}^{10}$. Both the results for $\mathsf{EuclideanPH1}$ and $\mathsf{Edge\_collapse}$ were not taken further in the interest of time. Neither exhausted the memory of the machine.}
        \label{fig:gudhi-experiments-hollow-torus-runtimes}
    \end{subfigure}
    \hfill
    \begin{subfigure}[b]{0.45\textwidth}
        \centering
        \includegraphics[width=\textwidth]{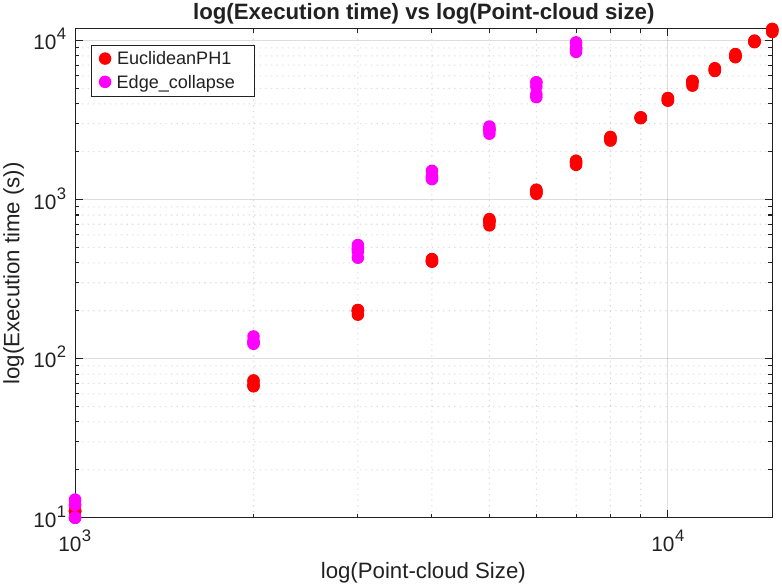}
        \caption{The log-log plot of Figure \ref{fig:gudhi-experiments-hollow-torus-runtimes}. The $\mathsf{EuclideanPH1}$ slope is 2.56 and the $\mathsf{Edge\_collapse}$ slope is 3.42. }
        \label{fig:gudhi-experiments-hollow-torus-log-log-runtimes}
    \end{subfigure}
 
    \vspace{1em}   
 
 
    \begin{subfigure}[b]{0.45\textwidth}
        \centering
        \includegraphics[width=\textwidth]{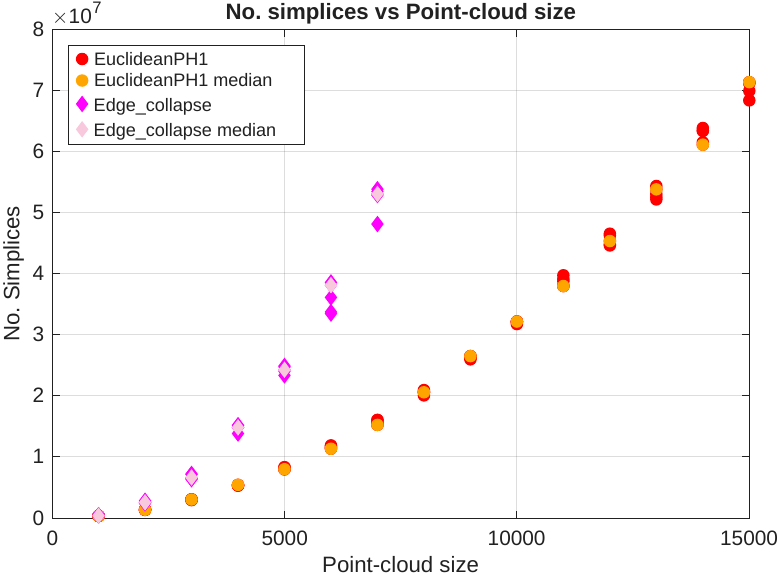}
        \caption{Number of simplices used in the computation of $\ph{1}(X)$ vs point-cloud size for point-clouds taken from a hollow noisy $2$-torus in $\mathbb{R}^{10}$.}
        \label{fig:gudhi-experiments-hollow-torus-simplices}
    \end{subfigure}
    \hfill
    \begin{subfigure}[b]{0.45\textwidth}
        \centering
        \includegraphics[width=\textwidth]{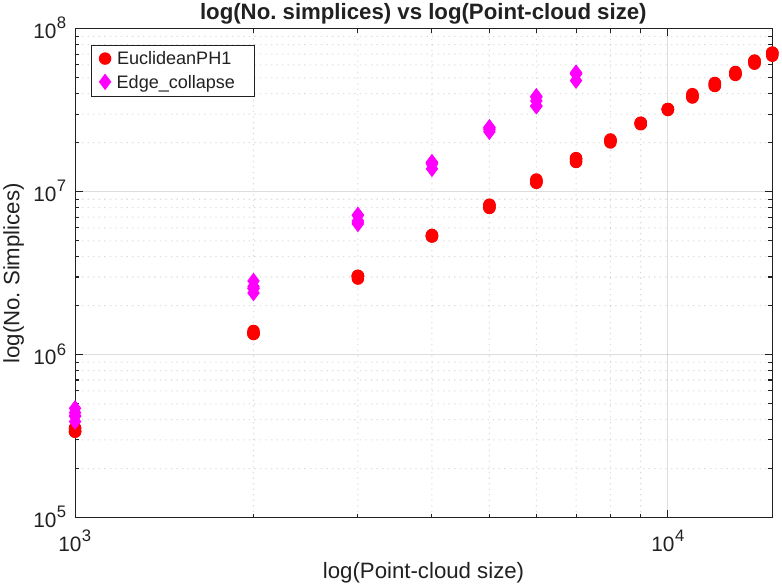}
        \caption{The log-log plot of Figure \ref{fig:gudhi-experiments-hollow-torus-simplices}. The $\mathsf{EuclideanPH1}$ slope is 1.96 and the $\mathsf{Edge\_collapse}$ slope is 2.46. }
        \label{fig:gudhi-experiments-hollow-torus-log-log-simplices}
    \end{subfigure}
 
    \caption{%
        Experimental results comparing $\mathsf{EuclideanPH1}$ against $\mathsf{Edge\_collapse}$ for point-clouds taken from the noisy hollow $2$-torus in $\mathbb{R}^{10}$. 
    }
    \label{fig:gudhi-experiments-hollow}
 
\end{figure}

\begin{figure}[htbp]
    \centering
 
 
    \begin{subfigure}[b]{0.45\textwidth}
        \centering
        \includegraphics[width=\textwidth]{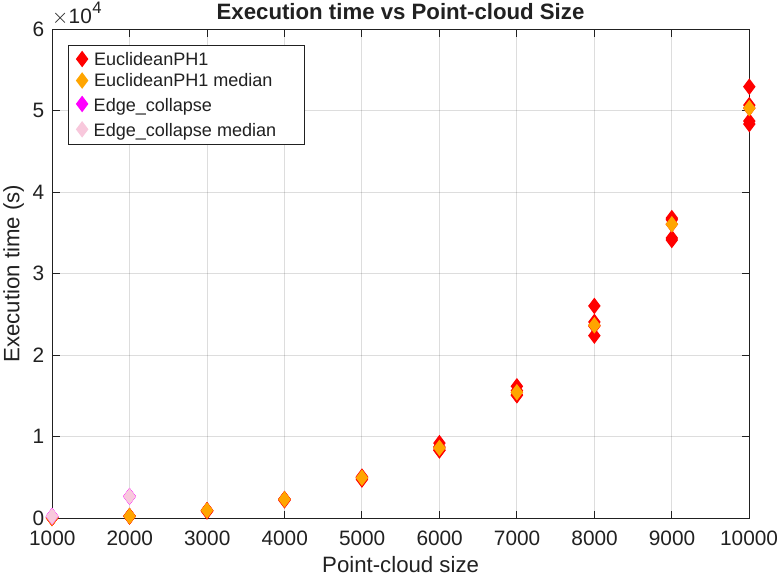}
        \caption{Execution times vs point-cloud size for point-clouds taken from the noisy $5$-torus. The results for $\mathsf{EuclideanPH1}$ were not taken further in the interest of time. The results of $\mathsf{Edge\_collapse}$ were not continued as $\mathsf{Edge\_collapse}$ exhausted the memory of the machine.}
        \label{fig:gudhi-experiments-cursed-torus-runtimes}
    \end{subfigure}
    \hfill
    \begin{subfigure}[b]{0.45\textwidth}
        \centering
        \includegraphics[width=\textwidth]{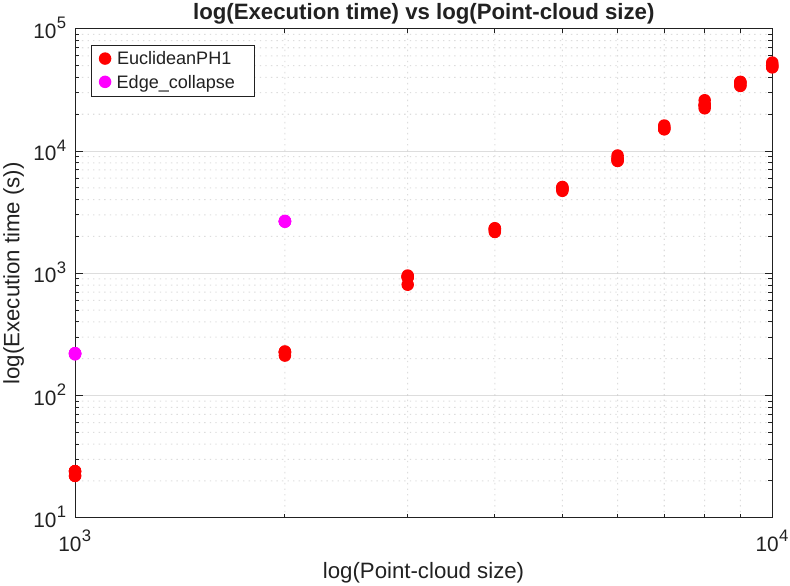}
        \caption{The log-log plot of Figure \ref{fig:gudhi-experiments-cursed-torus-runtimes}. The $\mathsf{EuclideanPH1}$ slope is 3.34 and the $\mathsf{Edge\_collapse}$ slope is 3.59. }
        \label{fig:gudhi-experiments-cursed-torus-log-log-runtimes}
    \end{subfigure}
 
    \vspace{1em}   
 
 
    \begin{subfigure}[b]{0.45\textwidth}
        \centering
        \includegraphics[width=\textwidth]{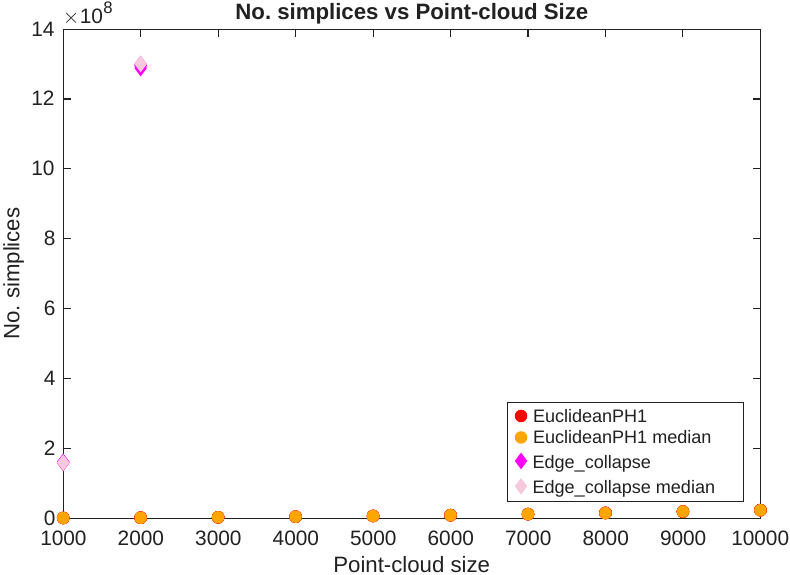}
        \caption{Number of simplices used in the computation of $\ph{1}(X)$ vs point-cloud size for point-clouds taken from the noisy $5$-torus.}
        \label{fig:gudhi-experiments-cursed-torus-simplices}
    \end{subfigure}
    \hfill
    \begin{subfigure}[b]{0.45\textwidth}
        \centering
        \includegraphics[width=\textwidth]{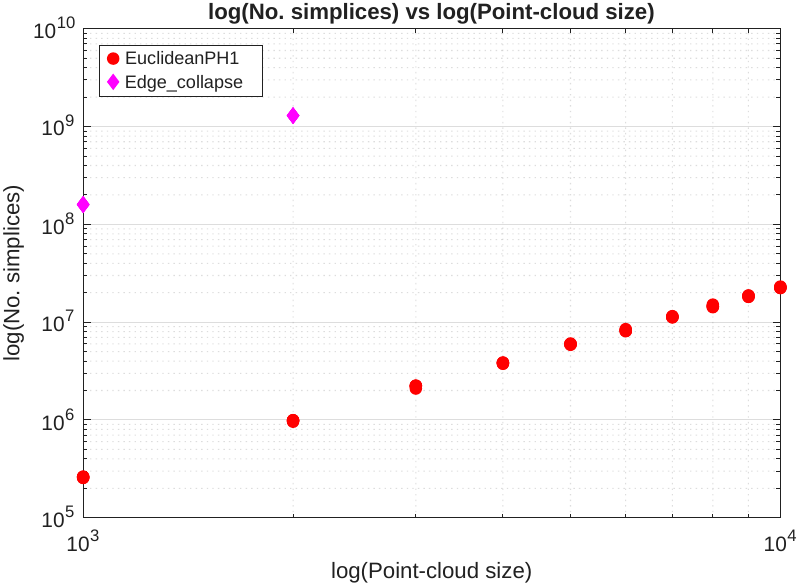}
        \caption{The log-log plot of Figure \ref{fig:gudhi-experiments-cursed-torus-simplices}. The $\mathsf{EuclideanPH1}$ slope is 1.94 and the $\mathsf{Edge\_collapse}$ slope is 3.02.}
        \label{fig:gudhi-experiments-cursed-torus-log-log-simplices}
    \end{subfigure}
 
    \caption{%
        Experimental results comparing $\mathsf{EuclideanPH1}$ against $\mathsf{Edge\_collapse}$ for point-clouds taken from the noisy $5$-torus. 
    }
    \label{fig:gudhi-experiments-cursed}
 
\end{figure}

\section{Discussion: $\mathsf{EuclideanPH1}$ vs $\mathsf{Edge\_collapse}$}
\label{Discussion:EdgeCollapse}
The experiments indicate that $\mathsf{EuclideanPH1}$ generally performs more favorably than $\mathsf{Edge\_collapse}$ in terms of runtime and memory usage when computing $\ph{1}(X)$ for the point-clouds tested. It should be stressed however, that the smaller filtration provided by $\mathsf{Edge\_collapse}$ can also be used to compute higher-degree persistent homology, while the smaller filtration constructed for $\mathsf{EuclideanPH1}$ can only be used to compute $\ph{1}(X)$. For the point-clouds studied, $\mathsf{EuclideanPH1}$ was better at reducing the number of simplices needed to compute $\ph{1}(X)$. For all point-clouds in these experiments it consistently reduced the number of simplices used from $O(n^3)$ to $O(n^2)$ as can be seen from the slopes in Figures
\ref{fig:gudhi-experiments-solid-torus-log-log-simplices}, \ref{fig:gudhi-experiments-hollow-torus-log-log-simplices} and
\ref{fig:gudhi-experiments-cursed-torus-log-log-simplices}, which are all very close to 2.

Whilst $\mathsf{EuclideanPH1}$ generally provides a lower runtime for computing $\ph{1}(X)$ than $\mathsf{Edge\_collapse}$, it should be noted that $\mathsf{EuclideanPH1}$ has been specifically optimized for Euclidean point-clouds, taking advantage of the geometry of Euclidean space by using $kd$-trees and specialized algorithms for computing $\rng(X)$. On the other hand, the $\mathsf{Edge\_collapse}$ implementation in $\mathrm{GUDHI}$ can be used for any filtration of flag complexes.

In what follows, we will more closely analyze the number of simplices used by $\mathsf{EuclideanPH1}$ and $\mathsf{Edge\_collapse}$. 

\subsection{Solid tori embedded in $\mathbb{R}^{10}$}
\label{subsection: gudhi-discussion-solid-3d-tori-embedded-in-R10}
In this set of experiments $\mathsf{Edge\_collapse}$ had the most favourable performance in terms of the scaling in the number of simplices, over the tested range, used for computation of $\ph{1}(X)$. This is not surprising, a $1$-simplex $\langle yz \rangle$ will be dominated provided there is a point $v\in \lune(\langle yz \rangle)$ such that for all $x\in \lune(\langle yz \rangle)$, $\langle vx \rangle < \langle yz \rangle$. This will be the case for many $1$-simplices that lie ``within" the overall torus structure. In fact, many of these $1$-simplices will have non empty lens, which can be seen by the comparably favourable $\mathsf{EuclideanPH1}$ run times.   

\subsection{Hollow noisy $2$-tori embedded in $\mathbb{R}^{10}$}
In this set of experiments, in terms of the scaling of the numbers of simplices used, $\mathsf{Edge\_collapse}$'s performance did not experience a significant change when compared to the results of Section \ref{subsection: gudhi-discussion-solid-3d-tori-embedded-in-R10}. As can be inferred from the longer runtimes of $\mathsf{EuclideanPH1}$ there are less simplices with non-empty lens, however recall from Remark \ref{remark:lens-inside-dominated-inside-onecc} that this does not mean these edges are not dominated. Many $1$-simplices with empty lens will still be dominated, a simple example being $1$-simplices which have one point in their lune which does not lie in the lens. 

\subsection{Noisy $5$-torus}
In this set of experiments we witnessed the poorest performance of $\mathsf{Edge\_collapse}$. Whilst the run times for $\mathsf{EuclideanPH1}$ were very large, $\mathsf{Edge\_collapse}$ could not even compute $\ph{1}(X)$ for point-clouds for point-clouds containing $3000$ or more points. The number of simplices $\mathsf{Edge\_collapse}$ used was basically equal to the complete Vietoris-Rips filtration. For $n=1000$ an average of $95.74\%$ of all simplices in the complete filtration were used in the subfiltration. For $n=2000$ an average $97.16\%$ of all simplices in the complete filtration were used in the subfiltration. This can be confirmed from the results of Figure \ref{fig:gudhi-experiments-cursed-torus-simplices} and the following numerical facts. 

\begin{equation}
\binom{1000}{3} + \binom{1000}{2} + \binom{1000}{1} = 166667500 \approx 1.67 \times 10^8
\end{equation}

and 

\begin{equation}
\binom{2000}{3} + \binom{2000}{2} + \binom{2000}{1} = 1333350000 \approx 1.33 \times 10^9
\end{equation}

This suggests that $\mathsf{Edge\_collapse}$ is only removing a small fraction of simplices and that very few $1$-simplices were dominated. This is not surprising considering the nature of the $5$-torus. Unlike the other two spaces this space is not intrinsically three dimensional. Principal component analysis shows the percentages for the components of one of the $n=1000$ point-clouds to be \\$23.23\%, 22.74\%, 14.86\%, 13.81\%, 8.51\%, 7.78\%, 3.70\%, 3.54\%, 0.97\%, 0.86\%$. Due to the inherent higher dimensionality of the data, lunes are in some sense ``bigger" and are able to have more connected components, which in turn leads to more spatial configurations of points that cause a lune of a $1$-simplex to have one connected component but that $1$-simplex to not be dominated by any point. This means that in this setup, the condition for domination of an edge is more demanding. 

\section{Conclusion and Future Work}

In this paper we have defined the notion of Reduced Vietoris-Rips complexes and filtrations on a point cloud $X$ and have shown in Theorem \ref{theorem-VR-RVR-isomorphism} that they give the same degree-1 persistent homology as the standard Vietoris-Rips filtration. We also provide a higher degree analogue of Theorem \ref{theorem-VR-RVR-isomorphism} in the appendix in the form of Theorem \ref{deg-q-theorem-VR-RVR-isomorphism}. We have also designed and implemented code as the $\mathsf{EuclideanPH1}$ package \cite{EuclideanPH1} which can compute $\ph{1}(X)$ for a point cloud in Euclidean space. An output analysis of this code has been presented in Sections \ref{sec:experiments} and \ref{sec:comparing-euclideanph1-with-gudhis-edge-collapse}. We have shown that for smaller point-clouds $\mathsf{Ripser}$ generally has better runtimes. However for larger point clouds the lower memory requirements of $\mathsf{EuclideanPH1}$ mean that $\mathsf{EuclideanPH1}$ is able to compute $\ph{1}(X)$ while $\mathsf{Ripser}$ is unable to do so due to running out of memory. We have also shown $\mathsf{EuclideanPH1}$ fares favourably against $\mathsf{Edge\_collapse}$ in terms of runtime and memory requirements, albeit at the sacrifice of producing a subfiltration which only computes degree-$1$ persistent homology. 

A potential avenue of future work is to develop an algorithm for $\ph{q}(X)$ where $q \geq 2$. It is possible to extend Theorem \ref{theorem-VR-RVR-isomorphism} to higher dimensions, but the benefits diminish as the degree of homology increases. The extension to Theorem \ref{theorem-VR-RVR-isomorphism} is given in Appendix~\ref{higher-degree-homologies-section}.

This work focuses on point clouds in Euclidean space, so another avenue of future work is to find alternative algorithms for point clouds in non-Euclidean metric spaces.  
Theorem $\ref{theorem-VR-RVR-isomorphism}$ applies to point clouds in general metric spaces but it may not lead to such significant computational efficiencies without the use of a $kd$-tree. 
One would need to an algorithm to find the lune of each 1-simplex without resorting to testing all points to see if they are in the lune. 
It might be more fruitful to restrict oneself to a certain class of metric spaces rather than find an algorithm which can be applied to all metric spaces. One such class might be the set of doubling metric spaces. These metric spaces are useful since the number of connected components in a lune is  bounded above in a way that is not dependent on $n$. 
This means that we can still assert that the use of Theorem \ref{theorem-VR-RVR-isomorphism} will reduce the overall number of 2-simplices from $O(n^3)$ to $O(n^2)$. 

\begin{appendices}

\section{Extension into higher degree homology}
\label{higher-degree-homologies-section}

Here we present the higher degree homology equivalent of Theorem \ref{theorem-VR-RVR-isomorphism}. Before we do so, we need to define the higher degree notions of a lune and Reduced Vietoris-Rips complex. 

\begin{definition}[Lune]
Consider a point-cloud $X$ with its corresponding Vietoris-Rips filtration $\Vr _{\bullet}(X)$. Then for a $q$-simplex $\sigma = \langle y_0....y_{q}\rangle$ we define $\lune (\sigma)$ in the following fashion. 
\begin{equation}
    \lune (\sigma) = \{x \in X \;| \; \langle x y_0...\hat{y_i}...y_{q} \rangle < \langle y_0...y_{q} \rangle \; \forall i\in \{0,...,q\} \}
\end{equation}
\end{definition}

Here $\hat{y_i}$ denotes the omission of $y_i$ from $\langle xy_{0}...\hat{y_i}...y_q\rangle$. Note that if $q = 1$ the definition reduces to the definition of a lune for an edge. Having defined the lune for higher dimension simplices we define the number of connected components of a lune for a higher dimension simplex. 

\begin{definition}[Connected components of a lune]
\label{def-connected-components-of-lune-appendix}
Consider a $q$-simplex $\sigma = \langle y_{0}...y_{q} \rangle$. Consider a graph with vertices consisting of the points in $\lune (\sigma)$. Join two points $p_{1},p_{2} \in \lune ( \sigma )$ by an edge if $\langle p_{1}p_{2} y_{0}...\hat{y}_{i}...\hat{y}_{j}...y_{q} \rangle  < \sigma $ for all $0 \leq i < j \leq q$. Suppose this graph has $c$ connected components, then we say that $\lune (\sigma)$ has $c$ connected components. 
\end{definition}

We will also need a degree-$q$ lune function. 

\begin{definition}[Degree-$q$ lune function]
\label{lune-function-appendix}
    Recall $\spx{q}(\Vr_{\infty}(X))$ denotes the set of $q$-simplices in $\Vr_{\infty}(X)$. We define a Lune function $L^{q}:\spx{q}(\Vr_{\infty}(X))\rightarrow 2^{X}$ as a function that takes a  $q$-simplex $\sigma$, with $c_{\sigma}$ connected components in its lune, to a set $\{x_{1},...,x_{c_{\sigma}}\}$. The points $x_{1},...,x_{c_{\sigma}}$ are chosen from the connected components of $\sigma$, one point from each connect component. 
\end{definition}

We are now ready to define the higher degree analogue to Definition \ref{def-reduced-vietoris-rips-complex}. 

\begin{definition}[Degree-$q$ reduced Vietoris-Rips complex]
\label{deg-q-def-reduced-vietoris-rips-complex}
Consider a point-cloud $X$. The degree-$q$ Reduced Vietoris-Rips Complex of $X$ with scale $r$, 
denoted $\mathcal{R}^{q}_r(X)$ 
is the simplicial complex which consists of the following:

\begin{itemize}
    \item All $i$-simplices for $i = 0,...,q$. 

    \item ($q+1$)-simplices are as follows. 
    For each $q$-simplex $\sigma = \langle y_{0}...y_{q} \rangle \in \mathcal{R}^{q}_{r}(X)$ the $(q+1)$-simplices $\langle x y_0...y_{q} \rangle, x\in L^{q}(\sigma)$ are in $\mathcal{R}^{q}_{r}(X)$. 
\end{itemize}
\end{definition}

Now we present the higher degree analogue of Theorem \ref{theorem-VR-RVR-isomorphism}. 

\begin{theorem}
\label{deg-q-theorem-VR-RVR-isomorphism}
Let $q \geq 1$ and consider a point-cloud $X$.   
Then there exists a family of isomorphisms $\theta_{\bullet}$ such that the following diagram commutes for all $r_1$ and $r_2$ such that $0 \leq  r_1 < r_2$

\begin{equation}
\begin{tikzcd}
H_{q}(\rvr^{q}_{r_1}(X)) \arrow[r, "f_{r_1}^{r_2}"] \arrow[d, "\theta_{r_1}"'] & H_{q}(\rvr^{q}_{r_2}(X)) \arrow[d, "\theta_{r_2}"] \\
H_{q}(\Vr_{r_1}(X)) \arrow[r, "g_{r_1}^{r_2}"] & H_{q}(\Vr_{r_2}(X))
\end{tikzcd}
\end{equation}

 Above, $f_{r_1}^{r_2}$ and $g_{r_1}^{r_2}$ are the maps at homology level obtained by applying $H_{q}(-)$ to the inclusions $\rvr_{r_1}^{q}(X) \subset \rvr_{r_2}^{q}(X)$ and $\Vr_{r_1}(X) \subset \Vr_{r_2}(X)$. 

\end{theorem}

\paragraph{The proof of Theorem \ref{deg-q-theorem-VR-RVR-isomorphism}} 

Let $r$ be arbitrary. We first define $\theta_{r}$ and then show it is an isomorphism. Let $\gamma + B_{q}(\rvr^{q}_{r}(X))$ be an element in $H_{q}(\rvr^{q}_{r}(X))$. Then we define $\theta_{r}$ as follows:
\begin{equation}
\theta_{r}(\gamma + B_{q}(\rvr_{r}^{q}(X))) := \gamma + B_{q}(\Vr_{r}(X))
\end{equation}
This mapping is well defined because $B_{q}(\rvr_{r}^{q}(X)) \subset B_{q}(\Vr_{r}(X))$.

\begin{lemma}
\label{q-theta-is-surjective}
    $\theta_{r}$ is surjective for all $r \geq 0$. 
\end{lemma}

\begin{proof}
    Let $\gamma + B_{q}(\Vr_{r}(X)) \in H_{q}(\Vr_{r}(X))$. 
    Since $\gamma \in C_{q}(\Vr_{r}(X))$ and all $q$-simplices of $\Vr_{r}(X)$ are also in $\rvr_{r}^{q}(X)$, it follows that $\gamma + B_{q}(\rvr_{r}^{q}(X))$ is mapped to $\gamma + B_{q}(\Vr_{r}(X))$ by $\theta_{r}$.
\end{proof}

In order to show that $\theta_{r}$ is an isomorphism for all $r$ we need to show that $\theta_{r}$ is injective for all $r$. Before we do so, we need the following preliminary lemmas and new notation. 

\begin{definition}
    Consider a point $x$ and let $\sigma$ be the $q$-simplex denoted by $\langle y_{0}...y_{q} \rangle$. Then we define the $(q+1)$-simplex $\langle x \sigma \rangle$ as $\langle xy_{0}...y_{q} \rangle$. 
\end{definition}

\begin{lemma}
\label{q-isomorphism-inclusions}
    Let $s > 0$. If $\theta_{s}$ is an isomorphism, then it follows that $B_{q}(\Vr_{s}(X)) = B_{q}(\rvr^{q}_{s}(X))$. 
\end{lemma}

\begin{proof}
The inclusion $B_{q}(\rvr_{s}^{q}(X)) \subset B_{q}(\Vr_{s}(X))$ follows from the fact that $\rvr_{s}^{q}(X) \subset \Vr_{s}(X)$. To get the reverse inclusion let $\gamma \in B_{q}(\Vr_{s}(X))$. Then we have $\theta_{s}(\gamma + B_{q}(\rvr_{s}^{q}(X))) = \gamma + B_{q}(\Vr_{s}(X)) = 0 + B_{q}(\Vr_{s}(X))$. Since $\theta_{s}$ is an isomorphism it follows that $\gamma \in \rvr_{s}^{q}(X)$. Thus we have the reverse inclusion $B_{q}(\Vr_{s}(X)) \subset B_{q}(\rvr_{s}^{q}(X))$.
\end{proof}

\begin{lemma}
\label{q-first-edge-proof}
    Consider $r>0$ arbitrary and suppose that $\theta_{s}$ is injective for all $s<r$. Let all the $q$-simplices of diameter $r$ be $\tau_{1}  < \cdots < \tau_{m_r} $. Then $\lune (\tau_{1}) = \emptyset$ and thus
    \begin{equation}
    \label{q-same-edges-equation-1}
         \{ \partial (\langle  x \tau_{1} \rangle) \; | \; x\in \lune (\tau_{1}) \} = \emptyset 
    \end{equation}
    and thus we trivially have 
    \begin{equation}
    \label{q-same-edges-equation-1-1}
         \{ \partial (\langle  x \tau_{1} \rangle) \; | \; x\in \lune (\tau_{1}) \} \subset B_{q}(\rvr_{r}^{q}(X)).
    \end{equation}
\end{lemma}
\begin{proof}
    Denote $\tau_{i}$ by $\langle y_{0i} y_{1i} ... y_{qi} \rangle.$ Suppose for contradiction that $\lune (\tau_{1}) \neq \emptyset$. Let $x\in \lune(\tau_{1})$, then by the definition of of $\lune(\tau_{1})$ we must have that $\langle x y_{01} ...\hat{y_{i1}}...y_{q1} \rangle < \tau_{1}$. At least one of these simplices must have diameter $r$, which contradicts the fact that $\tau_{1}$ is the first $q$-simplex with diameter $r$. Thus we must have $\lune(\tau_{1}) = \emptyset$.  
\end{proof}

\begin{lemma}
\label{q-the-big-lemma}
    Consider $r>0$ arbitrary and suppose that $\theta_{s}$ is injective for all $s<r$. Let $\sigma$ be a $(q+1)$-simplex in $\Vr_{r}(X)$ such that $\diam (\sigma) = r$, then we have that $\partial \sigma \in B_{q}(\rvr_{r}^{q}(X))$. That is we have $B_{q}(\rvr_{r}^{q}(X)) = B_{q}(\Vr_{r}(X))$.
\end{lemma}

\begin{proof}
    First observe that by Lemma \ref{q-theta-is-surjective} that $\theta_s$ is an isomorphism for all $s < r$. This means that $B_{q}(\rvr_{s}^{q}(X)) = B_{q}(\Vr_{s}(X))$ by Lemma \ref{q-isomorphism-inclusions}. 
    
    Let all the $q$-simplices of diameter $r$ be $ \tau_{1}  < \cdots <  \tau_{m_{r}} $ and denote $\tau_{k} = \langle y_{0k}....,y_{qk} \rangle$. We wish to show that all $(q+1)$-simplices $\sigma$, with $\diam(\sigma) = r$ are such that $\partial \sigma \in B_{q}(\rvr_{r}^{q}(X))$. This is equivalent to showing that 

    \begin{equation}
    \label{q-same-edges-equation}
         \{ \partial ( \langle  x \tau_{j} \rangle)\; | \; x\in \lune (\tau_{j}) \} \subset B_{q}(\rvr_{r}^{q}(X))
    \end{equation}

for $j = 1,...,m_{r}$. We will prove that (\ref{q-same-edges-equation}) holds for $j = 1,...,m_{r}$ by induction. That is we will show the following:

\begin{itemize}
    \item That (\ref{q-same-edges-equation}) holds for $j=1$. This is just Lemma \ref{q-first-edge-proof}. 

    \item If (\ref{q-same-edges-equation}) holds for $j=1,...,k$, then it holds for $j = k+1$. 
\end{itemize}

We now show that if (\ref{q-same-edges-equation}) holds for $j = 1,...,k$, then it holds for $j = k+1$. 

Let $x$ be any point in $\lune (\langle y_{0 (k+1)}...y_{q (k+1)} \rangle)$ and consider the \\ $(q+1)$-simplex $\langle x y_{0(k+1)}...y_{q(k+1)} \rangle$. Let $w$ be the point in the connected component of $\lune (\langle y_{0(k+1)}...y_{q(k+1)} \rangle)$ that also contains $x$ such that $w\in L(\langle y_{0(k+1)}...y_{q(k+1)}\rangle)$. If $x=w$ then there is nothing to prove since $\sigma \in \rvr^{q}_{r}(X)$ by definition and thus $\partial \sigma \in B_{1}(\rvr_{r}^{q}(X))$. Otherwise, since $x$ and $w$ are in the same component we know there exists $v_{0},...,v_{t} \in \lune (\langle y_{0(k+1)}...y_{q(k+1)} \rangle)$ such that $x = v_0, v_1, ..., v_t = w$ forms a path with

\begin{equation}
\langle v_{i}v_{i+1}y_{0(k+1)}...\hat{y}_{l_1(k+1)}...\hat{y}_{l_2(k+1)}...y_{q(k+1)} \rangle < \langle y_{0(k+1)}...y_{q(k+1)} \rangle 
\label{path-equation}
\end{equation}

for all $0 \leq l_1 < l_2 \leq q$.

We know that $\partial (\partial (\langle y_{0(k+1)}...y_{q(k+1)} v_{i}v_{i+1} \rangle )) = 0$ and thus 
\begin{equation}
\begin{aligned}
    \partial (\langle y_{0(k+1)}...y_{q(k+1)} v_{i} \rangle ) = 
    \partial (\langle y_{0(k+1)}...y_{q(k+1)} v_{i+1} \rangle) \\ + \sum_{l=0}^{q}\partial (\langle v_{i}v_{i+1}y_{0(k+1)}...\hat{y}_{l(k+1)}...y_{q(k+1)} \rangle)
    \end{aligned}
    \label{q-equation-general-proof}
\end{equation}
We now show that $\partial (\langle v_{i}v_{i+1}y_{0(k+1)}...\hat{y}_{l(k+1)}...y_{q(k+1)} \rangle) \in B_{q}(\rvr_{r}^{q}(X))$ for $l = 0,...,q$ :

\begin{itemize}
    \item $\langle v_{i}v_{i+1}y_{0(k+1)}...\hat{y}_{l_{1}(k+1)}...\hat{y}_{l_{2}(k+2)}...y_{q(k+1)} \rangle < \langle y_{0(k+1)}...y_{q(k+1)} \rangle$ \\ for all $0 \leq l_1 < l_2 \leq q$ by (\ref{path-equation}). 

    \item $\langle v_{i} y_{0(k+1)}...\hat{y}_{l(k+1)}...y_{q(k+1)} \rangle < \langle y_{0(k+1)}...y_{q(k+1)} \rangle $ for $l =0,...,q$ since 
    
    \vspace{3pt}
    $v_{i} \in \lune (\langle y_{0(k+1)}...y_{q(k+1)} \rangle)$. 

    \item $\langle v_{i+1} y_{0(k+1)}...\hat{y}_{l(k+1)}...y_{q(k+1)} \rangle < \langle y_{0(k+1)}...y_{q(k+1)} \rangle $ for $l =0,...,q$ since 
    
    \vspace{3pt}
    $v_{i+1} \in \lune (\langle y_{0(k+1)}...y_{q(k+1)} \rangle)$.
    
\end{itemize}

For $l= 0,...,q$, all faces of $\langle v_{i}v_{i+1}y_{0(k+1)}...\hat{y}_{l(k+1)}...y_{q(k+1)} \rangle$ appear in the filtration before $\langle y_{0(k+1)}...y_{q(k+1)} \rangle$. For each $l=0,...,q$ we either have all \\faces of $\langle v_{i}v_{i+1}y_{0(k+1)}...\hat{y}_{l(k+1)}...y_{q(k+1)} \rangle$ have diameter less than $r$ or at least one of them has diameter $r$. If they all have diameter less than $r$ then \\ $\partial (\langle v_{i}v_{i+1}y_{0(k+1)}...\hat{y}_{l(k+1)}...y_{q(k+1)} \rangle) \in B_{q}(\Vr_{s}(X)) = B_{q}(\rvr^{q}_{s}(X))$ for $s < r$. If at \\ least one of the faces of $\langle v_{i}v_{i+1}y_{0(k+1)}...\hat{y}_{l(k+1)}...y_{q(k+1)} \rangle$ has diameter $r$ then let $\eta$ be the face with diameter $r$ that appears latest in the filtration. Let $z$ be the \\ point in $\langle v_{i}v_{i+1}y_{0(k+1)}...\hat{y}_{l(k+1)}...y_{q(k+1)} \rangle$ that does not appear in $\eta$, then we have that $z\in \lune(\eta)$. Thus $\partial (\langle v_{i}v_{i+1}y_{0(k+1)}...\hat{y}_{l(k+1)}...y_{q(k+1)} \rangle) = \partial(\langle z \eta\rangle)$. Since $\eta$ \\ comes before $\tau_{k+1}$ in the filtration it follows that 

\begin{equation}
\partial (\langle v_{i}v_{i+1}y_{0(k+1)}...\hat{y}_{l(k+1)}...y_{q(k+1)} \rangle) \in \{\partial (\langle x \tau_{j} \rangle) | x \in \lune (\tau_{j})\}
\end{equation}

for some $j = 1,...,k$.
Plugging $i = t-1$ into  (\ref{q-equation-general-proof}) we obtain the following.

\begin{equation}
\begin{aligned}
\label{q-equation-general-proof-q-1}
    \partial (\langle y_{0(k+1)}...y_{q(k+1)} v_{t-1} \rangle ) = 
    \partial (\langle y_{0(k+1)}...y_{q(k+1)} v_{t} \rangle)\\ + \sum_{l=0}^{q}\partial (\langle v_{t-1}v_{t}y_{0(k+1)}...\hat{y}_{l(k+1)}...y_{q(k+1)} \rangle)
    \end{aligned}
\end{equation}
We have already shown that the summation on the right hand side of (\ref{q-equation-general-proof-q-1}) is in $B_{q}(\rvr_{r}^{q}(X))$. Remembering that $v_{t} = w$, the fact that the first term on the right hand side of (\ref{q-equation-general-proof-q-1}) is in $B_{q}(\rvr_{r}^{q}(X))$ follows from $\langle v_{t}\tau_{k+1}\rangle$ being a $(q+1)$-simplex in $\rvr_{r}^{q}(X)$ by the definition of $\rvr_{r}^{q}(X)$. Thus it follows that $\partial (\langle v_{t-1} \tau_{k+1} \rangle) \in B_{q}(\rvr_{r}^{q}(X))$. Using (\ref{q-equation-general-proof}) repeatedly by letting $i = t-2, t-3, ...,0$ we can subsequently show that $\partial (\langle \tau_{k+1}v_{t-2}),...,\partial(\langle \tau_{k+1}v_{0}\rangle) \in B_{q}(\rvr_{r}^{q}(X))$. Remembering that $v_{0} = x$ we thus have that $\partial (\langle x\tau_{k+1} \rangle)   \in B_{q}(\rvr_{r}^{q}(X))$.

\end{proof}

We are finally in a position to show that $\theta_{r}$ is injective for all $r \geq 0$.

\begin{lemma}
    $\theta_{r}$ is injective for all $r\geq 0$.
\end{lemma}

\begin{proof}
We now show that $\theta_{r}$ is 
for all values of $r$. We will use a proof by induction by inducting on the value of $r$. That is to say, we will prove the following.

\begin{itemize}
    \item $\theta_{0}$ is injective. 

    \item If $\theta_{s}$ is injective for all $s<r$, then $\theta_{r}$ is injective. 
\end{itemize}

Since $H_{q}(\Vr_{0}(X))$ and $H_{q}(\rvr_{0}^{q}(X))$ are zero the injectivity of $\theta_{0}$ is immediate. Now we show that if $\theta_{s}$ is injective for all $s<r$, then $\theta_{r}$ is injective. Suppose that $\theta_{r}(\gamma + B_{q}(\rvr_{r}^{q}(X))) = \gamma + B_{q}(\Vr_{r}(X)) = 0 + B_{q}(\Vr_{r}(X))$. If we can show that $\gamma \in B_{q}(\rvr_{r}^{q}(X))$ then this would show that $\theta_{r}$ is injective. Since $\gamma + B_{q}(\Vr_{r}(X)) = 0 + B_{q}(\Vr_{r}(X))$ it follows that $\gamma \in B_{q}(\Vr_{r}(X))$. From Lemma \ref{q-the-big-lemma} we know that \\ $B_{q}(\Vr_{r}(X)) = B_{q}(\rvr_{r}^{q}(X))$ and therefore $\gamma \in B_{q}(\rvr_{r}^{q}(X))$. 

\end{proof}
Finally, we show that the diagram in Theorem \ref{deg-q-theorem-VR-RVR-isomorphism} is commutative to complete the proof of Theorem $\ref{deg-q-theorem-VR-RVR-isomorphism}$. 

\begin{lemma}
    The diagram in Theorem \ref{deg-q-theorem-VR-RVR-isomorphism} is commutative.
\end{lemma}

\begin{proof}
    Let $0 < r_1 < r_2$ since the case where $r_1 = 0$ is trivial. Let $\gamma + B_{q}(\rvr^{q}_{r_1}(X))$ be an element of $H_{q}(\rvr^{q}_{r_1}(X))$. Then $g_{r_1}^{r_2}(\theta_{r_1} (\gamma + B_{q}(\rvr^{q}_{r_1}(X))) = g_{r_1}^{r_2}(\gamma + B_{q}(\Vr_{r_1}(X))) = \gamma + B_{q}(\Vr_{r_2}(X))$ and $\theta_{r_2}(f_{r_1}^{r_2}(\gamma + B_{q}(\rvr^{q}_{r_1}(X))) = \theta_{r_2}(\gamma + B_{q}(\rvr^{q}_{r_2}(X))) = \gamma + B_{q}(\Vr_{r_2}(X))$. Thus commutativity is proven.
\end{proof}

\end{appendices}

\newpage

\bibliography{references}
\bibliographystyle{ieeetr}

\end{document}